\documentclass[11pt]{article}
\usepackage{amsthm,amsmath,amssymb,bbm}
\usepackage{color}
\def\bbe{\mathbb E}
\def\bbp{\mathbb P}
\def\bbr{\mathbb R}

\def\bbone{{\mathbbm 1}}

\theoremstyle{plain}
\newtheorem{thm}{Theorem}[section]

\newtheorem{cor}{Corollary}[section]
\newtheorem{prop}{Proposition}[section]
\newtheorem{lem}{Lemma}[section]
\theoremstyle{definition}
\newtheorem{rem}{Remark}[section]
\numberwithin{equation}{section}
\allowdisplaybreaks[1]
\makeatletter  \@addtoreset{equation}{section} \makeatother

\begin{document}
\title{Covariance Representations, $L^p$-Poincar\'e Inequalities, Stein's Kernels and High Dimensional CLTs}

\author{Benjamin Arras\thanks{Universit\'e de Lille, Laboratoire Paul Painlev\'e, CNRS U.M.R. 8524, 
59655 Villeneuve d'Ascq, France; benjamin.arras@univ-lille.fr}\; and Christian Houdr\'e\thanks{Georgia Institute of Technology, 
School of Mathematics, Atlanta, GA 30332-0160, USA; houdre@math.gatech.edu. 
Research supported in part by the grant \# $524678$ from the Simons Foundation.
 }}

\maketitle

\vspace{\fill}
\begin{abstract}
We explore connections between covariance representations, Bismut-type formulas and Stein's method.~First, using the theory of closed symmetric forms, we derive covariance representations for several well-known probability measures on $\bbr^d$,  $d \geq 1$.  When strong gradient bounds are available, these covariance representations immediately lead to $L^p$-$L^q$ covariance estimates, for all $p \in (1, +\infty)$ and $q = p/(p-1)$.~Then, we revisit the well-known $L^p$-Poincar\'e inequalities ($p \geq 2$) for the standard Gaussian probability measure on $\bbr^d$ based on a covariance representation.~Moreover, for the nondegenerate symmetric $\alpha$-stable case, $\alpha \in (1,2)$,  we obtain $L^p$-Poincar\'e and pseudo-Poincar\'e inequalities, for $p \in (1, \alpha)$, via a detailed analysis of the various Bismut-type formulas at our disposal.~Finally, using the construction of Stein's kernels by closed forms techniques,  we obtain quantitative high-dimensional CLTs in $1$-Wasserstein distance when the limiting Gaussian probability measure is anisotropic.~The dependence on the parameters is completely explicit and the rates of convergence are sharp.  
\end{abstract}
\vspace{\fill}

\section*{Introduction}

Covariance representations and Bismut-type formulas play a major role in modern probability theory.~The most striking (and simple) instances are without a doubt the ones regarding the standard Gaussian probability measure on $\bbr^d$.~These identities have many applications ranging from functional inequalities, concentration phenomena, regularization along semigroup, continuity of certain singular integral operators and Stein's method.~The main objective of the present manuscript is to illustrate this circle of ideas.~While some of the results presented here might be well-known to specialists, others seem to be new.~Let us further describe the main content of these notes.~In the first section, we revisit covariance identities based on closed form techniques and semigroup arguments. In particular, when strong gradient bounds are available, $L^p$-$L^q$ asymmetric covariance estimates ($p \in [1,+\infty)$ and $q = p/(p-1)$) are put forward.~In the second section, based on various representation formulas, we discuss $L^p$-Poincar\'e inequalities ($p \geq 2$) and pseudo-Poincar\'e inequality for the standard Gaussian measure and for the nondegenerate symmetric $\alpha$-stable probability measures on $\bbr^d$ with $\alpha \in (1,2)$.~Finally, in the third section, as an application of our methodology, we build Stein's kernels in order to obtain, in $1$-Wasserstein distance, rates of convergence for high-dimensional central limit theorems when the limiting probability measure is a centered Gaussian measure with nondegenerate covariance matrix.~The methodology is based on Stein's method for multivariate Gaussian probability measures and on closed forms techniques under a finite Poincar\'e-type constant assumption. 

\section{Notations and Preliminaries}
\noindent
Throughout, the Euclidean norm on $\bbr^d$ is denoted by $\|\cdot\|$ and the Euclidean inner product by $\langle ;\rangle$.    Then, $X \sim ID(b, \Sigma, \nu)$  indicates that the 
$d$-dimensional random vector $X$ is infinitely divisible with characteristic triplet $(b, \Sigma, \nu)$.  In other words, its characteristic function $\varphi_X$ is given, for all $\xi \in \bbr^d$, by 
$$
\varphi_X(\xi) = \exp\left(i\langle b; \xi\rangle -\frac{1}{2}\langle\Sigma \xi;\xi\rangle + \int_{\bbr^d} (e^{i\langle \xi;u\rangle} - 1 - i\langle \xi; u\rangle\bbone_{\|u\|\leq 1}) \nu(du)\right),
$$
where $b\in \bbr^d$, where $\Sigma$ is a symmetric positive semi-definite 
$d\times d$ matrix, and where $\nu$, the L\'evy measure, is a positive Borel measure on $\bbr^d$ such that $\nu(\{0\}) = 0$ and such that $\int_{\bbr^d} (1 \wedge \|u\|^2) \nu(du) <+\infty$.  
In particular, if $b = 0$, $\Sigma = I_d$, the $d\times d$ identity matrix, and $\nu = 0$, then 
$X$ is a standard Gaussian random vector with law $\gamma$ and 
its characteristic function is given, for all $\xi \in \bbr^d$, by
\begin{equation}\label{def:gauss}
\hat{\gamma}(\xi) := \int_{\bbr^d} e^{i \langle y ; \xi\rangle} \gamma(dy) =  \exp\left(- \frac{\|\xi\|^2}{2}\right).
\end{equation}
For $\alpha \in (0,2)$, let $\nu_\alpha$ be a L\'evy measure such that, for all $c>0$,
\begin{align}\label{eq:scale}
c^{-\alpha}T_c(\nu_\alpha)(du)=\nu_\alpha(du),
\end{align}
where $T_c(\nu_\alpha)(B):=\nu_{\alpha}(B/c)$, for all $B$ Borel set of $\bbr^d\setminus \{0\}$.  Recall that such a  L\'evy measure admits the polar decomposition
\begin{align}\label{eq:polar}
\nu_\alpha(du) = \bbone_{(0,+\infty)}(r) \bbone_{\mathbb{S}^{d-1}}(y)\dfrac{dr}{r^{\alpha+1}}\sigma(dy),
\end{align}
where $\sigma$ is a positive finite measure on the Euclidean unit sphere of $\bbr^d$ denoted by $\mathbb{S}^{d-1}$.~In the sequel, it is assumed that the measure $\sigma$ is symmetric and that $\nu_\alpha$ is nondegenerate in that 
\begin{align}\label{eq:non_deg}
\underset{y \in \mathbb{S}^{d-1}}{\inf} \int_{\mathbb{S}^{d-1}} |\langle y;x \rangle|^{\alpha} \lambda_1(dx) \ne 0,
\end{align}
where $\lambda_1$, the spectral measure, is a symmetric finite positive measure on $\mathbb{S}^{d-1}$ proportional to $\sigma$ (namely, $\lambda_1(dx)= - \cos (\alpha \pi /2) \Gamma(2-\alpha)/(\alpha (\alpha-1))\sigma(dx)$, $\alpha \in (1,2)$ and where $\Gamma$ is the Euler Gamma function).~Let $\mu_\alpha$ be the $\alpha$-stable probability measure on $\bbr^d$ defined through the corresponding characteristic function, for all $\xi \in \bbr^d$, by
\begin{align}\label{stable:characteristic}
\varphi_\alpha(\xi):= \int_{\bbr^d} e^{i \langle y ; \xi\rangle} \mu_\alpha(dy) = \left\{
    \begin{array}{ll}
        \exp\left(\int_{\bbr^d} (e^{i \langle u;\xi \rangle}-1-i\langle \xi;u\rangle) \nu_\alpha(du)\right),  & \alpha \in (1,2), \\
        \exp\left(\int_{\bbr^d}\left(e^{i \langle u;\xi \rangle} - 1 -i \langle \xi;u\rangle \bbone_{|u|\leq 1}\right) \nu_1(du) \right), & \alpha=1, \\
	\exp\left(\int_{\bbr^d} (e^{i \langle u;\xi \rangle}-1) \nu_\alpha(du)\right), & \alpha \in (0,1).
    \end{array}
\right.
\end{align}
For $\sigma$ symmetric, \cite[Theorem 14.13.]{S} provides a useful alternative representation for the characteristic function $\varphi_\alpha$ given, for all $\xi \in \bbr^d$, by
\begin{align}\label{eq:rep_spectral_measure}
\varphi_\alpha(\xi) = \exp\left(- \int_{\mathbb{S}^{d-1}} |\langle y;\xi \rangle|^\alpha \lambda_1(dy)\right).
\end{align}
Let $\lambda$ denote a uniform measure on the Euclidean unit sphere of $\bbr^d$.~For $\alpha \in (1,2)$, let $\nu_{\alpha}^{\operatorname{rot}}$ be the L\'evy measure on $\bbr^d$ with polar decomposition 
\begin{align}\label{eq:Levy_Rot}
\nu_{\alpha}^{\operatorname{rot}}(du) = c_{\alpha,d} \bbone_{(0,+\infty)}(r) \bbone_{\mathbb{S}^{d-1}}(y)\dfrac{dr}{r^{\alpha+1}}\lambda(dy),
\end{align}
and with,
\begin{align}\label{eq:renorm}
c_{\alpha,d} = \dfrac{-\alpha (\alpha-1) \Gamma\left(\frac{\alpha+d}{2}\right)}{4 \cos \left(\frac{\alpha\pi}{2}\right)\Gamma\left(\frac{\alpha+1}{2}\right) \pi^{\frac{d-1}{2}}\Gamma(2-\alpha)}.
\end{align}
Finally, denote by $\mu_{\alpha}^{\operatorname{rot}}$ the rotationally invariant $\alpha$-stable probability measure on $\bbr^d$ with L\'evy measure given by \eqref{eq:Levy_Rot} and with the choice of $\lambda$ ensuring that, for all $\xi \in \bbr^d$,
\begin{align}\label{eq:charac_rot}
\varphi^{\operatorname{rot}}_\alpha(\xi) = \hat{\mu}_{\alpha}^{\operatorname{rot}}\left(\xi\right) = \exp\left(-\frac{\|\xi\|^\alpha}{2}\right).
\end{align}
As well-known the probability measure $\mu_{\alpha}^{\operatorname{rot}}$ is absolutely continuous with respect to the $d$-dimensional Lebesgue measure and its Lebesgue density, denoted by $p_{\alpha}^{\operatorname{rot}}$, is infinitely differentiable and is such that, for all $x\in\bbr^d$, 
\begin{align*}
\frac{C_2}{\left(1+\|x\|\right)^{\alpha+d}} \leq p_{\alpha}^{\operatorname{rot}}(x) \leq \dfrac{C_1}{\left(1+\|x\|\right)^{\alpha+d}},
\end{align*}
for some constants $C_1,C_2>0$ depending only on $\alpha$ and on $d$. For $\alpha \in (1,2)$, let $\nu_{\alpha,1}$ be the L\'evy measure on $\bbr$ given by
\begin{align}\label{sym:1}
\nu_{\alpha,1}(du) = c_{\alpha}\frac{du}{|u|^{\alpha+1}},
\end{align}
with,
\begin{align}
c_{\alpha} = \left( \dfrac{-\alpha (\alpha -1)}{4 \Gamma(2-\alpha)\cos\left(\frac{\alpha \pi}{2}\right)} \right).
\end{align}
Next, let $\mu_{\alpha,1}$ be the $\alpha$-stable probability measure on $\bbr$ with L\'evy measure $\nu_{\alpha,1}$ and with corresponding 
characteristic function defined, for all $\xi \in \bbr$, by
\begin{align}
\hat{\mu}_{\alpha,1}(\xi) = \exp \left(\int_{\bbr} \left(e^{i \langle u;\xi\rangle}-1-i \langle u;\xi \rangle\right) \nu_{\alpha,1}(du)\right) = \exp\left(-\frac{|\xi|^{\alpha}}{2}\right).
\end{align}
Finally, throughout let $\mu_{\alpha,d} = \mu_{\alpha,1} \otimes \dots \otimes \mu_{\alpha,1}$ be the product probability measure on $\bbr^d$ with corresponding 
characteristic function given, for all $\xi \in \bbr^d$, by
\begin{align}\label{eq:StableIndAxes}
\hat{\mu}_{\alpha,d}(\xi) = \prod_{k=1}^d \hat{\mu}_{\alpha,1}(\xi_k) = \exp\left(\int_{\bbr^d} \left(e^{i \langle \xi;u \rangle}-1-i\langle \xi;u \rangle\right)\nu_{\alpha,d}(du)\right),
\end{align}
and with,
\begin{align}\label{eq:LevyIndAxes}
\nu_{\alpha,d}(du) = \sum_{k=1}^d \delta_0(du_1) \otimes \dots \otimes \delta_0(du_{k-1}) \otimes \nu_{\alpha,1}(du_k) \otimes \delta_0(du_{k+1}) \otimes \dots \otimes \delta_0(du_d),
\end{align}
where $\delta_0$ is the Dirac measure at $0$.

In the sequel, $\mathcal{S}(\bbr^d)$ is the Schwartz space of infinitely differentiable functions which, with their derivatives of any order, are rapidly decreasing, and $\mathcal{F}$ is the Fourier transform operator given, for all $f \in \mathcal{S}(\bbr^d)$ and all $\xi \in \bbr^d$, by
\begin{align*}
\mathcal{F}(f)(\xi) = \int_{\bbr^d} f(x) e^{- i \langle x; \xi \rangle} dx.
\end{align*}
On $\mathcal{S}(\bbr^d)$, the Fourier transform is an isomorphism and the following well-known inversion formula holds
\begin{align*}
f(x) = \frac{1}{(2\pi)^d} \int_{\bbr^d} \mathcal{F}(f)(\xi)e^{i \langle \xi ; x\rangle} d\xi,\quad x\in \bbr^d.
\end{align*}
$\mathcal{C}_c^\infty(\bbr^d)$ is the space of infinitely differentiable functions on $\bbr^d$ with compact support and $\|\cdot\|_{\infty,\bbr}$ denotes the supremum norm on $\bbr$.  Let $\mathcal{C}_b(\bbr^d)$ be the space of bounded continuous functions on $\bbr^d$ and let $\mathcal{C}^1_b(\bbr^d)$ be the space of continuously differentiable functions which are bounded on $\bbr^d$ together with their first derivatives.~For $p \in (1,+\infty)$, $L^p(\mu_\alpha)$ denotes the space of equivalence classes (with respect to $\mu_\alpha$-almost everywhere equality) of functions which are Borel measurable and which are $p$-summable with respect to the probability measure $\mu_\alpha$.~This space is endowed with the usual norm $\|\cdot\|_{L^p(\mu_\alpha)}$ defined, for all suitable $f$, by
\begin{align*}
\|f\|_{L^p(\mu_\alpha)} := \left(\int_{\bbr^d} |f(x)|^p \mu_\alpha(dx)\right)^{\frac{1}{p}}.
\end{align*}
Similarly, for $p \in (1, +\infty)$, $L^p(\bbr^d,dx)$ denotes the classical Lebesgue space where the reference measure is the Lebesgue measure. It is endowed with the norm $\|\cdot\|_{L^p(\bbr^d,dx)}$ defined, for all suitable $f$, by
\begin{align*}
\|f\|_{L^p(\bbr^d,dx)} := \left(\int_{\bbr^d} |f(x)|^p dx\right)^{\frac{1}{p}}.
\end{align*}
Next, let us introduce two semigroups of operators acting on $\mathcal{S}(\bbr^d)$ naturally associated with $\gamma$ and $\mu_\alpha$. Let $(P^{\gamma}_t)_{t\geq 0}$ and $(P^{\nu_\alpha}_t)_{t\geq 0}$ be defined, for all $f\in \mathcal{S}(\bbr^d)$, all $x\in \bbr^d$ and all $t \geq 0$, by
\begin{align}\label{eq:OUSM}
P^{\gamma}_t(f)(x) = \int_{\mathbb{R}^d} f(xe^{-t}+\sqrt{1-e^{-2t}} y) \gamma(dy),
\end{align}
\begin{align}\label{eq:StOUSM}
P^{\nu_\alpha}_t(f)(x) = \int_{\mathbb{R}^d} f(xe^{-t}+(1-e^{-\alpha t})^{\frac{1}{\alpha}} y)\mu_\alpha(dy).
\end{align}
The semigroup \eqref{eq:OUSM} is the classical Gaussian Ornstein-Uhlenbeck semigroup and the semigroup \eqref{eq:StOUSM} is the Ornstein-Uhlenbeck semigroup associated with the $\alpha$-stable probability measure $\mu_\alpha$ and recently put forward in the context of Stein's method for self-decomposable distributions (see \cite{AH18_1,AH19_2,AH20_3}). Finally, denoting by $((P^{\nu_\alpha}_t)^*)_{t\geq 0}$ the formal adjoint of the semigroup $(P^{\nu_\alpha}_t)_{t\geq 0}$, the ``carr\'e de Mehler" semigroup is defined, for all $t \geq 0$, by
\begin{align}\label{def:squared_Mehler}
\mathcal{P}_t = (P^{\nu_\alpha}_{\frac{t}{\alpha}})^* \circ P^{\nu_\alpha}_{\frac{t}{\alpha}}=P^{\nu_\alpha}_{\frac{t}{\alpha}} \circ (P^{\nu_\alpha}_{\frac{t}{\alpha}})^*.
\end{align}
In the sequel, $\partial_k$ denotes the partial derivative of order $1$ in the variable $x_k$,  $\nabla$ the gradient operator, $\Delta$ the Laplacean operator and $D^{\alpha-1}$, $(D^{\alpha-1})^*$ and $\mathbf{D}^{\alpha-1}$ the fractional operators defined, for all $f \in \mathcal{S}(\bbr^d)$ and all $x\in \bbr^d$, by
\begin{align}\label{eq:FracGrad}
D^{\alpha-1}(f)(x) :=\int_{\bbr^d} (f(x+u)-f(x)) u \nu_\alpha(du),
\end{align}
\begin{align}\label{eq:FracGradDual}
(D^{\alpha-1})^*(f)(x) := \int_{\bbr^d} (f(x-u)-f(x)) u \nu_\alpha(du),
\end{align}
\begin{align}\label{eq:FracGradWB}
\mathbf{D}^{\alpha-1}(f)(x) :=\frac{1}{2}\left(D^{\alpha-1}\left(f\right)(x)-(D^{\alpha-1})^*\left(f\right)(x)\right) .
\end{align}
Let us introduce also a gradient-length, $\nabla_{\nu}$, which is linked to the energy form appearing in the Poincar\'e-type inequality for the infinitely divisible probability measures (see \cite[Corollary $2$]{HPAS}).  For any L\'evy measure $\nu$ on $\bbr^d$, all $f \in \mathcal{S}(\bbr^d)$ and all $x \in \bbr^d$, 
\begin{align}\label{eq:gradient_length_ID}
\nabla_{\nu}(f)(x) = \left(\int_{\bbr^d}|f(x+u) - f(x)|^2 \nu(du)\right)^{\frac{1}{2}}. 
\end{align}
Also, let us recall the definition of the gamma transform of order $r>0$.~For $(P_t)_{t\ge 0}$ a $C_0$-semigroup of contractions on a Banach space, with generator $\mathcal{A}$, the gamma transform of order $r>0$ is defined, for all suitable $f$, by
\begin{align}\label{eq:Gam_Tr}
\left(E-\mathcal{A} \right)^{-\frac{r}{2}} f = \frac{1}{\Gamma(\frac{r}{2})} \int_0^{+\infty} \dfrac{e^{-t}}{t^{1-\frac{r}{2}}} P_t(f) dt,
\end{align}
where $E$ is the identity operator and where the integral on the right-hand side has to be understood in the Bochner sense.~Moreover, for all $\lambda >0$, all $r>0$ and all $f$ suitable, 
\begin{align*}
\left(\lambda E-\mathcal{A} \right)^{-\frac{r}{2}} f = \frac{1}{\Gamma(\frac{r}{2})} \int_0^{+\infty} \dfrac{e^{-\lambda t}}{t^{1-\frac{r}{2}}} P_t(f) dt,
\end{align*}
In particular, when this makes sense, as $\lambda$ tends to $0^+$, 
\begin{align*}
\left(-\mathcal{A} \right)^{-\frac{r}{2}} f = \frac{1}{\Gamma(\frac{r}{2})} \int_0^{+\infty} t^{\frac{r}{2}-1} P_t(f) dt.
\end{align*}
Finally, the generators of the two semigroups can be obtained through the Fourier representation formulas and it is straightforward to check that the respective generators are given, for $\alpha \in (1,2)$, for all $f\in \mathcal{S}(\bbr^d)$ and all $x\in \bbr^d$, by
\begin{align}\label{eq:OUgen}
\mathcal{L}^\gamma(f)(x)=-\langle x;\nabla(f)(x)\rangle + \Delta(f)(x),
\end{align}
\begin{align}\label{eq:StOUgen}
\mathcal{L}^\alpha(f)(x)= -\langle x;\nabla(f)(x) \rangle + \int_{\bbr^d} \langle \nabla(f)(x+u)-\nabla(f)(x);u \rangle \nu_\alpha(du). 
\end{align}
\noindent
Recall also one of the main results of \cite{AH20_4}, giving a Bismut-type formula for the nondegenerate symmetric $\alpha$-stable probability measures on $\bbr^d$ with $\alpha \in (1,2)$: for all $f \in \mathcal{S}(\bbr^d)$, all $x \in \bbr^d$ and all $t>0$, 
\begin{align}\label{eq:Bismut_Stable}
D^{\alpha-1} P_t^{\nu_\alpha}(f)(x) = \dfrac{e^{- (\alpha-1)t}}{\left(1- e^{- \alpha t}\right)^{1- \frac{1}{\alpha}}} \int_{\bbr^d} y  f\left(xe^{-t} + (1- e^{- \alpha t})^{\frac{1}{\alpha}}y\right) \mu_\alpha(dy). 
\end{align}
Finally, recall the covariance representation obtained in \cite[Proposition 2]{HPAS} for $X \sim ID(b,\Sigma,\nu)$: for all $f,g \in \mathcal{S}(\bbr^d)$,
\begin{align}\label{eq:covariance_representation_gaussian}
\operatorname{Cov}(f(X),g(X))  = \int_0^1 \bbe \left[ \langle \Sigma \nabla(f)(X_z)  ; \nabla(g)(Y_z) \rangle + \int_{\bbr^d} \Delta_u(f)(X_z)\Delta_u(g)(Y_z) \nu(du) \right] dz, 
\end{align}
where $\Delta_u(f)(x) = f(x+u)-f(x)$ and where, for all $z \in [0,1]$, $(X_z,Y_z)$ has characteristic function given, for all $\xi_1, \xi_2 \in \bbr^d$, by
\begin{align*}
\varphi(\xi_1, \xi_2) = (\varphi_X(\xi_1) \varphi_X(\xi_2))^{1-z} \varphi_X(\xi_1+\xi_2)^{z}.
\end{align*}
Next, let us investigate new covariance identities based on semigroup techniques (inspired from \cite[Section $5$]{AH18_1}) and the corresponding asymmetric covariance estimates (see, e.g., \cite{CCL13} for log-concave measures and \cite{HN19} for convex measures which include Cauchy-type probability measures on $\bbr^d$). But, first, as a simple consequence of the covariance identity obtained in \cite[Proposition $2$]{HPAS} combined with H\"older's inequality, one has the following proposition for the Gaussian and for the general infinitely divisible cases (we refer to \cite{HPAS} and to \cite{S} for any unexplained definitions regarding infinitely divisible distributions).

\begin{prop}\label{prop:BL_type_ID}
(i) Let $X\sim \gamma$.~Then, for all $f,g\in \mathcal{S}(\bbr^d)$, all $p \in [1, +\infty)$ and $q=p/(p-1)$ (with $q=+\infty$ when $p=1$), 
\begin{align}\label{ineq:BL_Gaussian}
|\operatorname{Cov}(f(X),g(X))| \leq \|\nabla(f)\|_{L^p(\gamma)} \|\nabla(g)\|_{L^q(\gamma)}. 
\end{align}
(ii) Let $X \sim ID(b,0, \nu)$.~Then, for all $f,g\in \mathcal{S}(\bbr^d)$, all $p \in [1,+\infty)$ and $q=p/(p-1)$ (with $q=+\infty$ when $p=1$), 
\begin{align}\label{ineq:BL_ID}
|\operatorname{Cov}(f(X),g(X))| \leq \|\nabla_\nu(f)\|_{L^p(\mu)} \|\nabla_\nu(g)\|_{L^q(\mu)},
\end{align}
where $X \sim \mu$.  
\end{prop}
\noindent
In the context of Stein's method for self-decomposable laws, a general covariance identity has been obtained in \cite[Theorem $5.10$]{AH20_3} in the framework of closed symmetric non-negative definite bilinear forms with dense domain under some coercive assumption.~Indeed, the identity $(5.15)$ there can be understood as a generalization of $(2.5)$ in \cite{CCL13} and of $(3.2)$ in \cite{HN19} from which asymmetric covariance estimates can be obtained. Let us generalize Proposition \ref{prop:BL_type_ID} beyond the scope of infinitely divisible distributions.~Key properties in order to establish these Brascamp-Lieb-type inequalities are sub-commutation and/or some form of regularization (see, e.g., \cite{CCL13,ABJ20}).~Actually, let us provide, first, a slight extension of \cite[Theorem $5.10$]{AH20_3}.    

\begin{thm}\label{thm:cov_rep_dirichlet_form}
Let $H$ be a real Hilbert space with inner product $\langle \cdot ;\cdot \rangle_H$ and induced norm $\| \cdot \|_H$. Let $\mathcal{E}$ be a closed symmetric non-negative definite bilinear form with dense linear domain $\mathcal{D}(\mathcal{E})$.~Let $\{G_\alpha:\, \alpha>0\}$ and $\{P_t:\, t>0\}$ be, respectively, the strongly continuous resolvent and the strongly continuous semigroup on $H$ associated with $\mathcal{E}$. Moreover, let there exist a closed linear subspace $H_0 \subset H$ and a function $\psi$ continuous on $(0,+\infty)$ with values in $(0,1]$ such that $\underset{t \rightarrow +\infty}{\lim} \psi(t)=0$,
\begin{align*}
\int_0^{+\infty}\psi(t)dt <+\infty,
\end{align*}
and such that, for all $t>0$ and all $u \in H_0$,
\begin{align}\label{ineq:Poinc}
\|P_t(u)\|_H \leq \psi(t) \|u\|_H.
\end{align}
Let $G_{0^+}$ be the operator defined by
\begin{align}\label{eq:Gzero}
G_{0^+}(u):=\int_{0}^{+\infty} P_t(u) dt,\quad u \in H_0,
\end{align}
where the above integral is understood to be in the Bochner sense. Then, for all $u\in H_0$, $G_{0^+}(u)$ belongs to $\mathcal{D}(\mathcal{E})$ and, for all $v\in \mathcal{D}(\mathcal{E})$,
\begin{align}\label{eq:SteinKernel}
\mathcal{E} \left(G_{0^+}(u),v\right)=\langle u ; v \rangle_H .
\end{align}
Moreover, for all $u \in H_0$,
\begin{align}\label{ineq:energy}
\mathcal{E}\left(G_{0^+}(u),G_{0^+}(u)\right) \leq \left(\int_0^{+\infty} \psi(t)dt \right)\|u\|^2_H. 
\end{align}
\end{thm}

\begin{proof}
The proof follows closely the lines of the one of \cite[Theorem $5.10$]{AH20_3} and so is omitted.
\end{proof}
\noindent
\begin{rem}\label{rem:weak_strong}
(i) Note that, from the semigroup property, for all $n \geq 1$, all $t>0$ and all $f \in H_0$, 
\begin{align*}
\|P_t(f)\|_H \leq \left(\psi\left(\frac{t}{n}\right)\right)^n \|f\|_H. 
\end{align*}
Then, the behavior at $0^+$ of the function $\psi$ can lead to an exponential convergence result. Indeed, assuming that the function $\psi$ is regular near $0$ with $\psi(0^+)=1$ and with $\psi'(0^+)<0$, one gets, for all $t>0$,
\begin{align*}
\underset{n \rightarrow +\infty}{\lim} \left(\psi\left(\frac{t}{n}\right)\right)^n = \exp\left(\psi'(0)t\right).
\end{align*}
(ii) Let us next consider a rather long list of examples where our results provide covariance identities and $L^2$-estimates. In some situations, where strong gradient bounds are known, it is possible to obtain $L^p-L^q$ asymmetric covariance estimates. First, very classically, the Dirichlet form associated with the standard Gaussian probability measure on $\bbr^d$ is given, for all (real-valued) $f,g \in \mathcal{S}(\bbr^d)$, by
\begin{align}\label{eq:dirichlet_form_gamma}
\mathcal{E}_{\gamma}(f,g) = \int_{\bbr^d} \langle \nabla(f)(x) ; \nabla(g)(x) \rangle \gamma(dx),   
\end{align}
and an integration by parts formula ensures that $\mathcal{E}_{\gamma}$ is closable. The associated semigroup is the Ornstein-Uhlenbeck semigroup $(P_t^{\gamma})_{t \geq 0}$ given in \eqref{eq:OUSM} with generator given by \eqref{eq:OUgen}.~It is well-known, thanks to the Gaussian Poincar\'e inequality, that, for all $f\in \mathcal{S}(\bbr^d)$ with $\gamma(f) := \int_{\bbr^d} f(x) \gamma(dx) = 0$ and all $t>0$, 
\begin{align*}
\|P_t^\gamma(f)\|_{L^2(\gamma)} \leq e^{-t} \|f\|_{L^2(\gamma)}. 
\end{align*}
Thus, Theorem \ref{thm:cov_rep_dirichlet_form} provides the following covariance representation: for all $f,g \in \mathcal{S}(\bbr^d) $ with $\int_{\bbr^d} f(x) \gamma(dx) = 0$,
\begin{align}\label{eq:cov_rep_gauss_SG}
\int_{\mathbb{R}^d} f(x)g(x) \gamma(dx) = \int_{\bbr^d} \langle \nabla(\tilde{f}_\gamma)(x) ; \nabla(g)(x) \rangle \gamma(dx) = \mathcal{E}_\gamma(g,\tilde{f}_\gamma),
\end{align}
where
\begin{align*}
\tilde{f}_\gamma = \int_{0}^{+\infty}P_t^\gamma(f)dt =(-\mathcal{L}^\gamma)^{-1} (f). 
\end{align*}
A straightforward application of \cite[Theorem $5.10$]{AH20_3} ensures, for all $f \in \mathcal{S}(\bbr^d)$ with $\int_{\bbr^d}f(x) \gamma(dx) = 0$, that
\begin{align*}
\mathcal{E}_\gamma(\tilde{f}_\gamma, \tilde{f}_\gamma) \leq \|f\|^{2}_{L^2(\gamma)} \leq \mathcal{E}_\gamma(f,f),
\end{align*}
so that, by the Cauchy-Schwarz inequality,
\begin{align*}
\left|\int_{\bbr^d}f(x)g(x) \gamma(dx) \right| \leq \mathcal{E}_\gamma(g,g)^{\frac{1}{2}} \mathcal{E}_\gamma(f,f)^{\frac{1}{2}}, 
\end{align*}
which is a particular instance of Proposition \ref{prop:BL_type_ID} (i), with $p=q=2$. To obtain the general $L^p-L^q$ covariance estimates based on the covariance representation \eqref{eq:cov_rep_gauss_SG}, one can use the commutation formula $\nabla(P_t^\gamma(f)) = e^{-t} P_t^\gamma(\nabla(f))$ so that the following inequality holds true: for all $p \in (1, +\infty)$ and all $f \in \mathcal{S}(\bbr^d)$ with $\int_{\bbr^d} f(x) \gamma(dx) = 
0$, 
\begin{align}\label{ineq:Lp_est}
\|\nabla(\tilde{f}_\gamma)\|_{L^p(\gamma)} \leq \|\nabla(f)\|_{L^p(\gamma)}.  
\end{align}
A direct application of H\"older's inequality combined with \eqref{eq:cov_rep_gauss_SG} and with \eqref{ineq:Lp_est} provides the general case of Proposition \ref{prop:BL_type_ID}~(i).~Note that the previous lines of reasoning do not depend on the dimension of the ambient space so that the covariance estimate  \eqref{ineq:BL_Gaussian} extends to the infinite dimensional setting and to the Malliavin calculus framework (see, e.g., \cite[Section 2.9]{NP12}). The details are left to the interested reader (see, also \cite{HPA95, HP02}).\\
(iii) Next, let us consider another hypercontractive semigroup related to a classical probability measure on $\bbr$ with finite exponential moments. (The corresponding multi-dimensional version follows by tensorization.) Let $\alpha \geq 1/2$ and let $\gamma_{\alpha, 1}$ be the gamma probability measure on $(0,+\infty)$ with shape parameter $\alpha$ and scale parameter $1$. Let $\mathcal{E}_\alpha$ denote the Dirichlet form associated with the Laguerre dynamics and given, for all $f,g \in \mathcal{C}_c^{\infty}((0, +\infty))$, by
\begin{align*}
\mathcal{E}_\alpha(f,g) = \int_{0}^{+\infty} x f'(x)g'(x) \gamma_{\alpha,1}(dx). 
\end{align*}
This closable symmetric bilinear form generates the well-known Laguerre semigroup $(P_t^{\alpha,1})_{t\geq 0}$ (see, e.g., \cite{BGL14, AS_17}) with generator given, for all $f \in \mathcal{C}^{\infty}_c((0,+\infty))$ and all $x \in (0, +\infty)$, by
\begin{align*}
\mathcal{L}^{\alpha,1} (f)(x) = x f''(x) + (\alpha -x) f'(x).
\end{align*}
Recall that the Poincar\'e inequality for this dynamics follows, e.g., from the spectral expansion of a test function belonging to $\mathcal{D}(\mathcal{E}_\alpha)$ along the Laguerre orthonormal polynomials which are the eigenfunctions of $\mathcal{L}^{\alpha,1}$.~Moreover, letting $\partial_\sigma$ be the differential operator defined, for all $f \in \mathcal{C}^{\infty}_c((0,+\infty))$ and all $x \in (0, +\infty)$, by
\begin{align*}
\partial_\sigma(f)(x) : = \sqrt{x} f'(x), 
\end{align*}
the following intertwining formula has been proved in \cite[Lemma $11$]{AS_17}: for all $f \in \mathcal{C}^{\infty}_c((0,+\infty))$, all $x \in (0, +\infty)$ and all $t>0$,
\begin{align}\label{eq:intert_laguerre}
\partial_\sigma(P_t^{\alpha,1}(f))(x) = e^{-\frac{t}{2}} \bbe \left( \frac{\left(e^{-\frac{t}{2}}\sqrt{x} + \sqrt{\frac{1-e^{-t}}{2}} Z\right)}{(X_t^x)^{\frac{1}{2}}} \partial_\sigma(f)(X^x_t) \right),
\end{align}
where $Z$ is a standard normal random variable and where $X_t^x$ is given by, 
\begin{align*}
X_t^x = (1-e^{-t} ) X_{\alpha-\frac{1}{2},1} + \left( e^{-\frac{t}{2}} \sqrt{x} + \sqrt{\frac{1-e^{-t}}{2}} Z \right)^2,
\end{align*}
with $X_{\alpha-1/2,1} \sim \gamma_{\alpha-1/2,1}$ independent of $Z$.~Using \eqref{eq:intert_laguerre}, the following sub-commutation inequality holds true: for all $f \in \mathcal{C}_c^{\infty}((0,+\infty))$, all $x >0$, and all $t>0$
\begin{align}\label{ineq:sub_comm}
|\partial_\sigma(P_t^{\alpha,1}(f))(x)| \leq e^{-\frac{t}{2}} P_t^{\alpha,1}(|\partial_\sigma(f)|)(x).
\end{align}
Performing a reasoning similar to the one in the Gaussian case, one gets the following asymmetric covariance estimate: for all $f,g \in \mathcal{C}^{\infty}_{c}((0, +\infty))$, 
all $p \in (1, +\infty)$ and $q = p/(p-1)$, 
\begin{align*}
\left|\operatorname{Cov}(f(X_{\alpha,1}),g(X_{\alpha,1}))\right| \leq 2 \|\partial_\sigma(g)\|_{L^q(\gamma_{\alpha,1})} \|\partial_{\sigma}(f)\|_{L^p(\gamma_{\alpha,1})}. 
\end{align*}
The previous sub-commutation inequality can be seen as a direct consequence of the Bakry-Emery criterion since for this Markov diffusion semigroup $\Gamma_2(f) \geq  \Gamma(f)/2$, for all $f \in  \mathcal{C}_c^{\infty}((0,+\infty))$.\\
(iv) Next, let us consider the Jacobi semigroup $(Q_t^{\alpha,\beta})_{t\ge 0}$, related to the beta probability measures on $[-1,1]$ of the form
\begin{align*}
\mu_{\alpha,\beta}(dx) =C_{\alpha,\beta} (1-x)^{\alpha-1} (1+x)^{\beta-1} \bbone_{[-1,1]}(x)dx,
\end{align*}
with $\alpha>0,\beta>0$ such that $\min (\alpha,\beta)>3/2$ and where $C_{\alpha,\beta}>0$ is a normalization constant.~The generator of the Jacobi semigroup is given, for all $f \in  \mathcal{C}_c^{\infty}([-1,1])$ and all $x \in [-1,1]$, by
\begin{align*} 
\mathcal{L}_{\alpha,\beta}(f)(x) = (1-x^2)f''(x)+((\beta-\alpha)-(\alpha+\beta)x)f'(x).
\end{align*}
Moreover, the corresponding ``carr\'e du champs" operator is $\Gamma_{\alpha,\beta}(f)(x) = (1-x^2)(f'(x))^2$, for all $f \in \mathcal{C}_c^{\infty}([-1,1])$ and all $x \in [-1,1]$, so that the natural gradient associated with the Jacobi operator is given, for all $f \in \mathcal{C}_c^\infty([-1,1])$ and all $x \in [-1,1]$, by
\begin{align*}
\partial_{\alpha,\beta}(f)(x) : = \sqrt{1-x^2}f'(x). 
\end{align*} 
According to \cite[Section $2.1.7$]{BGL14}, this Markov diffusion operator satisfies a curvature-dimension condition of the type $CD(\kappa_{\alpha,\beta},n_{\alpha,\beta})$ for some $\kappa_{\alpha,\beta},n_{\alpha,\beta}>0$ depending only on $\alpha$ and on $\beta$. In particular, it satisfies a curvature dimension condition $CD(\kappa_{\alpha,\beta},\infty)$ and so one has the following sub-commutation formula: for all $f \in \mathcal{C}_c^{\infty}([-1,1])$, all $t>0$ and all $x \in [-1,1]$, 
\begin{align*}
\left|\partial_{\alpha,\beta}\left(Q^{\alpha,\beta}_t(f)\right)(x)\right| \leq e^{-\kappa_{\alpha,\beta}t} Q^{\alpha,\beta}_t\left(\left|\partial_{\alpha,\beta}(f)\right|\right)(x). 
\end{align*}
The covariance representation then reads as: for all $f , g \in \mathcal{C}_c^{\infty}([-1,1])$ with $\mu_{\alpha,\beta}(f) = 0$, 
\begin{align*}
\operatorname{Cov}(f(X_{\alpha,\beta}),g(X_{\alpha,\beta})) = \int_{[-1,1]} (1-x^2) g'(x) \tilde{f}'(x) \mu_{\alpha,\beta}(dx),  \quad \tilde{f} = \int_{0}^{+\infty} Q_t^{\alpha,\beta}(f)dt,
\end{align*} 
where $X_{\alpha,\beta} \sim \mu_{\alpha, \beta}$. Applying the same strategy as before, gives the following asymmetric covariance estimate: for all $f,g \in \mathcal{C}_c^{\infty}([-1,1])$, all $p \in (1, + \infty)$ and $q=p/(p-1)$, 
\begin{align*}
\left|\operatorname{Cov}(f(X_{\alpha,\beta}),g(X_{\alpha,\beta}))\right| \leq \frac{1}{\kappa_{\alpha,\beta}} \|\partial_{\alpha,\beta}(f)\|_{L^p(\mu_{\alpha,\beta})} \|\partial_{\alpha,\beta}(g)\|_{L^q(\mu_{\alpha,\beta})}. 
\end{align*}
(v) Let $\mu$ be a centered probability measure on $\bbr^d$ given by 
\begin{align*}
\mu(dx) = \frac{1}{Z(\mu)} \exp\left(-V(x)\right)dx,
\end{align*}
where $Z(\mu)>0$ is a normalization constant and where $V$ is a non-negative smooth function on $\bbr^d$ such that,
\begin{align*}
\operatorname{Hess}(V)(x) \geq \kappa I_d,\quad x\in \bbr^d,
\end{align*}
for some $\kappa>0$ and where $I_d$ is the identity matrix.~It is well-known that such a probability measure satisfies a Poincar\'e inequality with respect to the classical ``carr\'e du champs" (see, e.g., \cite[Theorem $4.6.3$]{BGL14}), namely, for all $f$ smooth enough on $\bbr^d$, 
\begin{align*}
\operatorname{Var}_\mu(f) \leq C_\kappa \int_{\bbr^d} \|\nabla(f)(x)\|^2 \mu(dx), 
\end{align*}
for some $C_\kappa>0$ depending on $\kappa$ and on $d\geq 1$ (here and in the sequel $\operatorname{Var}_\mu(f)$ denotes the variance of $f$ under $\mu$).~In particular, from the Brascamp and Lieb inequality, $C_\kappa \leq 1/\kappa$. For all $f,g \in \mathcal{C}_c^{\infty}(\bbr^d)$, let
\begin{align*}
\mathcal{E}_\mu(f,g) = \int_{\bbr^d} \langle \nabla(f)(x) ; \nabla(g)(x) \rangle \mu(dx).
\end{align*}
The bilinear form $\mathcal{E}_\mu$ is clearly symmetric on $L^2(\mu)$, and let us discuss briefly its closability following \cite[Section $2.6$]{Bog10}.~ First, $\rho_\mu$, the Lebesgue density of $\mu$  is such that, for all $p\in (1, + \infty)$ and for any compact subset, 
$K$, of $\bbr^d$,
\begin{align}\label{eq:cond_density}
\int_K \rho_\mu(x)^{-\frac{1}{p-1}}dx < +\infty.
\end{align}
Based on \eqref{eq:cond_density}, it is not difficult to see that the weighted Sobolev norms $\|\cdot\|_{1,p,\mu}$, defined, for all $f \in \mathcal{C}_c^{\infty}(\bbr^d)$, by 
\begin{align*}
\|f\|_{1,p,\mu} = \|f\|_{L^p(\mu)} + \sum_{k=1}^d \|\partial_k(f)\|_{L^p(\mu)},
\end{align*}
are closable. In particular, for $p=2$, this provides the closability of the form $\mathcal{E}_\mu$. Then, thanks to Theorem \ref{thm:cov_rep_dirichlet_form}, for all $f,g \in \mathcal{S}(\bbr^d)$ with $\mu(f)=0$, 
\begin{align}\label{eq:cov_logconcave}
\mathcal{E}_{\mu}(g,\tilde{f}_\mu) = \langle f;g \rangle_{L^2(\mu)}, \quad \tilde{f}_\mu = \int_0^{+\infty} P^{\mu}_t(f) dt,
\end{align}
with $(P_t^{\mu})_{t\geq 0}$ being the symmetric Markovian semigroup generated by the smallest closed extension of the symmetric bilinear form $\mathcal{E}_\mu$.~Finally, for all $f \in \mathcal{S}(\bbr^d)$, 
\begin{align*}
\Gamma_2(f)(x) & = \sum_{j,k} \left(\partial^2_{j,k}(f)(x)\right)^2 + \langle \nabla(f)(x) ; \operatorname{Hess}(V)(x)\nabla(f)(x)\rangle \geq \kappa \Gamma(f)(x).
\end{align*}
In other words, the curvature-dimension condition $CD(\kappa,\infty)$ is satisfied so that the following strong gradient bound holds true: for all $t>0$ and all $f \in \mathcal{S}(\bbr^d)$,
\begin{align*}
\sqrt{\Gamma(P^{\mu}_t)(f)(x)} \leq e^{- \kappa t} P^{\mu}_t \left(\sqrt{\Gamma(f)}\right)(x).  
\end{align*}
Then, the following asymmetric covariance estimate holds true: for all $f,g \in \mathcal{S}(\bbr^d)$, all $p \in (1,+\infty)$ and $q=p/(p-1)$, 
\begin{align*}
\left| \operatorname{Cov}(f(X);g(X))\right| \leq \frac{1}{\kappa} \|\nabla(g)\|_{L^q(\mu)} \|\nabla(f)\|_{L^p(\mu)},
\end{align*} 
with $X \sim \mu$.~Combining \eqref{eq:cov_logconcave} with estimates from \cite{CCL13}, one retrieves the Brascamp and Lieb inequality for strictly log-concave measures (i.e., such that  $\operatorname{Hess}(V)(x) > 0$, for all $x \in \bbr^d$) as well as its asymmetric versions (see, e.g., \cite[Theorem $1.1$]{CCL13}).\\
(vi) Again, let us discuss a class of probability measures which lies at the interface of the non-local and the local frameworks. Let $m>0$ and let $\mu_m$ be the probability measure on $\bbr^d$ given by 
\begin{align*}
\mu_m(dx) = c_{m,d} \left(1+ \|x\|^2\right)^{-m-d/2}dx,
\end{align*} 
for some normalization constant $c_{m,d}>0$ depending only on $m$ and on $d$.~First, the  characteristic function of a random vector with law $\mu_m$ is given (see \cite[Theorem II]{Tak89}), for all $\xi \in \bbr^d$, by
\begin{align}\label{eq:rep_charac_Cauchy} 
\varphi_m(\xi) = \exp \left(\int_{\bbr^d}\left(e^{i \langle \xi;u \rangle}-1- \frac{i \langle u ; \xi \rangle}{1+\|u\|^2}\right)\nu_m(du)\right). 
\end{align}
Above, $\nu_m$ is the L\'evy measure on $\bbr^d$ given by 
\begin{align}\label{eq:rep_Levy_measure}
\nu_m(du) = \frac{2}{\|u\|^d} \left( \int_0^{+\infty} g_m(2w) L_{\frac{d}{2}}\left(\sqrt{2w}\|u\|\right)dw \right) du,
\end{align}
with, for all $w>0$, 
\begin{align*}
g_m(w) = \frac{2}{\pi^2 w} \dfrac{1}{J_m^2(\sqrt{w})+Y_m^2(\sqrt{w})}, \quad L_{\frac{d}{2}}(w) = \frac{1}{(2\pi)^{\frac{d}{2}}} w^{\frac{d}{2}} K_{\frac{d}{2}}(w),
\end{align*}
$J_m$, $Y_m$ and $K_{d/2}$ denoting respectively the Bessel functions of the first and of the second kind and the modified Bessel function of order $d/2$. Based on the representations \eqref{eq:rep_charac_Cauchy} and \eqref{eq:rep_Levy_measure}, it is clear that $\mu_m$ is self-decomposable so that it is naturally associated with the non-local Dirichlet form given, for all $f,g \in \mathcal{C}_c^{\infty}(\bbr^d)$, by 
\begin{align*}
\mathcal{E}_m(f,g) = \int_{\bbr^d}\int_{\bbr^d} \Delta_u(f)(x)\Delta_u(g)(x) \nu_m(du) \mu_m(dx). 
\end{align*}
Moreover, since $\mu_m$ is infinitely divisible on $\bbr^d$, \cite[Corollary $2$]{HPAS} 
ensures that $\mu_m$ satisfies the following Poincar\'e-type inequality, for all $f$ smooth enough on $\bbr^d$, 
\begin{align*} 
\operatorname{Var}_m(f) \leq \int_{\bbr^d} \int_{\bbr^d} |f(x+u)-f(x)|^2 \nu_m(du) \mu_m(dx),
\end{align*}
(see, \cite[Proposition $5.1$]{AH20_3} for a proof based on semigroup argument when $m>1/2$).~Since $\mu_m \ast \nu_m << \mu_m$, the symmetric bilinear form $\mathcal{E}_m$ is closable so that $(\mathcal{P}^{m}_t)_{t\ge 0}$ the symmetric semigroup generated by the smallest closed extension $(\mathcal{E}_m, \mathcal{D}(\mathcal{E}_m))$ verifies the following ergodic property: for all $f \in L^2(\mu_m)$ with $\int_{\bbr^d} f(x) \mu_m(dx) = 0$, 
\begin{align*}    
\|\mathcal{P}^{m}_t(f)\|_{L^2(\mu_m)} \leq e^{-t} \|f\|_{L^2(\mu_m)}. 
\end{align*}
Then, one can apply Theorem \ref{thm:cov_rep_dirichlet_form} to obtain the following covariance representation: for all $f \in \mathcal{S}(\bbr^d)$ with $\mu_m(f)=
0$ and all $g \in \mathcal{S}(\bbr^d)$, 
\begin{align}\label{eq:covariance_Cauchy_1}
\mathcal{E}_m(g,\tilde{f}_m) = \langle f; g \rangle_{L^2(\mu_m)}, \quad \tilde{f}_m = \int_0^{+\infty} \mathcal{P}^m_t(f) dt.
\end{align}
Now, this class of probability measures has been investigated in the context of weighted Poincar\'e-type inequality (see, e.g., \cite{Bob_Led_09,Bo_Jou_Ma_16}). Indeed, for all $f$ smooth enough on $\bbr^d$ and all $m>0$, 
\begin{align*}
\operatorname{Var}_m(f) \leq C_{m,d} \int_{\bbr^d} f(x) (-\mathcal{L}^\sigma_m)(f)(x) \mu_m(dx),
\end{align*}
where the operator $\mathcal{L}^\sigma_m$ is given, on smooth functions, by
\begin{align*}
\mathcal{L}^\sigma_m(f)(x) = (1+ \|x\|^2) \Delta(f)(x) +2\left(1 - m -\frac{d}{2}\right)\langle x ; \nabla(f)(x) \rangle, 
\end{align*}
and where $C_{m,d}>0$ is a constant depending on $m$ and on $d$ which can be explicitly computed or bounded, depending on the relationships between $m$ and $d$ (see, \cite[Corollaries $5.2$ and $5.3$]{Bo_Jou_Ma_16}). The corresponding Dirichlet form is given, for all $f,g \in \mathcal{C}_c^{\infty}(\bbr^d)$, by
\begin{align*}
\mathcal{E}_m^\sigma (f,g) = \int_{\bbr^d} \left(1+\|x\|^2\right) \langle \nabla(f)(x) ; \nabla(g)(x) \rangle \mu_m(dx). 
\end{align*}
Once again, using the exponential $L^2(\mu_m)$-convergence to equilibrium of the semigroup induced by the form $\mathcal{E}_m^\sigma$ (denoted by $(\mathcal{P}^{m,\sigma}_t)_{t\geq 0}$), one obtains the following covariance representation formula: for all $f,g \in \mathcal{C}_c^{\infty}(\bbr^d)$ with $\mu_m(f)=0$,
\begin{align}\label{eq:covariance_Cauchy_2}
\mathcal{E}_m^\sigma(\tilde{f}_m^\sigma,g) = \langle f;g \rangle_{L^2(\mu_m)}, \quad \tilde{f}_m^\sigma = \int_0^{+\infty} \mathcal{P}^{m,\sigma}_t(f) dt.
\end{align}
Now, using either \eqref{eq:covariance_Cauchy_1} or \eqref{eq:covariance_Cauchy_2}, one gets 
\begin{align*}
&\left|\operatorname{Cov}(f(X_m),g(X_m))\right| \leq \|\nabla_{\nu_m}(f)\|_{L^2(\mu_m)}\|\nabla_{\nu_m}(g)\|_{L^2(\mu_m)} ,\\
&\left|\operatorname{Cov}(f(X_m),g(X_m))\right| \leq C_{m,d} \|\sigma \nabla(f)\|_{L^2(\mu_m)}  \|\sigma \nabla(g)\|_{L^2(\mu_m)},
\end{align*}
with $X_m \sim \mu_m$ and $\sigma(x)^2 = (1+\|x\|^2)$, for all $x \in \bbr^d$.\\
(vii) To conclude this long list of examples, let us return to the $\alpha$-stable probability measures on $\bbr^d$, $\alpha \in (0,2)$.  Let $\mathcal{E}$ be the non-negative definite symmetric bilinear form given, for all $f,g \in \mathcal{S}(\bbr^d)$, by
\begin{align*}
\mathcal{E}(f,g) = \int_{\bbr^d} \int_{\bbr^d} \Delta_u(f)(x) \Delta_u(g)(x) \nu_\alpha(du) \mu_\alpha(dx). 
\end{align*}
Recall that since $\mu_\alpha \ast \nu_\alpha << \mu_\alpha$ the bilinear form $\mathcal{E}$ is closable.~The associated semigroup, $(\mathcal{P}_t)_{t\ge 0}$ is the ``carr\'e de Mehler" semigroup defined in \eqref{def:squared_Mehler} whose $L^2(\mu_\alpha)$-generator is given, for all $f \in\mathcal{S}(\bbr^d)$, by 
\begin{align*}
\mathcal{L}(f) = \frac{1}{\alpha} \left( \mathcal{L}^\alpha + (\mathcal{L}^{\alpha})^*\right)(f),
\end{align*}
and already put forward in \cite{AH20_3,AH20_4}. Moreover, the Poincar\'e inequality for the $\alpha$-stable probability measure implies that, for all $f \in L^2(\mu_\alpha)$ with $\mu_\alpha(f)=0$ and all $t>0$,
\begin{align*}
\|\mathcal{P}_t(f)\|_{L^2(\mu_\alpha)} \leq e^{-t} \|f\|_{L^2(\mu_\alpha)}. 
\end{align*}
Thus, \cite[Theorem $5.10$]{AH20_3} provides the following covariance representation: for all $f,g \in \mathcal{S}(\bbr^d) $ with $\int_{\bbr^d} f(x) \mu_\alpha(dx) = 0$,
\begin{align*}
\int_{\mathbb{R}^d} f(x)g(x) \mu_\alpha(dx) = \int_{\bbr^d} \int_{\bbr^d} \Delta_u(g)(x) \Delta_u(\tilde{f})(x) \nu_\alpha(du)\mu_\alpha(dx) = \mathcal{E}(g,\tilde{f}),
\end{align*}
where
\begin{align*}
\tilde{f} = \int_{0}^{+\infty}\mathcal{P}_t(f)dt =(-\mathcal{L})^{-1} (f). 
\end{align*}
Moreover, still from \cite[Theorem $5.10$]{AH20_3}, for all $f \in \mathcal{S}(\bbr^d)$ with $\int_{\bbr^d}f(x) \mu_\alpha(dx) = 0$, 
\begin{align*}
\mathcal{E}(\tilde{f}, \tilde{f}) \leq \|f\|^{2}_{L^2(\mu_\alpha)} \leq \mathcal{E}(f,f),
\end{align*}
so that, by the Cauchy-Schwarz inequality,
\begin{align*}
\left|\int_{\bbr^d}f(x)g(x) \mu_\alpha(dx) \right|^2 \leq \mathcal{E}(g,g) \mathcal{E}(f,f), 
\end{align*}
which is a particular instance of Proposition \ref{prop:BL_type_ID} (ii), with $p=q=2$ and with $\nu = \nu_\alpha$. 
\end{rem}
\noindent
As detailled next, it is possible to refine the existence result given by Theorem \ref{thm:cov_rep_dirichlet_form}, by using a celebrated characterization of surjective closed densely defined linear operators on a Hilbert space by \textit{a priori} estimates on their adjoints.~This abstract existence result is well-known in the theory of partial differential equations (see, e.g., \cite[Theorem $2.20$]{HBre}) and seems to go back to \cite[Lemma $1.1$]{Hor_55}. Combined with an integration by parts and the Poincar\'e inequality, it allows to retrieve the covariance representations of Remark \ref{rem:weak_strong}. In particular, this characterization result allows to go beyond the assumption of a Poincar\'e inequality for the underlying probability measure in order to prove the existence of Stein's kernels. 

\begin{thm}\label{thm:beyond_Poincar_Hilbert_Version}
Let $H$ be a separable real Hilbert space with inner product $\langle \cdot ; \cdot \rangle_H$ and induced norm $\|\cdot\|_H$. Let $\mathcal{A}$ be a closed and densely defined linear operator on $H$ with domain $\mathcal{D}(\mathcal{A})$ and such that, for all $u \in \mathcal{D}(\mathcal{A}^*)$,
\begin{align}\label{ineq:a_priori_estimate}
\|u\|_{H} \leq C \|\mathcal{A}^*(u)\|_H,
\end{align} 
for some $C>0$ not depending on $u$ and where $(\mathcal{A}^*,\mathcal{D}(\mathcal{A}^*))$ is the adjoint of $(\mathcal{A},\mathcal{D}(\mathcal{A}))$. Then, for all $u \in H$, there exists $G(u) \in \mathcal{D}(\mathcal{A})$ such that, for all $v \in H$,
\begin{align*}
\langle \mathcal{A}(G(u)) ; v \rangle_H = \langle u;v \rangle_H.
\end{align*}
Moreover, if $\mathcal{A}$ is self-adjoint, then, for all $u \in H$,
\begin{align*}
\left|\langle \mathcal{A}(G(u)) ; G(u) \rangle\right| \leq C\|u\|^2_H. 
\end{align*} 
\end{thm}
\noindent
Let us further provide the Banach-space version of the previous result (see, e.g., \cite[Theorem $2.20$]{HBre} for a proof).  

\begin{thm}\label{thm:beyond_Poincar_Banach_Version}
Let $E$ and $F$ be two Banach spaces with respective norms $\|\cdot\|_E$ and $\|\cdot\|_F$. Let $\mathcal{A}$ be a closed and densely defined linear operator on $E$ with domain $\mathcal{D}(\mathcal{A})$ and such that, for all $u^* \in \mathcal{D}(\mathcal{A}^*)$,
\begin{align}\label{ineq:a_priori_estimate_2}
\|u^*\|_{F^*} \leq C \|\mathcal{A}^*(u^*)\|_{E^*},
\end{align} 
for some $C>0$ not depending on $u^*$ and where $(\mathcal{A}^*,\mathcal{D}(\mathcal{A}^*))$ is the adjoint of $(\mathcal{A},\mathcal{D}(\mathcal{A}))$. Then, for all $u \in F$, there exists $G(u) \in \mathcal{D}(\mathcal{A})$ such that,
\begin{align*}
\mathcal{A}(G(u)) = u.
\end{align*}
Conversely, if $\mathcal{A}$ is surjective then, for all $u^* \in \mathcal{D}(\mathcal{A}^*)$,
\begin{align}\label{ineq:a_priori_estimate_2}
\|u^*\|_{F^*} \leq C \|\mathcal{A}^*(u^*)\|_{E^*},
\end{align} 
for some $C>0$ not depending on $u^*$.
\end{thm}

\begin{rem}\label{rem:without_Poinc}
Let us briefly explain how one can apply the previous existence theorem in the classical Gaussian setting. Let $H = L^2(\bbr^d,\gamma)$, let $H_0 = \{f \in H,\, \int_{\bbr^d}f(x) \gamma(dx) = 0\}$ and let $\mathcal{A} =- \mathcal{L}^{\gamma}$.  Note that if $f \in \mathcal{D}(\mathcal{A})$ then, for all $c \in \bbr$, $f+c \in \mathcal{D}(\mathcal{A})$ so that $f_0 = f - \int_{\bbr^d} f(x)\gamma(dx) \in \mathcal{D}(\mathcal{A})$. Moreover, $\mathcal{L}^\gamma$ is a linear densely defined self-adjoint operator on $H$. Finally, from the Gaussian Poincar\'e inequality, for all $f \in \mathcal{D}(\mathcal{A})$ such that $\int_{\bbr^d} f(x) \gamma(dx) = 0$, 
\begin{align*}
\|f\|^2_{H} \leq \int_{\bbr^d} f(x) \left(-\mathcal{L}^{\gamma}\right)(f)(x) \gamma(dx),
\end{align*} 
and so the Cauchy-Schwarz inequality gives, for all such $f$, 
\begin{align*}
\|f\|_{H} \leq \|(-\mathcal{L}^\gamma)^*(f)\|_H.
\end{align*}
Thus, from Theorem \ref{thm:beyond_Poincar_Hilbert_Version}, for all $f \in H_0$, there exists $\tilde{f} \in \mathcal{D}_0(\mathcal{A})$, such that, for all $v \in H_0$, 
\begin{align}\label{eq:cov_ident}
\langle \mathcal{A}(\tilde{f}) ; v  \rangle_H = \langle f; v \rangle_H, 
\end{align}
where $\mathcal{D}_0(\mathcal{A}) = \mathcal{D}(\mathcal{A}) \cap H_0$. Equation \eqref{eq:cov_ident} extends to all $v \in H$ by translation.
\end{rem}
\noindent
To conclude, this section discusses an example where the underlying probability measure does not satisfy an $L^2$-$L^2$ Poincar\'e inequality with respect to the classical energy form but for which it is possible to obtain a covariance identity with the standard ``carr\'e du champs operator" and corresponding estimates. 

\begin{prop}\label{lem:example_one_dim}
Let $\delta \in (0,1)$ and let $\mu_\delta$ be the probability measure on $\bbr$ given by 
\begin{align*}
\mu_\delta(dx) = C_\delta \exp(-|x|^\delta)dx =p_\delta(x)dx,
\end{align*}
for some normalizing constant $C_\delta>0$.~Let $p \in [2,+\infty)$ and let $g$ be in $L^p(\mu_\delta)$ such that $\int_{\bbr}g(x)\mu_\delta(dx) = 0$. Let $f_\delta$ be defined, for all $x\in \bbr$, by
\begin{align}\label{eq:def_f_delta}
f_\delta(x) = \frac{1}{p_\delta(x)} \int_x^{+\infty} g(y)p_\delta(y)dy = -\frac{1}{p_\delta(x)} \int_{-\infty}^x g(y)p_\delta(y)dy.
\end{align}
Then, for all $r\in [1,+\infty)$ such that $r/(r-1)>q = p/(p-1)$,
\begin{align*}
\|f_\delta\|_{L^r(\mu_\delta)} \leq C(p,\delta,r) \|g\|_{L^p(\mu_\delta)},
\end{align*}
for some $C(p,\delta,r)>0$ depending only on $p$, $\delta$ and $r$. Moreover, $f_\delta$ is a weak to solution to 
\begin{align*}
\mathcal{A}_\delta(f_\delta) = g,
\end{align*}
where, for all $f\in \mathcal{C}_c^\infty(\bbr)$ and all $x\in \bbr\setminus \{0\}$, 
\begin{align}\label{def:operator_A}
\mathcal{A}_\delta(f)(x) = - f^\prime(x) + \delta |x|^{\delta-1}\operatorname{sign}(x)f(x). 
\end{align}
\end{prop}

\begin{proof}
First, applying H\"older's inequality, for all $x\in \bbr$, 
\begin{align}\label{ineq:pointwise}
\left| f_\delta(x) \right| \leq \frac{1}{p_\delta(x)} \left(\int_x^{+\infty}  p_\delta(y) dy \right)^{\frac{1}{q}} \|g\|_{L^p(\mu_\delta)} = G_\delta(x) \|g\|_{L^p(\mu_\delta)}, 
\end{align}
with $G_\delta(x) = \frac{1}{p_\delta(x)} \left(\int_x^{+\infty}  p_\delta(y) dy \right)^{\frac{1}{q}} $, for all $x \in \bbr$. Moreover, by a change of variables, for all $x>0$, 
\begin{align*}
\int_x^{+\infty} p_\delta(y)dy = \frac{1}{\delta} \int_{x^{\delta}}^{+\infty} z^{\frac{1}{\delta}-1}e^{-z}dz = \frac{1}{\delta} \Gamma\left(\frac{1}{\delta},x^{\delta}\right),
\end{align*}
where $\Gamma\left(\frac{1}{\delta},x\right)$ is the incomplete gamma function at $x>0$. Now,
\begin{align*}
\Gamma\left(\frac{1}{\delta},x\right) \underset{x\rightarrow +\infty}{\sim} x^{\frac{1}{\delta}-1} e^{-x} \Rightarrow \int_x^{+\infty} p_\delta(y)dy \underset{x\rightarrow +\infty}{\sim} \frac{1}{\delta} x^{1-\delta} e^{-x^{\delta}}, 
\end{align*}
and similarly as $x \rightarrow -\infty$ for the integral $\int_{-\infty}^x p_\delta(y)dy$
(recall that $\mu_\delta$ is symmetric). Now, 
\begin{align*}
\int_0^{+\infty} |f_\delta(x)|^r \mu_\delta(dx) & \leq \|g\|^r_{L^p(\mu_\delta)}\left(\int_0^{+\infty} p_\delta(x)(G_\delta(x))^r dx \right), \\
& \leq C_{1,p,\delta,r} \|g\|^r_{L^p(\mu_\delta)},
\end{align*}
with,
\begin{align*}
C_{1,p,\delta,r} := \int_0^{+\infty} p_\delta(x)(G_\delta(x))^r dx<+\infty,
\end{align*}
since $r/(r-1)>q$. A similar estimate holds true for the integral of $|f_\delta|^r$ on $(-\infty,0)$ thanks to the second integral representation of $f_\delta$. Finally, for all $\psi \in \mathcal{C}_c^{\infty}(\bbr)$,
\begin{align*}
\int_\bbr \mathcal{A}_\delta(f_\delta)(x)\psi(x)p_\delta(x)dx = \int_\bbr f_\delta(x)\psi'(x) p_\delta(x)dx = \int_{\bbr} g(x)\psi(x) p_\delta(x)dx.
\end{align*}
The conclusion follows.
\end{proof}
\noindent
The next proposition investigates the properties of the unique primitive function $F_\delta$ of $f_\delta$ such that $\int_{\bbr} F_\delta(x)p_\delta(x)dx=0$.

\begin{prop}\label{lem:example_one_dim_primi}
Let $\delta\in (0,1)$, let $p\in [2,+\infty)$, let $g\in L^p(\mu_\delta)$ with $\mu_\delta(g)=0$ and let $f_\delta$ be given by \eqref{eq:def_f_delta}. Let $F_\delta$ be defined, for all $x\in \bbr$, by
\begin{align}\label{eq:def_F_delta}
F_\delta(x) = F_\delta(0) + \bbone_{(0,+\infty)}(x)\int_0^x f_\delta(y)dy-\bbone_{(-\infty,0)}(x)\int_x^0f_\delta(y)dy,
\end{align}
with,
\begin{align*}
F_\delta(0)= \int_{-\infty}^0 \left(\int_x^0 f_\delta(y)dy\right) p_\delta(x)dx - \int_0^{+\infty} \left(\int_0^{x}f_\delta(y)dy\right) p_\delta(x)dx. 
\end{align*}
Then, for all $r\in [1,+\infty)$ such that $r/(r-1)>q=p/(p-1)$, 
\begin{align}\label{eq:Lp_Lr_estimates}
\|F_\delta\|_{L^r(\mu_\delta)} \leq C_2(p,\delta,r) \|g\|_{L^p(\mu_\delta)},
\end{align}
for some $C_2(p,\delta,r)>0$ depending only on $p$, $\delta$ and $r$. Moreover, $F_\delta$ is a weak solution to 
\begin{align*}
(-\mathcal{L}_\delta)(F_\delta) = g,
\end{align*}
where, for all $f \in \mathcal{C}_c^{\infty}(\bbr)$ and all $x\in \bbr\setminus \{0\}$, 
\begin{align}\label{def:operator_L}
\mathcal{L}_\delta(f)(x) = f''(x) -\delta |x|^{\delta-1}\operatorname{sgn}(x)f'(x). 
\end{align}
\end{prop}

\begin{proof}
Let $ \delta$, $p$ and $r$ be as in the statement of the lemma. Then, by convexity and Fubini's theorem,  
\begin{align*}
\int_{0}^{+\infty} \left|\int_0^x f_\delta(y)dy\right|^r p_\delta(x)dx & \leq \int_0^{+\infty} x^{r-1}\left(\int_0^x |f_\delta(y)|^r dy\right) p_\delta(x)dx , \\
& \leq \int_0^{+\infty} |f_\delta(y)|^r \left(\int_y^{+\infty} x^{r-1} p_\delta(x)dx\right) dy.
\end{align*}
Using \eqref{ineq:pointwise}, 
\begin{align*}
\int_{0}^{+\infty} \left|\int_0^x f_\delta(y)dy\right|^r p_\delta(x)dx & \leq \|g\|^r_{L^p(\mu_\delta)} \int_0^{+\infty} (G_\delta(y))^r \left(\int_y^{+\infty} x^{r-1} p_\delta(x)dx\right) dy. 
\end{align*}
Now, since $r/(r-1)>q=p/(p-1)$,
\begin{align*}
\int_0^{+\infty} (G_\delta(y))^r \left(\int_y^{+\infty} x^{r-1} p_\delta(x)dx\right) dy <+\infty. 
\end{align*}
A similar analysis can be performed for the integral of $|\int_x^0 f_\delta(y)dy|^r$ over $(-\infty,0)$. This proves the $L^p-L^r$ estimate \eqref{eq:Lp_Lr_estimates}. Noticing that $(-\mathcal{L}_\delta)(F_\delta) = \mathcal{A}_\delta(f_\delta)$ the end of the proof follows.
\end{proof}
\noindent
To finish this section, let us discuss the higher dimensional situations, namely $d \geq 2$. As previously, the second order differential operator under consideration is given, for all $f \in  \mathcal{C}^{\infty}_c(\bbr^d)$ and all $x \in \bbr^d \setminus \{0\}$, by
\begin{align*}
\mathcal{L}_{d,\delta}(f)(x) = \Delta(f)(x) - \delta \|x\|^{\delta-2} \langle x; \nabla(f)(x)\rangle.
\end{align*} 
Once again, an integration by parts ensures that the operator is symmetric on $\mathcal{C}_c^{\infty}(\bbr^d)$ with
\begin{align*}
\int_{\bbr^d} (- \mathcal{L}_{d,\delta})(f)(x) g(x) \mu_{d,\delta}(dx) = \int_{\bbr^d} \langle \nabla(f)(x) ; \nabla(g)(x)\rangle \mu_{d,\delta}(dx), 
\end{align*}
and thus closable. This operator is essentially self-adjoint as soon as the logarithmic derivative of the Lesbegue density of the probability measure $\mu_{d,\delta}$ belongs to the Lesbegue space $L^4(\bbr^d, \mu_{d,\delta})$ which is the case when $\delta \in (1 - d/4,1)$ (see \cite{LiskSem92}). Note that if $d \geq 4$ this is true for all $\delta \in (0,1)$. 
Next, let $\varphi_\delta$ be the function defined, by
\begin{align*}
\varphi_\delta(t) = \exp\left(-t^{\frac{\delta}{2}}\right), \quad t\in (0,+\infty),
\end{align*}
so that, $p_{d,\delta}(x) = C_{d,\delta} \varphi_{\delta}(\|x\|^2)$, for all $x \in \bbr^d$, and let us consider the $\bbr^d$-valued function $\tau_{\delta}= (\tau_{\delta,1}, \dots, \tau_{\delta, d})$ defined, for all $k \in \{1, \dots, d\}$ and all $x \in \bbr^d$, by
\begin{align*}
\tau_{\delta,k}(x) = \frac{1}{2\varphi_{\delta}(\|x\|^2)} \int_{\|x\|^2}^{+\infty} \varphi_{\delta}(t) dt. 
\end{align*}
Then, by a straightforward integration by parts, for all $\psi \in \mathcal{C}_c^{\infty}(\bbr^d)$ and all $k \in \{1, \dots, d\}$, 
\begin{align*}
\int_{\bbr^d} \tau_{k,\delta}(x) \partial_k(\psi)(x) \mu_{d,\delta}(dx) = \int_{\bbr^d} x_k\psi(x) \mu_{d,\delta}(dx). 
\end{align*}
Moreover, it is clear that $\|\tau_{\delta,k}\|_{L^2(\bbr^d , \mu_{d,\delta})}<+\infty$, for all $k \in \{1, \dots, d\}$. More generally, let $g = (g_1, \dots, g_d)\in L^2(\bbr^d , \mu_{d,\delta})$ with $\int_{\bbr^d}g_k(x) \mu_{d,\delta}(dx) = 0$, for all $k \in \{1, \dots, d\}$, and let us study the following weak formulation problem: 
\begin{align*}
\mathcal{L}_{d,\delta}(f) = g, \quad \int_{\bbr^d} f(x)\mu_{d,\delta}(dx) = 0.
\end{align*}
A first partial answer to the previous weak formulation problem is through the use of semigroup techniques combined with weak Poincar\'e-type inequality.~From \cite[Example $1.4$, c)]{RW_01}, the semigroup, $(P^\delta_t)_{t\geq 0}$, generated by the self-adjoint extension of $\mathcal{L}_{d,\delta}$ satisfies the following estimate: for all $g \in L^{\infty}(\mu_{d,\delta})$ with $\int_{\bbr^d} g(x) \mu_{d,\delta}(dx)=0$ and all $t \geq 0$,
\begin{align*}
\|P^\delta_t(g)\|_{L^2(\mu_{d,\delta})} \leq c_1 \|g\|_{L^{\infty}(\mu_{d,\delta})} \exp \left(-c_2 t^{\frac{\delta}{4-3\delta}}\right),
\end{align*}
for some $c_1,c_2>0$ depending only on $d,\delta$. Thus, setting 
\begin{align*}
\tilde{f}_{\delta} = \int_0^{+\infty} P^{\delta}_t(g) dt,
\end{align*}
and reasoning as in \cite[Theorem $5.10$]{AH20_3}, for all $\psi \in \mathcal{C}^{\infty}_c(\bbr^d)$,
\begin{align*}
\int_{\bbr^d} \langle \nabla(\tilde{f}_{\delta})(x) ; \nabla(\psi)(x) \rangle \mu_{d,\delta}(dx) = \langle g ; \psi \rangle_{L^2(\mu_{d,\delta})},
\end{align*}
namely, $\tilde{f}_{\delta}$ is a solution to the weak formulation problem with $g \in L^{\infty}(\mu_{d,\delta})$ such that $\mu_{d,\delta}(g)=0$ and with
\begin{align*}
\|\tilde{f}_{\delta}\|_{L^2(\mu_{d,\delta})} \leq C_{d,\delta} \|g\|_{L^\infty(\mu_{d,\delta})}. 
\end{align*}

\section{Representation Formulas and $L^p$-Poincar\'e Inequalities}
\noindent
Let us start this section with a result valid for the nondegenerate symmetric $\alpha$-stable probability measures on $\bbr^d$, $\alpha \in (1,2)$.

\begin{prop}\label{prop:stable_ineq_1}
Let $d \geq 1$, let $\alpha \in (1,2)$ and let $\mu_\alpha$ be a nondegenerate symmetric $\alpha$-stable probability measure on $\bbr^d$.~Let $p \in (1, +\infty)$ and let $p_1,p_2$ be such that $1/p=1/p_1+1/p_2$ with $1 < p_1 < \alpha$. Then, for all $f$ smooth enough on $\bbr^d$, 
\begin{align*}
\|f - \mu_\alpha(f)\|_{L^p(\mu_\alpha)} \leq \left(\int_0^{+\infty} q_\alpha(t) dt\right) \left(\bbe \|X\|^{p_1} \right)^{\frac{1}{p_1}} \|\nabla(f)\|_{L^{p_2}(\mu_\alpha)},
\end{align*} 
where $X \sim \mu_\alpha$ and $q_\alpha$ is defined, for all $t>0$, by
\begin{align*}
q_\alpha(t) = \dfrac{e^{-t}}{(1-e^{-\alpha t})^{1- \frac{1}{\alpha}}} \left( (1-e^{- \alpha t})^{\alpha -1}+e^{-\alpha t} \right)^{\frac{1}{\alpha}}.
\end{align*}
\end{prop}

\begin{proof}
A straightforward application of Bismut formula \eqref{eq:Bismut_Stable} (see also \cite[Proposition $2.1$]{AH20_4}) for the action of the operator $D^{\alpha-1}$ on $P^{\nu_\alpha}_t(f)$, $t>0$ and $f \in \mathcal{S}(\bbr^d)$, together with the decomposition of the non-local part of the generator of the $\alpha$-stable Ornstein-Uhlenbeck semigroup imply, for all $f \in \mathcal{S}(\bbr^d)$ and all $x \in \bbr^d$, that   
\begin{align*}
f(x) - \bbe f(X) = \int_{0}^{+\infty} \int_{\bbr^d} \langle xe^{-t} - \frac{e^{-\alpha t}y}{(1- e^{-\alpha t})^{1- \frac{1}{\alpha}}} ; \nabla(f)(xe^{-t} + (1-e^{-\alpha t})^{\frac{1}{\alpha}}y) \rangle \mu_\alpha(dy) dt,
\end{align*}
where $X \sim \mu_\alpha$. Therefore,
\begin{align*}
\left|f(x) - \bbe f(X)\right| & \leq  \int_{0}^{+\infty} \bbe_Y \left| \langle x e^{-t} - \frac{e^{-\alpha t} Y}{(1 - e^{-\alpha t})^{1 -\frac{1}{\alpha}}} ;  \nabla(f)(xe^{-t} + (1 - e^{- \alpha t})^{\frac{1}{\alpha}}Y)\rangle \right| dt, 
\end{align*}
with $Y \sim \mu_\alpha$. Thus, by Minkowski's integral inequality and Jensen's inequality, 
\begin{align*}
\left(\bbe_X |f (X)- \bbe f(X)|^p \right)^{\frac{1}{p}} & \leq \int_0^{+\infty} \left(\bbe_{X,Y} \left|\langle X e^{-t} - \frac{e^{-\alpha t}}{(1 - e^{- \alpha t})^{1 - \frac{1}{\alpha}}}Y  ; \nabla(f)\left(X e^{- t} + (1 -e^{- \alpha t})^{\frac{1}{\alpha}}Y\right) \rangle\right|^p \right)^{\frac{1}{p}} dt.
\end{align*}
Now, thanks to stability, under the product measure $\mu_\alpha \otimes \mu_\alpha$, $X e^{- t} + (1 - e^{-\alpha t})^{\frac{1}{\alpha}} Y $ is distributed according to $\mu_\alpha$. Moreover,  
\begin{align*}
Xe^{-t} - \frac{e^{-\alpha t}Y}{(1-e^{- \alpha t})^{1 - \frac{1}{\alpha}}} =_{\mathcal{L}} \left(e^{- \alpha t}\dfrac{(1-e^{- \alpha t})^{\alpha -1}+e^{-\alpha t}}{(1-e^{-\alpha t})^{\alpha - 1}}\right)^{\frac{1}{\alpha}} X .
\end{align*}
Finally, observe that, for all $t>0$, 
\begin{align*}
q_\alpha(t) = \left(e^{- \alpha t}\dfrac{(1-e^{- \alpha t})^{\alpha -1}+e^{-\alpha t}}{(1-e^{-\alpha t})^{\alpha - 1}}\right)^{\frac{1}{\alpha}}&  = \dfrac{e^{-t}}{(1-e^{-\alpha t})^{1- \frac{1}{\alpha}}} \left( (1-e^{- \alpha t})^{\alpha -1}+e^{-\alpha t} \right)^{\frac{1}{\alpha}}.
\end{align*}
Let $F_1$ and $F_2$ be the two functions defined, for all $x,y \in \bbr^d$ and all $t>0$, by
\begin{align*}
& F_1(x,y,t) = x e^{-t} - \frac{e^{-\alpha t}y}{(1-e^{- \alpha t})^{1 - \frac{1}{\alpha}}} , \\
& F_2(x,y,t) = \nabla(f) \left(xe^{-t} + (1- e^{-\alpha t})^{\frac{1}{\alpha}} y\right).
\end{align*}
Then, from the generalized H\"older's inequality, for all $p \in (1, +\infty)$, 
\begin{align*}
\| \langle F_1(,t) ; F_2(,t) \rangle\|_{L^p(\mu_\alpha\otimes \mu_\alpha)} \leq \| F_1(,t)\|_{L^{p_1}(\mu_\alpha \otimes \mu_\alpha)} \|F_2(,t)\|_{L^{p_2}(\mu_\alpha \otimes \mu_\alpha)},
\end{align*}
where $1/p_1 + 1/p_2 = 1/p$. Take $1 <p_1<\alpha$.  From the previous identities in law, one gets
\begin{align*}
\|f - \mu_\alpha(f) \|_{L^p(\mu_\alpha)} \leq \|X\|_{L^{p_1}(\mu_\alpha)} \|\nabla(f)\|_{L^{p_2}(\mu_\alpha)} \left(\int_{0}^{+\infty} q_\alpha(t) dt \right).
\end{align*}
This concludes the proof of the proposition. 
\end{proof}
\noindent
Before moving on, let us briefly comment on the Gaussian situation.~Let $\gamma$ be the standard Gaussian probability measure on $\bbr^d$. Following lines of reasoning as above, it is not difficult to obtain the corresponding inequality for the standard Gaussian probability measure on $\bbr^d$. However, a crucial difference with the general symmetric $\alpha$-stable situation is that, under the product probability measure $\gamma \otimes \gamma$, the Gaussian random vectors given, for all $t>0$, by
\begin{align*}
Xe^{-t} + \sqrt{1 - e^{-2t}}Y, \quad Xe^{-t} - \frac{e^{-2 t}Y}{\sqrt{1-e^{-2t}}},
\end{align*} 
where $(X,Y) \sim \gamma \otimes \gamma$, are independent of each other and equal in law to a standard $\bbr^d$-valued Gaussian random vector (up to some constant depending on $t$ for the second one). Finally, conditioning, one gets the following classical dimension free inequality which is a particular case of a result of Pisier (see, e.g.,  \cite[Theorem $2.2$]{P86}), for all $f$ smooth enough on $\bbr^d$ and all $p \in (1,+\infty)$,
\begin{align}\label{ineq:Pisier}
\|f- \gamma(f)\|_{L^p(\gamma)} \leq \frac{\pi}{2} \left(\bbe |Z|^p \right)^{\frac{1}{p}} \|\nabla(f)\|_{L^p(\gamma)}, 
\end{align}
where $Z \sim \mathcal{N}(0,1)$. Note that the constant in \eqref{ineq:Pisier} is not optimal since for $p=2$ the best constant is known to be equal to $1$ and for large $p$ it is of the order $\sqrt{p}$. Note also that the previous lines of reasoning continue to hold in the vector-valued setting (namely when $f$ and $g$ are vector-valued in a general Banach space).~Finally, a different estimate has been obtained in \cite[Theorem $7.1$ and Remark $7.2$]{BH97} which is linked to the isoperimetric constant and to the product structure of the standard Gaussian probability measure on $\bbr^d$. Note that the dependency on $p$ in \cite[inequality $7.5$]{BH97} is of the order $p$.  Moreover, in \cite[Proposition 3.1]{NPR_09}, the following fine version of the $L^p$-Poincar\'e inequality on the Wiener space is proved: for all \textit{even integers} $p \geq 2$ and all $F \in \mathbb{D}^{1,p}$ such that $\bbe F = 0$, 
\begin{align*}
\left(\bbe |F|^p \right)^{\frac{1}{p}} \leq (p-1)^{\frac{1}{2}} \left( \bbe \|D F\|^p_{\mathcal{H}} \right)^{\frac{1}{p}},
\end{align*}
where $DF$ is the Malliavin derivative of $F$,  $\mathcal{H}$ is a real separable Hilbert space on which the isonormal Gaussian process is defined and $\mathbb{D}^{1,p}$ is the $L^p$-Sobolev-Watanabe-Kree space of order $1$. Finally, recently, it has been proved in \cite[Theorem $2.6$]{AMR_21} that: for all $p \geq 2$ and all $F \in \mathbb{D}^{1,p}$ such that $\bbe F = 0$,  
\begin{align*}
\left(\bbe |F|^p \right)^{\frac{1}{p}} \leq (p-1)^{\frac{1}{2}} \left( \bbe \|D F\|^p_{\mathcal{H}} \right)^{\frac{1}{p}}.
\end{align*}
As shown next, with an argument based on the covariance identity obtained in \cite[Proposition $2$]{HPAS}, it is possible to easily retrieve, for $p \geq 2$, such estimates (see also the discussion in \cite[pages 14-15]{SW19}). 

\begin{prop}\label{prop:sharp_gauss_ineq}
Let $d\geq 1$, let $\gamma$ be the standard Gaussian probability measure on $\bbr^d$ and let $p \in [2, +\infty)$. Then, for all $f\in \mathcal{S}(\bbr^d)$ such that $\int_{\bbr^d} f(x) \gamma(dx) = 0$,
\begin{align}\label{ineq:sharp_Lp_Gauss}
\|f\|_{L^p(\gamma)} \leq \sqrt{p-1} \|\nabla(f)\|_{L^p(\gamma)}.
\end{align}
\end{prop}

\begin{proof}
From \eqref{eq:covariance_representation_gaussian}, for all $f, g$ smooth enough and real-valued with $\int_{\bbr^d} f(x) \gamma(dx) = 0$, 
\begin{align*}
\bbe f(X)g(X) = \int_0^1 \bbe \langle \nabla(f)(X_z) ; \nabla(g)(Y_z) \rangle dz,
\end{align*}
where $X_z =_{\mathcal{L}} Y_{z} =_{\mathcal{L}} X \sim \gamma$, for all $z \in [0,1]$.~Next, let $p \geq 2$ and take $g = \Phi'_p(f)/p$ where $\Phi_p(x) = |x|^p$, for $x \in \bbr$. Then, since $\Phi_p$ is twice continuously differentiable on $\bbr$,
\begin{align*}
\nabla(g)(x) = \frac{1}{p} \nabla(f)(x) \Phi''_p(f(x)) = (p-1) \nabla(f)(x) |f(x)|^{p-2}, \quad x \in \bbr^d.
\end{align*}
Thus, for all $f$ smooth enough with mean $0$ with respect to the Gaussian measure $\gamma$,  
\begin{align*}
\bbe |f(X)|^p = (p-1) \int_0^1 \bbe |f(Y_z)|^{p-2} \langle \nabla(f)(X_z) ; \nabla(f)(Y_z) \rangle dz.
\end{align*}
Using H\"older's inequality with $r = p/(p-2)$ and $r^* = p/2$ as well as the Cauchy-Schwarz inequality, 
\begin{align*}
|\bbe |f(Y_z)|^{p-2} \langle \nabla(f)(X_z) ; \nabla(f)(Y_z) \rangle| & \leq \left(\bbe \left| f(X) \right|^p \right)^{1-\frac{2}{p}} \left(\bbe \| \nabla(f)(X_z) \|^{\frac{p}{2}} \|\nabla(f)(Y_z)\|^{\frac{p}{2}} \right)^{\frac{2}{p}}, \\
& \leq \left(\bbe \left| f(X) \right|^p \right)^{1-\frac{2}{p}} \left(\bbe \| \nabla(f)(X) \|^p \right)^{\frac{2}{p}}.
\end{align*}
Assuming that $f \ne 0$, the rest of the proof easily follows. 
\end{proof}

\begin{rem}\label{rem:Lp_Poincar_anisotrope}
From the covariance representation \eqref{eq:covariance_representation_gaussian} in the general case, it is possible to obtain a version of the $L^p$-Poincar\'e inequality for the centered Gaussian probability measure with covariance matrix $\Sigma$.~Namely, for all $ p \in [2,+\infty)$ and all $f\in \mathcal{S}(\bbr^d)$ such that $\int_{\bbr^d} f(x) \gamma_\Sigma(dx) = 0$,
\begin{align}\label{ineq:sharp_Lp_Gauss_anisotrope}
\|f\|_{L^p(\gamma_\Sigma)} \leq \sqrt{p-1} \| \Sigma^{\frac{1}{2}} \nabla(f)\|_{L^p(\gamma_\Sigma)}.
\end{align}

\end{rem}
\noindent
As a corollary of the previous $L^p$-Poincar\'e inequality, let us prove a Sobolev-type inequality with respect to the standard Gaussian measure on $\bbr^d$. 

\begin{cor}\label{cor:Sobolev_type_ineq}
Let $d \geq 1$, let $\gamma$ be the standard Gaussian probability measure on $\bbr^d$ and let $p \in [2, +\infty)$.~Then, for all $f \in \mathcal{S}(\bbr^d)$ such that $\int_{\bbr^d} f(x) \gamma(dx) = 0$ and all $\lambda >0$,
\begin{align*}
\|f\|_{L^p(\gamma)} \leq \sqrt{p-1} C(\lambda,p) \left( \lambda \|f\|_{L^p(\gamma)} + \|(\mathcal{-L}^\gamma)(f)\|_{L^p(\gamma)}\right)
\end{align*}
where $C(\lambda,p)$ is given by
\begin{align*}
C(\lambda,p) : = \gamma_2(q) \left(\int_{0}^{+\infty} \dfrac{e^{- (\lambda+1) t}}{\sqrt{1 - e^{-2t}}} dt\right),  \quad \gamma_2(q) = (\bbe |X|^q)^{\frac{1}{q}},
\end{align*}
where $X \sim \gamma$ and where $q = p/(p-1)$.  In particular, 
\begin{align*}
\| f \|_{L^p(\gamma)} \leq \frac{\pi}{2} \sqrt{p-1} \gamma_2(q) \|(- \mathcal{L}^{\gamma})(f)\|_{L^p(\gamma)}. 
\end{align*}
\end{cor}

\begin{proof}
The proof is rather straightforward and is a consequence of the Bismut formula for the standard Gaussian measure on $\bbr^d$. For all $t>0$, all $f \in \mathcal{S}(\bbr^d)$ with $\int_{\bbr^d} f(x) \gamma(dx) = 0$ and all $x \in \bbr^d$, 
\begin{align*}
\nabla P^\gamma_t(f)(x) = \dfrac{e^{-t}}{\sqrt{1- e^{-2t}}} \int_{\bbr^d} y f(xe^{-t} + y \sqrt{1 - e^{-2t}}) \gamma(dy).
\end{align*} 
Now, based on the previous formula, it is clear that 
\begin{align*}
\nabla \circ \left(\lambda E - \mathcal{L}^\gamma\right)^{-1} (f)(x) & = \int_{0}^{+\infty} e^{- \lambda t} \nabla P_t^{\gamma}(f)(x) dt , \\
& = \int_{\bbr^d} y I_{2, \lambda}(f)(x,y) \gamma(dy) ,
\end{align*}
with,
\begin{align*}
I_{2, \lambda}(f)(x,y) = \int_{0}^{+\infty} \dfrac{e^{-(1+ \lambda)t}}{ \sqrt{1 - e^{-2t}}} f\left(xe^{-t} + y \sqrt{1 - e^{-2t}}\right)dt . 
\end{align*}
Next, by duality and H\"older's inequality, 
\begin{align*}
\| \nabla \circ \left(\lambda E - \mathcal{L}^\gamma\right)^{-1} (f)(x) \| & = \underset{z \in \bbr^d, \,  \|z\| = 1}{\sup} \left| \int_{\bbr^d} \langle z; y \rangle I_{2, \lambda}(f)(x,y) \gamma(dy)  \right| , \\
& \leq \gamma_2(q) \left(\int_{\bbr^d} |I_{2, \lambda}(f)(x,y)|^p \gamma(dy)\right)^{\frac{1}{p}}.
\end{align*}
Taking the $ L^p(\gamma)$-norm and applying Minkowski's integral inequality give
\begin{align*}
\| \nabla \circ \left(\lambda E - \mathcal{L}^\gamma\right)^{-1} (f) \|_{L^p(\gamma)} & \leq \gamma_2(q) \| I_{2, \lambda}(f) \|_{L^p(\gamma \otimes \gamma)} , \\
& \leq \gamma_2(q) \left( \int_{0}^{+\infty} \dfrac{e^{-(\lambda +1)t}}{\sqrt{1 - e^{-2t}}} dt \right) \|f\|_{L^p(\gamma)}. 
\end{align*}
Thus, for all $f \in \mathcal{S}(\bbr^d)$ such that $\int_{\bbr^d} f(x) \gamma(dx) = 0$, 
\begin{align*}
\|f\|_{L^p(\gamma)} \leq \sqrt{p-1} \| \nabla (f) \|_{L^p(\gamma)} & \leq \sqrt{p-1} \gamma_2(q) \left( \int_{0}^{+\infty} \dfrac{e^{-(1+\lambda)t}}{\sqrt{1 - e^{-2t}}} dt\right) \|(\lambda E - \mathcal{L}^{\gamma})(f)\|_{L^p(\gamma)} ,  \\
&  \leq \sqrt{p-1} \gamma_2(q) \left( \int_{0}^{+\infty} \dfrac{e^{-(1+\lambda)t}}{\sqrt{1 - e^{-2t}}} dt\right) \left( \lambda \|f\|_{L^p(\gamma)}  + \| (- \mathcal{L}^\gamma)(f)\|_{L^p(\gamma)} \right) . 
\end{align*}
The conclusion easily follows since 
\begin{align*}
\int_{0}^{+\infty} \dfrac{e^{-t}}{\sqrt{1-e^{-2t}}}dt = \frac{\pi}{2}. 
\end{align*}
\end{proof}
\noindent
Let us return to the nondegenerate symmetric $\alpha$-stable case with $\alpha \in (1,2)$.~Based on the following decomposition of the non-local part of the generator of the stable Ornstein-Uhlenbeck semigroup (and on Bismut-type formulas), 
\begin{align}\label{eq:frac_OU_decompo}
\mathcal{L}^\alpha(f)(x) = -\langle x; \nabla(f)(x) \rangle + \sum_{j = 1}^d \partial_j D^{\alpha-1}_j(f)(x), 
\end{align}
$L^p$-Poincar\'e-type inequalities for the symmetric nondegenerate $\alpha$-stable probability measures on $\bbr^d$ with $\alpha \in (1,2)$ and with $p \in [2, +\infty)$ are discussed.

At first, let us provide an analytic formula for the dual semigroup $((P^{\nu_\alpha}_t)^*)_{t \geq 0}$ of the $\alpha$-stable Ornstein-Uhlenbeck semigroup.~This representation follows from \eqref{eq:StOUSM}.~Recall that $p_\alpha$, the Lebesgue density of a nondegenerate $\alpha$-stable probability measure with $\alpha \in (1,2)$ is positive on $\bbr^d$ (see, e.g., \cite[Lemma $2.1$]{W07}). 

\begin{lem}\label{lem:int_rep_dual_semi}
Let $d \geq 1$, let $\alpha \in (1,2)$, let $\mu_\alpha$ be a nondegenerate symmetric $\alpha$-stable probability measure on $\bbr^d$, and let $p_\alpha$ be its Lebesgue density. Then, for all $g \in \mathcal{S}(\bbr^d)$, all $t>0$ and all $x \in \bbr^d$, 
\begin{align}\label{eq:int_rep_dual}
(P^{\nu_\alpha}_t)^*(g)(x) & = \dfrac{1}{(1-e^{-\alpha t})^{\frac{d}{\alpha}}} \int_{\bbr^d} g(u) p_\alpha(u) p_\alpha \left(\dfrac{x-u e^{-t}}{(1-e^{-\alpha t})^{\frac{1}{\alpha}}}\right)\frac{du}{p_\alpha(x)} , \nonumber \\
& = \dfrac{e^{td}}{(1-e^{-\alpha t})^{\frac{d}{\alpha}}} \int_{\bbr^d} g(e^t x +e^t z) \dfrac{p_\alpha(xe^{t}+ze^{t})}{p_\alpha(x)} p_\alpha\left(\frac{z}{(1-e^{-\alpha t})^{\frac{1}{\alpha}}}\right) dz. 
\end{align} 
\end{lem}

\begin{proof}
Let $f,g \in \mathcal{S}(\bbr^d)$ and let $t>0$. Then, 
\begin{align*}
\int_{\bbr^d} P^{\nu_\alpha}_t(f)(x) g(x) p_\alpha(x)dx = \int_{\bbr^{2d}} f\left(xe^{-t}+(1-e^{-\alpha t})^{\frac{1}{\alpha}}y\right) p_\alpha(y)g(x) p_\alpha(x) dxdy. 
\end{align*}
Now, let us perform several changes of variables: first change $y$ into $z/(1-e^{-\alpha t})^{\frac{1}{\alpha}}$, then $x$ into $e^t u$ and finally, $(u+z,u)$ into $(x,y)$. Then, 
\begin{align*}
\int_{\bbr^d} P^{\nu_\alpha}_t(f)(x) g(x) p_\alpha(x)dx = \int_{\bbr^{2d}} f(x) g(e^ty)p_\alpha(ye^t)p_\alpha\left(\dfrac{x-y}{(1-e^{-\alpha t})^{\frac{1}{\alpha}}}\right)\frac{e^{td}dxdy}{(1-e^{-\alpha t})^{\frac{d}{\alpha}}}. 
\end{align*}
This concludes the proof of the lemma. 
\end{proof}
\noindent
Note that this representation generalizes completely the case $\alpha=2$ for which $(P^{\gamma}_t)^* = P^{\gamma}_t$, for all $t\geq 0$. Also, note that the previous representation ensures that, for all $g \in \mathcal{S}(\bbr^d)$ and all $t>0$, 
\begin{align*}
\int_{\bbr^d} (P^{\nu_\alpha}_t)^*(g)(x) \mu_\alpha(dx) = \int_{\bbr^d} g(x) \mu_\alpha(dx),
\end{align*}
which can be seen using a duality argument and to the fact that $P^{\nu_\alpha}_t$, $t>0$, is mass conservative.~Based on \eqref{eq:int_rep_dual}, let us give a specific representation of the dual semigroup as the composition of three elementary operators.~For this purpose, denote by $M_\alpha$ the multiplication operator by the stable density $p_\alpha$. Namely, for all $g \in \mathcal{S}(\bbr^d)$ and all $x \in \bbr^d$,
\begin{align*}
M_\alpha(g)(x) = g(x) p_\alpha(x). 
\end{align*}
The inverse of $M_\alpha$ corresponds to multiplication by $1/p_\alpha$.~Now, denote by $(T_t^{\alpha})_{t\geq 0}$, the continuous family of operators, defined, for all $g \in \mathcal{S}(\bbr^d)$, all $x \in \bbr^d$ and all $t>0$, by
\begin{align*}
T_t^\alpha(g)(x) = \int_{\bbr^d} g(u) p_\alpha \left(\dfrac{x-u e^{-t}}{(1-e^{-\alpha t})^{\frac{1}{\alpha}}}\right) \frac{du}{(1-e^{-\alpha t})^{\frac{d}{\alpha}}},
\end{align*}
with the convention that $T_0^\alpha(g)=g$.~For fixed $t>0$, the previous operator admits a representation which is close in spirit to the Mehler representation of the semigroup $(P^{\nu_\alpha}_t)_{t \geq 0}$: for all $g \in \mathcal{S}(\bbr^d)$, all $x \in \bbr^d$ and all $t\geq 0$, 
\begin{align*}
T^{\alpha}_t(g)(x) = e^{td} \int_{\bbr^d} g \left(e^{t} x + (1-e^{-\alpha t})^{\frac{1}{\alpha}}e^{t}z\right) \mu_\alpha(dz). 
\end{align*}
Moreover, from Fourier inversion, for all $g \in \mathcal{S}(\bbr^d)$, all $x \in \bbr^d$ and all $t\geq 0$, 
\begin{align}\label{eq:Fourier_Rep_T}
T^{\alpha}_t(g)(x) = e^{td} \int_{\bbr^d} \mathcal{F}(g)(\xi) e^{i \langle \xi ; x e^t \rangle} \frac{\varphi_\alpha(e^t \xi)}{\varphi_\alpha(\xi)} \frac{d\xi}{(2\pi)^d}.
\end{align}
In particular, the Fourier transform of $T^\alpha_t(g)$ is given, for all $\xi \in \bbr^d$ and all $t\geq 0$, by
\begin{align*}
\mathcal{F}(T^\alpha_t(g))(\xi) = \mathcal{F}(g)(e^{-t} \xi) \frac{\varphi_\alpha(\xi)}{\varphi_\alpha(e^{-t}\xi)}.
\end{align*}
Then, thanks to Lemma \ref{lem:int_rep_dual_semi}, for all $g \in \mathcal{S}(\bbr^d)$ and all $t>0$, 
\begin{align}\label{eq:rep_dual_semigroup}
(P^{\nu_\alpha}_t)^*(g) = ((M_\alpha)^{-1} \circ T^{\alpha}_t \circ M_\alpha)(g).
\end{align}
The semigroup of operators $((P^{\nu_{\alpha}}_t)^*)_{t\geq 0}$ is the $h$-transform of the semigroup $(T_t^{\alpha})_{t \geq 0}$ by the positive function $p_\alpha$ (see, e.g., \cite[Section $1.15.8$]{BGL14}) which is harmonic for the generator of $(T_t^{\alpha})_{t \geq 0}$. 
The next technical lemma gathers standard properties of the continuous family of operators $(T_t^\alpha)_{t\geq 0}$.~First, define the following bilinear form which appears as a remainder in the product rule for the non-local operator $D^{\alpha-1}$: for all $f,g \in \mathcal{S}(\bbr^d)$ and all $x \in \bbr^d$, 
\begin{align}\label{eq:remainder_stable}
R^{\alpha}(f,g)(x) = \int_{\bbr^d} (f(x+u)-f(x))(g(x+u)-g(x))u \nu_\alpha(du). 
\end{align}
In particular, this remainder term is null when $\alpha=2$ since the classical product rule holds in this diffusive situation.~Finally, for all $x \in \bbr^d$, 
\begin{align}\label{eq:rule_palpha}
D^{\alpha-1}(p_\alpha)(x) = -x p_\alpha(x) , \quad  (D^{\alpha-1})^*(p_\alpha)(x) = x p_\alpha(x).
\end{align}
The next lemma states and proves many rather elementary properties of the family of operators $(T_t^{\alpha})_{t \geq 0}$. 

\begin{lem}\label{lem:prop_T_t}
For all $f\in \mathcal{C}_b(\bbr^d)$ and all $s,t \geq 0$, 
\begin{align*}
T_{s+t}^{\alpha}(f) = (T_t^{\alpha} \circ T_s^\alpha)(f) = (T_s^\alpha \circ T_t^{\alpha})(f). 
\end{align*}
For all $f \in \mathcal{S}(\bbr^d)$, 
\begin{align*}
\underset{t \rightarrow +\infty}{\lim} T_t^{\alpha}(f)(x) = M_\alpha \left( \int_{\bbr^d} f(x) dx \right) , \quad \underset{t \rightarrow 0^+}{\lim} T_t^{\alpha}(f)(x) = f(x). 
\end{align*}
For all $f \in \mathcal{C}_b(\bbr^d)$ and all $t\geq 0$,
\begin{align*}
\int_{\bbr^d} T_t^\alpha(p_\alpha f)(x)dx = \int_{\bbr^d} p_\alpha(x)f(x)dx.
\end{align*}
For all $f \in \mathcal{C}_b(\bbr^d)$ with $f \geq 0$ and all $t\geq 0$,
\begin{align*}
T_t^{\alpha}(f) \geq 0, \quad T_t^{\alpha}(1) = e^{td}, \quad T_t^{\alpha}(p_\alpha) = p_\alpha. 
\end{align*}
For all $f,g \in \mathcal{S}(\mathbb{R}^d)$ and all $t \geq 0$, 
\begin{align*}
\int_{\bbr^d} T_t^\alpha(f)(x)g(x)dx = \int_{\bbr^d} f(x) P^{\nu_\alpha}_t(g)(x)dx. 
\end{align*}
Namely, for all $t>0$, the dual operator of $T_t^\alpha$ in standard Lebesgue spaces is given, for all $f \in  \mathcal{S}(\bbr^d)$, by  
\begin{align*}
(T_t^\alpha)^*(f) = P^{\nu_\alpha}_t(f). 
\end{align*}
The generator $A_\alpha$ of $(T_t^\alpha)_{t \geq 0}$ is given, for all $f \in \mathcal{S}(\bbr^d)$ and all $x \in \bbr^d$,  by
\begin{align*}
A_\alpha(f)(x) & = d f(x) + \langle x; \nabla(f)(x) \rangle + \int_{\bbr^d} \langle \nabla(f)(x+u)- \nabla(f)(x) ; u \rangle \nu_\alpha(du).
\end{align*} 
For all $x \in \bbr^d$, 
\begin{align*}
A_\alpha(p_\alpha)(x) = 0. 
\end{align*}
The ``carr\'e du champs operator " associated with $A_\alpha$ is given, for all $f \in \mathcal{S}(\bbr^d)$ and all $x\in \bbr^d$, by
\begin{align*} 
\Gamma_\alpha(f,g)(x) = -  \frac{d f(x)g(x)}{2} + \frac{\alpha}{2} \int_{\bbr^d} (f(x+u)-f(x))(g(x+u)-g(x)) \nu_\alpha(du).
\end{align*}
For all $f \in \mathcal{S}(\bbr^d)$ and all $x \in \bbr^d$, 
\begin{align}\label{eq:formula_gen}
A_\alpha(p_\alpha f)(x) & = \sum_{k=1}^d \left(R_k^{\alpha}(\partial_k(p_\alpha),f)(x)+R_k^{\alpha}(\partial_k(f),p_\alpha)(x)\right) \nonumber \\ 
& \quad\quad + \sum_{k=1}^d \left(\partial_k(p_\alpha)(x) D_k^{\alpha-1}(f)(x) + p_\alpha(x) \partial_kD^{\alpha-1}_k(f)(x)\right). 
\end{align}
For all $g\in \mathcal{C}^1_b(\bbr^d)$, all $u \in \bbr^d$, all $x \in \bbr^d$ and all $t \geq 0$,
\begin{align*}
\Delta_{u} \left(T_t^{\alpha}(g)\right)(x) = e^{td} \int_{\bbr^d} \Delta_{ue^{t}}(g)(xe^t + (1-e^{-\alpha t})^{\frac{1}{\alpha}}e^t z) \mu_\alpha(dz),
\end{align*}
and, 
\begin{align*}
\nabla_{\nu_\alpha}(T^{\alpha}_t(g))(x) \leq e^{\frac{\alpha t}{2}} T^{\alpha}_{t} \left(\nabla_{\nu_\alpha}(g)\right)(x).
\end{align*}
In particular, for all $f \in \mathcal{S}(\bbr^d)$ and all $x \in \bbr^d$, 
\begin{align}\label{eq:Adjoint_Stable_OU}
(\mathcal{L}_\alpha)^*(f)(x) & = ((M_\alpha)^{-1} \circ A_\alpha \circ M_\alpha)(f)(x) ,  \nonumber\\
& = \frac{1}{p_\alpha(x)}\sum_{k=1}^d \left(R_k^{\alpha}(\partial_k(p_\alpha),f)(x)+R_k^{\alpha}(\partial_k(f),p_\alpha)(x)\right) \nonumber\\ 
& \quad\quad + \sum_{k=1}^d \left(\frac{\partial_k(p_\alpha)(x)}{p_\alpha(x)} D_k^{\alpha-1}(f)(x) + \partial_kD^{\alpha-1}_k(f)(x)\right). 
\end{align}
Finally, the ``carr\'e du champs" operator associated with the generator $(\mathcal{L}_\alpha)^*$ is given, for all $f,g\in \mathcal{S}(\bbr^d)$ and all $x \in \bbr^d$, by
\begin{align*}
\Gamma^*(f,g)(x)& = \frac{\alpha}{2} \int_{\bbr^d} (f(x+u)-f(x))(g(x+u)-g(x)) \nu_\alpha(du) \\
& \quad\quad + \frac{1}{2}\sum_{k=1}^d \frac{\partial_k(p_\alpha)(x)}{p_\alpha(x)}R_k^\alpha(f,g)(x)\\
& \quad\quad +\frac{1}{2p_\alpha(x)} \sum_{k=1}^d \left(\partial_k R_k^{\alpha}(p_\alpha,fg)(x)-g(x)\partial_k R_k^{\alpha}(p_\alpha,f)(x)-f(x)\partial_k R_k^{\alpha}(p_\alpha,g)(x)\right).
\end{align*}
\end{lem}

\begin{proof}
The proof is very classical and based on a characteristic function methodology and on the Fourier representation \eqref{eq:Fourier_Rep_T}.~The only non-trivial identity is given by \eqref{eq:formula_gen}. So, for all $f \in \mathcal{S}(\bbr^d)$ and all $x \in \bbr^d$, 
\begin{align*}
A_\alpha(p_\alpha f)(x) = d p_\alpha(x) f(x) + \langle x ; p_\alpha(x) \nabla(f)(x) \rangle + \langle x ; f(x) \nabla(p_\alpha)(x) \rangle + \sum_{k=1}^d \partial_k D_k^{\alpha-1}(p_\alpha f)(x).  
\end{align*}
Moreover, 
\begin{align*}
D_k^{\alpha-1}(p_\alpha f)(x) = p_\alpha(x) D_k^{\alpha-1}(f)(x) + f(x) D_k^{\alpha-1}(p_\alpha)(x) + R_k^\alpha(p_\alpha,f)(x). 
\end{align*}
Thus, 
\begin{align*}
\partial_k D_k^{\alpha-1}(p_\alpha f)(x) = A + B + C,
\end{align*}
where,
\begin{align*}
A = \partial_k R_k^\alpha(p_\alpha,f)(x), \quad B = \partial_k \left(p_\alpha(x) D_k^{\alpha-1}(f)\right)(x), \quad C = \partial_k \left(f(x) D_k^{\alpha-1}(p_\alpha)\right)(x).
\end{align*}
Now, using the classical product rule,
\begin{align*}
A = R_k^{\alpha} \left(\partial_k(p_\alpha) , f\right)(x) + R_k^{\alpha} \left(\partial_k(f) , p_\alpha\right)(x),
\end{align*}
and, 
\begin{align*}
B = \partial_k(p_\alpha)(x) D_{k}^{\alpha-1}(f)(x) + p_\alpha(x) \partial_k D_k^{\alpha-1}(f)(x).
\end{align*}
Finally, using \eqref{eq:rule_palpha}, 
\begin{align*}
C = -x_k p_\alpha(x) \partial_k(f)(x) + f(x) \left(- p _\alpha(x) - x_k \partial_k(p_\alpha)(x)\right),
\end{align*}
and putting everything together concludes the proof of \eqref{eq:formula_gen}. 
\end{proof}
\noindent
From the previous lemma and the decomposition \eqref{eq:rep_dual_semigroup}, it is clear, by duality, that the linear operator $(P^{\nu_\alpha}_t)^*$ is continuous on every $L^p(\mu_\alpha)$, for $p \in (1, +\infty)$. Indeed, for all $f,g \in \mathcal{S}(\bbr^d)$ and all $t \geq 0$, 
\begin{align*}
\langle (P^{\nu_\alpha}_t)^*(g);f \rangle_{L^2(\mu_\alpha)} & = \langle T_t^\alpha(M_\alpha(g)) ; f \rangle_{L^2(\bbr^d,dx)}  \\ 
& = \langle M_\alpha(g) ; P_t^{\nu_\alpha}(f)\rangle_{L^2(\bbr^d,dx)} = \langle g ;P_t^{\nu_\alpha}(f) \rangle_{L^2(\mu_\alpha)}. 
\end{align*} 
Moreover, based on the last statements of Lemma \ref{lem:prop_T_t}, one can infer the corresponding formulas for the generator of the ``carr\'e de Mehler" semigroup on $\mathcal{S}(\bbr^d)$ and for its corresponding square field operator: for all $f,g \in \mathcal{S}(\bbr^d)$ and all $x \in \bbr^d$, 
\begin{align*}
\mathcal{L}(f)(x) & = \frac{1}{\alpha} \bigg( -\langle x ; \nabla(f)(x)\rangle + \sum_{k=1}^d \partial_k D^{\alpha-1}_{k}(f)(x)+\frac{1}{p_\alpha(x)}\sum_{k=1}^d\partial_k R_k^\alpha(p_\alpha,f)(x)\\
&\quad\quad +\sum_{k=1}^d\left(\frac{\partial_k(p_\alpha)(x)}{p_\alpha(x)}D_k^{\alpha-1}(f)(x)+\partial_kD_k^{\alpha-1}(f)(x)\right)\bigg),
\end{align*}
and, 
\begin{align*}
\tilde{\Gamma}(f,g)(x) = \frac{1}{\alpha} \left(\Gamma(f,g)(x) + \Gamma^*(f,g)(x)\right). 
\end{align*}
Let us now prove two Bismut-type formulas associated with $(P^{\nu_\alpha}_t)_{t \geq 0}$ and $((P^{\nu_\alpha}_t)^*)_{t \geq 0}$ for integro-differential operators appearing in the generators of the respective semigroups. 

\begin{prop}\label{prop:Bismut_2}
Let $d \geq 1$, let $\alpha \in (1,2)$, let $\mu_\alpha$ be a nondegenerate symmetric $\alpha$-stable probability measure on $\bbr^d$ and let $p_\alpha$ be its positive Lebesgue density.~Then, for all $f \in \mathcal{S}(\bbr^d)$, all $x \in \bbr^d$ and all $t>0$,
\begin{align}\label{eq:Bismut_Gradient_OU}
\nabla P_t^{\nu_\alpha}(f)(x) = - \dfrac{e^{-t}}{(1-e^{-\alpha t})^{\frac{1}{\alpha}}} \int_{\bbr^d} \dfrac{\nabla(p_\alpha)(y)}{p_\alpha(y)} f(xe^{-t} + (1-e^{-\alpha t})^{\frac{1}{\alpha}}y) \mu_\alpha(dy),
\end{align}
and,
\begin{align}\label{eq:Bismut_Fract_DUO}
D^{\alpha-1}\left((P_t^{\nu_\alpha})^*(f)\right)(x) + \frac{1}{p_\alpha(x)} R^\alpha \left(p_\alpha, (P_t^{\nu_\alpha})^*(f) \right)(x) &= \dfrac{- x e^{-\alpha t}}{(1-e^{-\alpha t})}  (P^{\nu_\alpha}_t)^{*}(f)(x) \nonumber \\
&\quad\quad + \dfrac{e^{-t}}{\left(1-e^{-\alpha t}\right)} (P^{\nu_\alpha}_t)^{*}(xf)(x),
\end{align}
for all $x \in \bbr^d$. 
\end{prop}

\begin{proof}
The identity \eqref{eq:Bismut_Gradient_OU} is a direct consequence of the commutation relation and of a standard integration by parts.~Let us prove \eqref{eq:Bismut_Fract_DUO}.~For this purpose, for all $x \in \bbr^d$ and all $t>0$ fixed, denote by $F_{\alpha,x,t}$ the function defined, for all $u \in \bbr^d$, by
\begin{align*}
F_{\alpha,x,t}(u) = \frac{1}{\left(1- e^{-\alpha t}\right)^{\frac{d}{\alpha}}} \frac{p_\alpha(u)}{p_\alpha(x)} p_\alpha\left( \dfrac{x-ue^{-t}}{(1- e^{-\alpha t})^{\frac{1}{\alpha}}} \right).
 \end{align*}
Note that $F_{\alpha,x,t}$ is a probability density on $\bbr^d$.~Moreover,  for all $t>0$ and all $u \in \bbr^d$, 
\begin{align*}
\int_{\bbr^d} F_{\alpha,x,t}(u)  p_\alpha(x)dx = p_\alpha(u).
\end{align*}
First, for all $x \in \bbr^d$, all $u \in \bbr^d$ and all $t>0$,
\begin{align*}
\Delta_u\left( (P_t^{\nu_\alpha})^*(f) \right)(x) & = \int_{\bbr^d} f(v) (F_{\alpha, x+u, t}(v) - F_{\alpha, x, t}(v))dv , \\
&= \int_{\bbr^d} f(v) \Delta_u(F_{\alpha, ., t}(v))(x)dv.
\end{align*}
Thus, 
\begin{align*}
D^{\alpha-1}\left((P_t^{\nu_\alpha})^*(f)\right)(x) = \int_{\bbr^d} f(v) \left(\int_{\bbr^d}\Delta_u(F_{\alpha, ., t}(v))(x) u \nu_\alpha(du)\right) dv. 
\end{align*}
Similarly, by linearity, 
\begin{align*}
\frac{1}{p_\alpha(x)} R^\alpha \left(p_\alpha, (P_t^{\nu_\alpha})^*(f) \right)(x) & = \int_{\bbr^d} u \nu_\alpha(du) \left(\frac{p_\alpha(x+u)}{p_\alpha(x)}-1\right)\Delta_u\left( (P_t^{\nu_\alpha})^*(f) \right)(x)  , \\
& = \int_{\bbr^d} f(v) \left(\int_{\bbr^d} u \nu_\alpha(du) \left(\frac{p_\alpha(x+u)}{p_\alpha(x)}-1\right)\Delta_u(F_{\alpha, ., t}(v))(x)\right) dv.
\end{align*}
Then,  for all $t>0$ and all $x \in \bbr^d$, 
\begin{align}\label{eq:bismut_type}
D^{\alpha-1}\left((P_t^{\nu_\alpha})^*(f)\right)(x) + \frac{1}{p_\alpha(x)} R^\alpha \left(p_\alpha, (P_t^{\nu_\alpha})^*(f) \right)(x) = \int_{\bbr^d} f(v) \left(\int_{\bbr^d} u \nu_\alpha(du)\frac{p_\alpha(x+u)}{p_\alpha(x)}\Delta_u(F_{\alpha, ., t}(v))(x)\right) dv.
\end{align}
Let us fix $x,v \in \bbr^d$ and $t>0$. Then, 
\begin{align*}
\int_{\bbr^d} u \nu_\alpha(du)\frac{p_\alpha(x+u)}{p_\alpha(x)}\Delta_u(F_{\alpha, ., t}(v))(x)& = \dfrac{1}{(1-e^{-\alpha t})^{\frac{d}{\alpha}}} \int_{\bbr^d} u \nu_\alpha(du) \dfrac{p_\alpha(x+u)}{p_\alpha(x)} \bigg(\frac{p_\alpha(v)}{p_\alpha(x+u)} p_\alpha\left( \dfrac{x+u-ve^{-t}}{(1- e^{-\alpha t})^{\frac{1}{\alpha}}} \right)\\
&\quad\quad - \frac{p_\alpha(v)}{p_\alpha(x)} p_\alpha\left( \dfrac{x-ve^{-t}}{(1- e^{-\alpha t})^{\frac{1}{\alpha}}} \right)\bigg) , \\ 
& = \dfrac{p_\alpha(v)}{p_{\alpha}(x)^2} \dfrac{1}{(1-e^{-\alpha t})^{\frac{d}{\alpha}}} \int_{\bbr^d} u \nu_\alpha(du)\bigg( p_\alpha(x) p_\alpha\left( \dfrac{x+u-ve^{-t}}{(1- e^{-\alpha t})^{\frac{1}{\alpha}}} \right) \\
&\quad\quad -p_\alpha(x+u) p_\alpha\left( \dfrac{x-ve^{-t}}{(1- e^{-\alpha t})^{\frac{1}{\alpha}}} \right) \bigg) , \\
& =  \dfrac{p_\alpha(v)}{p_{\alpha}(x)^2} \dfrac{1}{(1-e^{-\alpha t})^{\frac{d}{\alpha}}} \int_{\bbr^d} u \nu_\alpha(du) p_\alpha(x) \bigg(p_\alpha\left( \dfrac{x+u-ve^{-t}}{(1- e^{-\alpha t})^{\frac{1}{\alpha}}}\right) \\
&\quad\quad - p_\alpha\left( \dfrac{x-ve^{-t}}{(1- e^{-\alpha t})^{\frac{1}{\alpha}}}\right) \bigg) -  \dfrac{p_\alpha(v)}{p_{\alpha}(x)^2} \dfrac{1}{(1-e^{-\alpha t})^{\frac{d}{\alpha}}}\\
& \quad\quad \times  \int_{\bbr^d} u \nu_\alpha(du) (p_\alpha(x+u) - p_\alpha(x))p_\alpha\left( \dfrac{x-ve^{-t}}{(1- e^{-\alpha t})^{\frac{1}{\alpha}}}\right).
\end{align*}
Recalling that, for all $x \in \bbr^d$, 
\begin{align*}
D^{\alpha-1}(p_\alpha)(x) = -x p_\alpha(x).
\end{align*}
Thus, 
\begin{align*}
\dfrac{p_\alpha(v)}{p_{\alpha}(x)^2} \dfrac{1}{(1-e^{-\alpha t})^{\frac{d}{\alpha}}} D^{\alpha-1}(p_\alpha)(x)p_\alpha\left( \dfrac{x-ve^{-t}}{(1- e^{-\alpha t})^{\frac{1}{\alpha}}}\right) = \dfrac{p_\alpha(v)}{p_{\alpha}(x)} \dfrac{(-x)}{(1-e^{-\alpha t})^{\frac{d}{\alpha}}} p_\alpha\left( \dfrac{x-ve^{-t}}{(1- e^{-\alpha t})^{\frac{1}{\alpha}}}\right),
\end{align*}
and, from scale invariance, 
\begin{align*}
\dfrac{p_\alpha(v)}{p_{\alpha}(x)} \dfrac{1}{(1-e^{-\alpha t})^{\frac{d}{\alpha}}} \int_{\bbr^d} u\nu_\alpha(du) \Delta_{\frac{u}{(1-e^{-\alpha t})^{\frac{1}{\alpha}}}} \left(p_\alpha\right)\left( \dfrac{x - ve^{-t}}{(1-e^{-\alpha t})^{\frac{1}{\alpha}}} \right)&  = \dfrac{p_\alpha(v)}{p_{\alpha}(x)} \dfrac{(1-e^{-\alpha t})^{\frac{1}{\alpha}-1}}{(1-e^{-\alpha t})^{\frac{d}{\alpha}}} \\
&\quad\quad \times D^{\alpha-1}(p_\alpha)\left(\dfrac{x - ve^{-t}}{(1-e^{-\alpha t})^{\frac{1}{\alpha}}}\right) , \\
&  = - \dfrac{p_\alpha(v)}{p_{\alpha}(x)} \dfrac{(1-e^{-\alpha t})^{-1}}{(1-e^{-\alpha t})^{\frac{d}{\alpha}}} \left(x - v e^{-t}\right) \\
&\quad \quad \times p _\alpha \left(\dfrac{x - ve^{-t}}{(1-e^{-\alpha t})^{\frac{1}{\alpha}}}\right).
\end{align*}
Then, using \eqref{eq:bismut_type}, 
\begin{align*}
D^{\alpha-1}\left((P_t^{\nu_\alpha})^*(f)\right)(x) + \frac{1}{p_\alpha(x)} R^\alpha \left(p_\alpha, (P_t^{\nu_\alpha})^*(f) \right)(x) &= \dfrac{- x e^{-\alpha t}}{(1-e^{-\alpha t})}  (P^{\nu_\alpha}_t)^{*}(f)(x) \\
& \quad \quad + \dfrac{e^{-t}}{\left(1-e^{-\alpha t}\right)} (P^{\nu_\alpha}_t)^{*}(hf)(x),
\end{align*}
where $h(v)=v$, for all $v \in \bbr^d$. This concludes the proof of the proposition.
\end{proof}
\noindent
Before moving on, let us prove a technical lemma providing a sharp upper bound for the asymptotic behavior of 
\begin{align*}
\frac{1}{p_\alpha(x)}R^{\alpha}(p_\alpha , f)(x), 
\end{align*}
as $\|x\| \rightarrow +\infty$, with $f \in \mathcal{C}_c^{\infty}(\bbr^d)$, and when the associated L\'evy measure on $\bbr^d$ is given by $\nu_\alpha(du) = du / \|u\|^{\alpha+d}$. 

\begin{lem}\label{lem:upper_bound_remainder}
Let $d \geq 1$, let $\alpha \in (1,2)$ and let $\nu_\alpha (du) = du/\|u\|^{\alpha+d}$.~Let $R^\alpha$ be given by \eqref{eq:remainder_stable} and let $p_\alpha$ be the positive Lebesgue density of the nondegenerate symmetric $\alpha$-stable probability measure $\mu_\alpha$ with L\'evy measure $\nu_\alpha$.~Then, for all $f \in \mathcal{C}_c^{\infty}(\bbr^d)$ and all $x$ large enough, 
\begin{align*}
\left\|\frac{1}{p_\alpha(x)}R^{\alpha}(p_\alpha , f)(x)\right\| \leq C \left(1 + \|x\| \right), 
\end{align*}
for some positive constant $C$ depending on $\alpha$, on $d$ and on $f$. 
\end{lem}

\begin{proof}
Without loss of generality, assume that $f \in \mathcal{C}_c^{\infty}(\bbr^d)$ is a bump function:~i.e., $\operatorname{Supp}(f) \subset \mathcal{B}(0,1)$ and $f(x) \in [0,1]$, for all $x \in \bbr^d$.~Then, for all $x \in \bbr^d$ such that $\|x\| \geq 3$, 
\begin{align*}
\frac{1}{p_\alpha(x)}R^{\alpha}(p_\alpha , f)(x)& = \frac{1}{p_\alpha(x)} \int_{\bbr^d} u \nu_\alpha(du) \left(p_\alpha(x+u) - p_\alpha(x)\right) f(x+u) , \\
& = \frac{1}{p_\alpha(x)} \int_{\mathcal{B}(0,1)} (u-x) \frac{du}{\|u - x\|^{\alpha+d}} \left(p_\alpha(u) - p_\alpha(x)\right) f(u).
\end{align*}
Thus,  since $\|x\| \geq 3$, 
\begin{align*}
\left\| \frac{1}{p_\alpha(x)}R^{\alpha}(p_\alpha , f)(x) \right\| & \leq \frac{C}{\|x\|^{\alpha + d}} \frac{1}{p_\alpha(x)} \int_{\mathcal{B}(0,1)} \|u -x \| du \left|p_\alpha(u) - p_\alpha(x)\right| |f(u)|.
\end{align*}
Moreover, for all $x \in \bbr^d$, 
\begin{align}
\frac{C_1}{\left(1+ \|x\|\right)^{\alpha+ d}} \leq p_\alpha(x) \leq \frac{C_2}{\left(1+ \|x\|\right)^{\alpha+ d}},
\end{align}
for some $C_1,C_2 >0$ two positive constants depending on $\alpha$ and on $d$.~Thus, for all $\|x\| \geq 3$, 
\begin{align*}
\left\| \frac{1}{p_\alpha(x)}R^{\alpha}(p_\alpha , f)(x) \right\| &\leq C_{\alpha,d,f} \dfrac{(1+ \|x\|)^{\alpha+d}}{\|x\|^{\alpha+d}} \left(1+ \|x\|\right), \\
& \leq C_{\alpha,d,f}  \left(1+ \|x\|\right).
\end{align*}
This concludes the proof of the lemma.
\end{proof}
\noindent
Next, let us investigate pseudo-Poincar\'e inequality (see, e.g., \cite{Ledoux_03} and the references therein) for the dual semigroup $((P_t^{\nu_\alpha})^*_{t \geq 0})$ in $L^p(\mu_\alpha)$, for all $ p \in (1, \alpha)$.  To start, let $(R^\alpha)^*$ be defined, for all $f,g \in \mathcal{S}(\bbr^d)$ and all $x \in \bbr^d$, by
\begin{align*}
(R^{\alpha})^*(g , f)(x) = \int_{\bbr^d} (g(x-u) -g(x))(f(x-u)- f(x)) u \nu_\alpha(du). 
\end{align*}

\begin{prop}\label{prop:pseudo_Poincare_DSG}
Let $d \geq 1$, let $\alpha \in (1,2)$, let $\mu_\alpha$ be a nondegenerate symmetric $\alpha$-stable probability measure on $\bbr^d$, and let $p_\alpha$ be its positive Lebesgue density.  Further,  assume that,
\begin{align}\label{eq:logarithmic_derivative_p}
\left \| \dfrac{\nabla(p_\alpha)}{p_\alpha}  \right\|_{L^p(\mu_\alpha)} <+\infty, \quad p \in (1, +\infty),
\end{align}
and that,  for all $p \in (1, \alpha)$ and all $f \in \mathcal{C}_c^{\infty}(\bbr^d)$, 
\begin{align}
\left\|  \frac{1}{p_\alpha}(R^{\alpha})^*(p_\alpha , f)  \right\|_{L^p(\mu_\alpha)} <+\infty.
\end{align}
Then, for all $p \in (1, \alpha)$, all $f \in \mathcal{C}_c^{\infty}(\bbr^d)$ and all $t >0$, 
\begin{align}\label{eq:pseudo_poinca_lp}
\| (P_t^{\nu_\alpha})^*(f) -f \|_{L^p(\mu_\alpha)} \leq C_\alpha t^{1 - \frac{1}{\alpha}} \left\| (D^{\alpha-1})^{*}(f)  + \frac{1}{p_\alpha}(R^{\alpha})^*(p_\alpha , f) \right\|_{L^p(\mu_\alpha)} \left \| \dfrac{\nabla(p_\alpha)}{p_\alpha}  \right\|_{L^p(\mu_\alpha)}, 
\end{align}
for some $C_\alpha >0$ depending only on $\alpha$. 
\end{prop}

\begin{proof}
The argument is based on duality and on \eqref{eq:Bismut_Gradient_OU}.~Let $f, g \in \mathcal{C}_c^{\infty}(\bbr^d)$, let $p \in (1, \alpha)$ and let $p^* = p/(p-1)$. Then, by standard semigroup arguments, 
\begin{align*}
\langle (P^{\nu_\alpha}_t)^*(f)-f;g \rangle_{L^2(\mu_\alpha)} & = \int_0^t \langle f ; \mathcal{L}^\alpha P^{\nu_\alpha}_s(g)  \rangle_{L^2(\mu_\alpha)} ds , \\
& = - \int_{0}^t \langle xf ; \nabla P_s^{\nu_\alpha}(g) \rangle_{L^2(\mu_\alpha)} ds + \int_{0}^t \langle f ; \nabla. D^{\alpha - 1} P_s^{\nu_\alpha}(g) \rangle_{L^2(\mu_\alpha)} ds , \\
& = \int_{0}^t \left\langle -x f + \dfrac{(D^{\alpha-1})^*(p_\alpha f)}{p_\alpha} ; \nabla P_s^{\nu_\alpha}(g) \right\rangle_{L^2(\mu_\alpha)} ds.
\end{align*}
First, thanks to \eqref{eq:Bismut_Gradient_OU}, for all $s \in (0,t]$ and all $x \in \bbr^d$, 
\begin{align*}
\nabla P_s^{\nu_\alpha}(g)(x) = - \dfrac{e^{-s}}{\left(1- e^{- \alpha s}\right)^{\frac{1}{\alpha}}} \int_{\bbr^d} \dfrac{\nabla(p_\alpha)(y)}{p_\alpha(y)} g \left(xe^{-s} + (1 - e^{-s \alpha})^{\frac{1}{\alpha}}y\right) \mu_\alpha(dy). 
\end{align*}
Moreover, for all $x \in \bbr^d$, 
\begin{align*}
\dfrac{(D^{\alpha-1})^*(p_\alpha f)(x)}{p_\alpha(x)} = (D^{\alpha-1})^{*}(f)(x) + xf(x) + \frac{1}{p_\alpha(x)}(R^{\alpha})^*(p_\alpha , f)(x),
\end{align*}
where, 
\begin{align*}
(R^{\alpha})^*(p_\alpha , f)(x) = \int_{\bbr^d} (p_\alpha(x-u) -p_\alpha(x))(f(x-u)- f(x)) u \nu_\alpha(du). 
\end{align*}
Thus, 
\begin{align*}
\langle (P^{\nu_\alpha}_t)^*(f)-f;g \rangle_{L^2(\mu_\alpha)} & = \int_{0}^t \left\langle (D^{\alpha-1})^{*}(f) + \frac{1}{p_\alpha}(R^{\alpha})^*(p_\alpha , f); \nabla P_s^{\nu_\alpha}(g) \right\rangle_{L^2(\mu_\alpha)} ds.
\end{align*}
Now, for all $s \in (0,t]$
\begin{align*}
\left\langle (D^{\alpha-1})^{*}(f)  + \frac{1}{p_\alpha}(R^{\alpha})^*(p_\alpha , f); \nabla P_s^{\nu_\alpha}(g) \right\rangle_{L^2(\mu_\alpha)} & =  - \dfrac{e^{-s}}{\left(1-e^{- \alpha s}\right)^{\frac{1}{\alpha}}} \int_{\bbr^d} \int_{\bbr^d} \langle(D^{\alpha-1})^{*}(f)(x) \\
&\quad \quad + \frac{1}{p_\alpha(x)}(R^{\alpha})^*(p_\alpha , f)(x) ; \dfrac{\nabla(p_\alpha)(y)}{p_\alpha(y)} \rangle \\
& \quad \quad \times g\bigg(xe^{-s} + (1-e^{-\alpha s})^{\frac{1}{\alpha}}y\bigg)\mu_\alpha(dx) \mu_\alpha(dy).
\end{align*}
By H\"older's inequality, 
\begin{align*}
\left|  \left\langle (D^{\alpha-1})^{*}(f)  + \frac{1}{p_\alpha}(R^{\alpha})^*(p_\alpha , f); \nabla P_s^{\nu_\alpha}(g) \right\rangle_{L^2(\mu_\alpha)} \right|& \leq \dfrac{e^{-s}}{(1-e^{-\alpha s})^{\frac{1}{\alpha}}} \left\| \dfrac{\nabla(p_\alpha)}{p_\alpha} \right\|_{L^p(\mu_\alpha)} \|g\|_{L^{p^*}(\mu_\alpha)}  \\
&\quad \quad  \times \left\| (D^{\alpha-1})^{*}(f)  + \frac{1}{p_\alpha}(R^{\alpha})^*(p_\alpha , f) \right\|_{L^p(\mu_\alpha)}.
\end{align*}
Standard arguments allow to conclude the proof of the proposition.
\end{proof}
\noindent
The inequality \eqref{eq:pseudo_poinca_lp} is a straightforward generalization of the Gaussian  pseudo-Poincar\'e inequality. Before moving on, let us discuss the condition \eqref{eq:logarithmic_derivative_p}.~In the rotationally invariant case, recall the following classical pointwise bounds: for all $x \in \bbr^d$, 
\begin{align}
\frac{C_2}{\left(1+ \|x\|\right)^{\alpha+ d}} \leq p^{\operatorname{rot}}_\alpha(x) \leq \frac{C_1}{\left(1+ \|x\|\right)^{\alpha+ d}},
\end{align}
for some $C_1, C_2$ positive constants. Moreover, (see, e.g, \cite{Chen_Zhang_16}), for all $x \in \bbr^d$, 
\begin{align*}
\| \nabla(p^{\operatorname{rot}}_\alpha)(x) \| \leq \dfrac{C_3}{\left(1+ \|x\|\right)^{\alpha + d+1}},
\end{align*}
for some positive constant $C_3$, so that the logarithmic derivative of $p^{\operatorname{rot}}_\alpha$ is uniformly bounded on $\bbr^d$ and so belongs to $L^p(\mu_\alpha)$, for all $p \geq 1$.~Another interesting case is when the coordinates are independent and distributed according to the same symmetric $\alpha$-stable law on $\bbr$ with $\alpha \in (1,2)$. It is straightforward to check that, in this case, the logarithmic derivative is uniformly bounded on $\bbr^d$.

\begin{rem}\label{rem:Lp_Poin_Stable}
Let us end the $\alpha$-stable case, $\alpha \in (1,2)$, with a discussion regarding $L^p$-Poincar\'e inequalities, for $p \geq 2$.  Classically, by formal semigroup arguments, 
\begin{align*}
\|f\|^p_{L^p(\mu_\alpha)} = \bbe f(X_\alpha) g(X_\alpha) = -  \int_0^{+\infty} \bbe (\mathcal{L}_\alpha)^* (P_t^{\nu_\alpha})^*(f)(X_\alpha) g(X_\alpha) dt, 
\end{align*}
with $f \in \mathcal{C}_c^{\infty}(\bbr^d)$ such that $\mu_\alpha(f) = 0$, with $p \geq 2$ and with $g(x)  = \operatorname{sign}(f(x)) |f(x)|^{p-1}$.~Moreover, using standard integration by parts and \eqref{eq:Bismut_Fract_DUO},  for all $t>0$, 
\begin{align*}
\bbe (\mathcal{L}_\alpha)^* (P_t^{\nu_\alpha})^*(f)(X_\alpha) g(X_\alpha) & = - \int_{\bbr^d} \bigg\langle \dfrac{- x e^{-\alpha t}}{(1-e^{-\alpha t})}  (P^{\nu_\alpha}_t)^{*}(f)(x) \\
&\quad\quad+ \dfrac{e^{-t}}{\left(1-e^{-\alpha t}\right)} (P^{\nu_\alpha}_t)^{*}(hf)(x)  ;  \nabla(g)(x)\bigg\rangle \mu_\alpha(dx).
\end{align*}
Now, based on Proposition \ref{prop:pseudo_Poincare_DSG} and on the fact that $p \geq 2$,  it does not seem possible to reproduce the semigroup proof of the $L^p$-Poincar\'e inequality 
presented in the Gaussian case. Indeed, the bad concentration properties of the $\alpha$-stable probability measures, with $\alpha \in (1,2)$, as well as the occurence of the remainder terms $R^\alpha$ and $(R^\alpha)^*$ prohibit the use of H\"older's inequality followed by the  Cauchy-Schwarz inequality.  
\end{rem}
\noindent
Very recently, moment estimates for heavy-tailed probability measures on $\bbr^d$ of Cauchy-type have been obtained in \cite[Corollary 4.3.]{AdPolSt20} (see, also $(4.2)$ and $(4.3)$ and the discussion above these) based on weighted Beckner-type inequalities. Note that the right hand side of these inequalities put into play weigthed norms of the classical gradient operator. Let us observe that it is possible to obtain these weighted Poincar\'e inequalities from the non-local ones in some cases, as shown in the next proposition. 

\begin{prop}\label{prop:exp}
Let $\mu$ be the standard exponential probability measure on $(0,+\infty)$ and let $\nu$ be the associated L\'evy measure on $(0, +\infty)$. 
Then, for all $f \in \mathcal{S}(\bbr)$, 
\begin{align*}
\int_{(0,+\infty)} \int_{(0,+\infty)} |f(x+u) - f(x)|^2 \nu(du) \mu(dx) \leq \int_{(0,+\infty)} w |f'(w)|^2 \mu(dw).
\end{align*}
\end{prop}

\begin{proof}
First, by Jensen's inequality,
\begin{align*}
|f(x+u)-f(x)|^2 \leq u^2 \int_0^1 |f'(x+tu)|^2 dt. 
\end{align*}
Thus, since $\nu(du)/du = e^{-u}/u$, $u>0$, 
\begin{align*}
\int_{(0,+\infty)^2}|f(x+u)-f(x)|^2 \nu(du) \mu(dx) \leq \int_{(0,+\infty)^2} u\left( \int_0^1 |f'(x+tu)|^2 dt\right) e^{-u} e^{-x} dx du. 
\end{align*}
Now, for all $t \in (0,1)$, let $\mathcal{D}_t = \{(w,z) \in (0,+\infty)^2 :\, w > tz\}$ and let $\Phi_t$ be the $\mathcal{C}^1$-diffeomorphism from $(0,+\infty)^2$ to $\mathcal{D}_t$ defined, for all $(x,u) \in (0,+\infty)^2$, by
\begin{align*}
\Phi_t(x,u) = (x+tu,u). 
\end{align*}
Thus, by the change of variables with $\Phi_t$,
\begin{align*}
\int_{(0,+\infty)^2}|f(x+u)-f(x)|^2 \nu(du) \mu(dx) & \leq \int_{(0,+\infty)^2} z\left( \int_0^1 |f'(w)|^2 dt\right) \bbone_{\mathcal{D}_t}(w,z)e^{-z} e^{-(w-tz)} dw dz , \\
& \leq \int_0^{+\infty} \int_0^1  |f'(w)|^2 e^{-w} \left(\int_0^{+\infty} z e^{-z} e^{tz} \bbone_{\mathcal{D}_t}(w,z) dz\right) dt dw.
\end{align*} 
Now,
\begin{align*}
\int_0^1 \left( \int_0^{\frac{w}{t}} z e^{-z} e^{z t} dz \right) dt & = \int_0^1 \left(\int_0^w \frac{y}{t}e^{- \frac{y}{t}} e^{y} \frac{dy}{t}\right) dt , \\
& = \int_0^w y e^{y} \left(\int_0^1 e^{-\frac{y}{t}} \frac{dt}{t^2} \right) dy = \int_0^w y e^{y} \frac{e^{-y}}{y} dy = w.
\end{align*}
This concludes the proof of the lemma. 
\end{proof}
\noindent

\section{Stein's Kernels and High Dimensional CLTs}
\noindent
This section shows how to apply \cite[Theorem 5.10.]{AH20_3}~or Theorem \ref{thm:cov_rep_dirichlet_form} to build Stein's kernels to provide stability result for Poincar\'e-type inequality and rates of convergence, in $1$-Wasserstein distance, in high dimensional central limit theorem.~Let $d \geq 1$ and let $\Sigma$ be a covariance matrix, which is not identically null, and let $\gamma_\Sigma$ be the centered Gaussian probability measure on $\bbr^d$ with covariance matrix, $\Sigma$, i.e.,  the characteristic function of the corresponding Gaussian random vector is given, for all $\xi \in \bbr^d$, by
\begin{align*}
\hat{\gamma}_{\Sigma}(\xi) = \exp\left(- \dfrac{\langle \xi ; \Sigma(\xi) \rangle}{2}\right). 
\end{align*}
Next, let $U_\Sigma$ be the Poincar\'e functional formally defined, for all suitable $\mu \in \mathcal{M}_1(\bbr^d)$ ($\mathcal{M}_1(\bbr^d)$ is the set of probability measures on $\bbr^d$), by 
\begin{align*}
U_\Sigma(\mu) := \underset{f \in \mathcal{H}_\Sigma(\mu)}{\sup} \dfrac{\operatorname{Var}_\mu(f)}{\int_{\bbr^d} \langle \nabla(f)(x) ; \Sigma\left(\nabla(f)(x)\right) \rangle \mu(dx)},
\end{align*}
where $\mathcal{H}_\Sigma(\mu)$ is the set of Borel measurable real-valued functions $f$ defined on $\bbr^d$ such that 
\begin{align*}
\int_{\bbr^d} \left| f(x) \right|^2 \mu(dx) < +\infty, \quad 0<\int_{\bbr^d} \langle \nabla(f)(x) ; \Sigma\left(\nabla(f)(x)\right) \rangle \mu(dx) < +\infty,
\end{align*}
and such that $\operatorname{Var}_\mu(f)>0$. It is well-known since the works \cite[Theorem 3]{BU_84} and \cite[Theorem $2.1$]{ChLo87} that the functional $U_\Sigma$ is rigid.  Let us adopt the methodology developed in \cite{CFP19,Fathi_SM,AH19_2,AH20_3} using Stein's method to obtain a stability result which generalizes the one for the isotropic case.  For the sake of completeness, the rigidity result is re-proved next via semigroup methods, although the result  is rather immediate 
from \eqref{eq:covariance_representation_gaussian}.

\begin{lem}\label{lem:poincare_gaussian}
Let $d \geq 1$ and let $\Sigma$ be a, not identically null,  $d \times d$ covariance matrix.~Then, 
\begin{align*}
U_\Sigma(\gamma_\Sigma) = 1. 
\end{align*}
\end{lem}

\begin{proof}
The proof is very classical and relies on a semigroup argument to prove the Poincar\'e inequality for the Gaussian probability measure $\gamma_\Sigma$ and on the fact that the functions $x \mapsto x_j$, for all $j \in \{1, \dots, d\}$, are eigenfunctions of the Ornstein-Uhlenbeck operator associated with $\gamma_\Sigma$.  Let $(P_t^{\Sigma})_{t \geq 0}$ be the Ornstein-Uhlenbeck semigroup given, for all $f \in \mathcal{C}_b(\bbr^d)$, all $t \geq 0$ and all $x \in \bbr^d$,  by
\begin{align*}
P_t^{\Sigma}(f)(x) = \int_{\bbr^d} f \left(x e^{-t} + \sqrt{1-e^{-2t}}y\right) \gamma_{\Sigma}(dy). 
\end{align*}
From the above Mehler formula, it is clear that the probability measure $\gamma_\Sigma$ is an invariant measure for the semigroup $(P_t^{\Sigma})_{t \geq 0}$, that $\mathcal{S}(\bbr^d)$ is a core for the generator, denoted by $\mathcal{L}^{\Sigma}$, of $(P_t^{\Sigma})_{t \geq 0}$ and, that for all $f \in \mathcal{S}(\bbr^d)$ and all $x \in \bbr^d$, 
\begin{align*}
\mathcal{L}^{\Sigma}(f)(x) = - \langle x ; \nabla(f)(x) \rangle + \Delta^{\Sigma}(f)(x),
\end{align*}
with,
\begin{align*}
\Delta^{\Sigma}(f)(x)  = \frac{1}{(2\pi)^d} \int_{\bbr^d} \mathcal{F}(f)(\xi) e^{i \langle x ; \xi \rangle}  \langle i\xi ; \Sigma(i\xi) \rangle  d\xi = \langle \Sigma ; \operatorname{Hess}(f)(x) \rangle_{HS},
\end{align*}
where $\langle A ; B \rangle_{HS} = \operatorname{Tr}(A^tB)$.  Next, let $f \in \mathcal{S}(\bbr^d)$ be such that $\int_{\bbr^d} f(x) \gamma_{\Sigma}(dx) = 0$.  Differentiating the variance of $P^{\Sigma}_t(f)$ with respect to the time parameter gives
\begin{align*}
\dfrac{d}{dt} \left( \bbe P^{\Sigma}_t(f)(X)^2 \right)  = 2 \bbe P^{\Sigma}_t(f)(X)\mathcal{L}^{\Sigma} P^{\Sigma}_t(f)(X),
\end{align*}
where $X \sim \gamma_{\Sigma}$. Hence, for all $t \geq 0$, 
\begin{align*}
\dfrac{d}{dt} \left( \bbe P^{\Sigma}_t(f)(X)^2 \right) = 2 \bbe P^{\Sigma}_t(f)(X) \left( -\langle X ; \nabla(P^{\Sigma}_t(f))(X) \rangle + \langle \Sigma ; \operatorname{Hess} \left(  P^{\Sigma}_t(f) \right)(X) \rangle_{HS} \right). 
\end{align*}
Now,  since $\mathcal{S}(\bbr^d)$ is a core for $\mathcal{L}^{\Sigma}$, invariant with respect to $P_t^{\Sigma}$, for all $t \geq 0$, and stable for the pointwise multiplication of functions, 
\begin{align*}
\bbe \langle X  ; \nabla\left(P^{\Sigma}_t(f)^2 \right)(X)  \rangle  = \bbe \langle \Sigma ; \operatorname{Hess} \left( P_t^{\Sigma}(f)^2\right)(X) \rangle_{HS}.
\end{align*}
Thus, by Leibniz formula, for all $t \geq 0$, 
\begin{align*}
\dfrac{d}{dt} \left( \bbe P^{\Sigma}_t(f)(X)^2 \right) & = - \left( \bbe \langle \Sigma ; \operatorname{Hess} \left( P_t^{\Sigma}(f)^2\right)(X) \rangle_{HS}  - 2 \bbe P^{\Sigma}_t(f)(X) \langle \Sigma ; \operatorname{Hess} \left(  P^{\Sigma}_t(f) \right)(X)\rangle_{HS} \right) , \\
& = - \bbe \langle \Sigma ; \operatorname{Hess} \left(P_t^{\Sigma}(f)^2\right)(X) - 2 P^{\Sigma}_t(f)(X) \operatorname{Hess} \left(  P^{\Sigma}_t(f) \right)(X)  \rangle_{HS},  \\
& = - 2 \bbe \langle \nabla\left( P_t^\Sigma(f)\right)(X) ; \Sigma \left(\nabla\left( P_t^\Sigma(f)\right)(X)\right) \rangle.
\end{align*}
Now, the commutation formula, $\nabla \left(P^\Sigma_t(f) \right) = e^{-t} P_t^{\Sigma}(\nabla(f))$, ensures that
\begin{align*}
\bbe \langle \nabla\left( P_t^\Sigma(f)\right)(X) ; \Sigma \left(\nabla\left( P_t^\Sigma(f)\right)(X)\right) \rangle & = e^{-2t} \bbe \langle P_t^{\Sigma}\left( \nabla(f) \right)(X) ; \Sigma P_t^{\Sigma} \left(\nabla(f)\right)(X) \rangle , \\
& =e^{-2t} \bbe \langle \sqrt{\Sigma} P_t^{\Sigma}\left( \nabla(f) \right)(X) ; \sqrt{\Sigma} P_t^{\Sigma} \left(\nabla(f)\right)(X) \rangle , \\
& =e^{-2t} \bbe \langle  P_t^{\Sigma}\left( \sqrt{\Sigma}(\nabla(f)) \right)(X) ;  P_t^{\Sigma} \left(\sqrt{\Sigma}(\nabla(f))\right)(X) \rangle , \\
& = e^{-2t} \bbe \left\|  P_t^{\Sigma}\left( \sqrt{\Sigma}(\nabla(f)) \right)(X) \right\|^2 , \\
& \leq e^{-2t} \bbe P_t^{\Sigma} \left(\left\| \sqrt{\Sigma}(\nabla(f))\right\|^2\right)(X) , \\
& \leq e^{-2t} \bbe \left\| \sqrt{\Sigma}(\nabla(f))(X)\right\|^2.
\end{align*}
Thus, for all $t \geq 0$, 
\begin{align*}
\dfrac{d}{dt} \left( \bbe P^{\Sigma}_t(f)(X)^2 \right)  \geq -2 e^{-2t} \bbe \left\| \sqrt{\Sigma}(\nabla(f))(X)\right\|^2.
\end{align*}
Integrating with respect to $t$ between $0$ and $+\infty$ ensures that
\begin{align}\label{ineq:Poincare_type_anisotropic_gauss_measure}
\bbe f(X)^2 \leq \bbe \left\| \sqrt{\Sigma} (\nabla(f))(X) \right\|^2. 
\end{align}
This last inequality implies that $U_\Sigma(\gamma_{\Sigma}) \leq 1$.~Next, for all $j \in \{1 , \dots,  d\}$,  let $g_j$ be the function defined, for all $x \in \bbr^d$, by $g_j(x) = x_j$.  Now,  for all $j \in \{1, \dots, d\}$,
\begin{align*}
\operatorname{Var}_{\gamma_\Sigma}(g_j) = \int_{\bbr^d} x_j^2 \gamma_{\Sigma}(dx) = \sigma_{j,j} = \int_{\bbr^d} \langle \nabla(g_j)(x) ; \Sigma(\nabla(g_j))(x) \rangle \gamma_{\Sigma}(dx),
\end{align*}
where $\Sigma = (\sigma_{i,j})_{1\leq i,j \leq d}$.~Thus, $U_\Sigma(\gamma_{\Sigma}) \geq 1$. This concludes the proof of the lemma.
\end{proof}

\begin{rem}\label{rem:what_about_the_degenerate_case}
The proof of the Poincar\'e-type inequality \eqref{ineq:Poincare_type_anisotropic_gauss_measure} for the probability measure $\gamma_\Sigma$ 
could have been performed without using the semigroup $(P_t^{\Sigma})_{t \geq 0}$. 
Instead, one could use the covariance representation \eqref{eq:covariance_representation_gaussian}.~Indeed, taking $g = f$ and using the  Cauchy-Schwarz inequality, one retrieves the inequality \eqref{ineq:Poincare_type_anisotropic_gauss_measure}.  Following the end of the proof of Lemma \ref{lem:poincare_gaussian}, one can conclude that $U_\Sigma (\gamma_\Sigma) = 1$ also when $\Sigma$ is generic but different of $0$. 
\end{rem}
\noindent
The next lemma provides the rigidity result for $U_{\Sigma}$. 

\begin{lem}\label{lem:rigidity}
Let $d \geq 1$ and let $\Sigma = (\sigma_{i,j})_{1\leq i,j \leq d}$ be a, not identically null, $d \times d$ covariance matrix.~Let $\mu$ be a probability measure on $\bbr^d$ with finite second moment such that, for all $i \in \{1, \dots, d\}$,
\begin{align*}
\int_{\bbr^d} x \mu(dx) = 0 , \quad \int_{\bbr^d} x_i^2 \mu(dx) = \sigma_{i,i}.
\end{align*}
Then,  $U_{\Sigma}(\mu) = 1$ if and only if $\mu= \gamma_{\Sigma}$. 
\end{lem}

\begin{proof}
The sufficiency is a direct consequence of Lemma \ref{lem:poincare_gaussian} or of Remark \ref{rem:what_about_the_degenerate_case}.~Thus, let us prove the direct implication.  Assume that $U_{\Sigma}(\mu) = 1$.~Then, for all $f \in \mathcal{H}_\Sigma(\mu)$, 
\begin{align}\label{ineq:Poincar}
\operatorname{Var}_{\mu}(f) \leq \int_{\bbr^d} \Gamma_{\Sigma}(f,f)(x) \mu(dx) = : \mathcal{E}_{\Sigma}(f,f), 
\end{align}
with, for all $x \in \bbr^d$,
\begin{align*}
\Gamma_{\Sigma}(f,f)(x) = \langle \nabla(f)(x); \Sigma (\nabla(f)(x)) \rangle = \left\|\sqrt{\Sigma} (\nabla(f)(x))\right\|^2.
\end{align*}
Now,  for all $j \in \{1, \dots, d\}$ and all $\varepsilon \in \bbr$ with $\varepsilon\ne 0$, let $f_j$ be defined, for all $x \in \bbr^d$, by 
\begin{align*}
f_j(x) = g_j(x) + \varepsilon f(x),
\end{align*}
for some $f \in \mathcal{S}(\bbr^d)$.~Then, for all $j \in \{1, \dots, d\}$, 
\begin{align*}
\operatorname{Var}_{\mu}(f_j) = \operatorname{Cov}_{\mu}(g_j + \varepsilon f,g_j + \varepsilon f) = \operatorname{Var}_{\mu}(g_j) + 2 \varepsilon   \operatorname{Cov}_{\mu}(g_j,f) + \varepsilon^2 \operatorname{Var}_{\mu}(f), 
\end{align*}
and, 
\begin{align*}
\mathcal{E}_{\Sigma}(f_j,f_j) = \mathcal{E}_{\Sigma}(g_j,g_j) + 2 \varepsilon \mathcal{E}_{\Sigma}(g_j,f) + \varepsilon^2 \mathcal{E}_{\Sigma}(f,f).
\end{align*}
(here and in the sequel, $\operatorname{Cov}_{\mu}(f,g)$ indicates the covariance of $f$ and of $g$ under $\mu$).~Thus, thanks to \eqref{ineq:Poincar}, for all $j \in \{1, \dots, d\}$, 
\begin{align*} 
\operatorname{Cov}_{\mu}(g_j,f) = \mathcal{E}_{\Sigma}(g_j,f).
\end{align*}
Namely, for all $j \in \{1, \dots, d\}$ and all $f \in \mathcal{S}(\bbr^d)$,
\begin{align*}
\bbe X_j f(X) = \bbe \langle \Sigma(e_j) ; \nabla(f)(X) \rangle,  \quad X \sim \mu,
\end{align*}
with $e_j  = (0, \dots, 0, 1, 0,\dots,0)^T$. The end of the proof follows easily by a standard argument involving the characteristic function. Indeed, by Fourier inversion and duality, for all 
$j \in \{1, \dots, d\}$ and all $\xi \in \bbr^d$,
\begin{align*}
\partial_{\xi_j} \left( \varphi_\mu \right)(\xi) = - \langle \Sigma(e_j) ; \xi \rangle \varphi_\mu(\xi), 
\end{align*}
where $\varphi_\mu$ is the characteristic function of $\mu$ which is $\mathcal{C}^1$ on $\bbr^d$ since $\mu$ has finite second moment.  Then,  for all $\xi \in \bbr^d$, 
\begin{align*}
\langle \xi ; \nabla(\varphi_\mu)(\xi) \rangle = - \langle \xi ; \Sigma(\xi) \rangle \varphi_\mu(\xi).
\end{align*}
Passing to spherical coordinates, for all $(r, \theta) \in (0, +\infty) \times \mathbb{S}^{d-1}$, 
\begin{align*}
\partial_r \left(\varphi_\mu \right)(r\theta) = -r \langle \theta ; \Sigma(\theta) \rangle \varphi_\mu(r \theta).  
\end{align*}
Fixing $\theta \in \mathbb{S}^{d-1}$,  integrating with respect to $r$ and using $\varphi_\mu(0) = 1$,  for all $(r, \theta) \in (0, +\infty) \times \mathbb{S}^{d-1} $, 
\begin{align*}
\varphi_\mu(r\theta) = \exp \left(- \frac{r^2}{2} \langle \theta ; \Sigma(\theta) \rangle\right).
\end{align*}
This concludes the proof of the lemma. 
\end{proof}
\noindent
Before moving to the proof of the stability result, let us recall some well-known facts about  Stein's method for the multivariate Gaussian probability measure $\gamma_{\Sigma}$ on $\bbr^d$. The standard references are \cite{Stein2,Bar90,Go1991,GR96,RR96,Rai04,ChM08,RR09,Mec09,NPR10,Shih11,NP12,R13,BC20,NPY21}.~In the sequel, let $h \in \mathcal{C}^{\infty}_c(\bbr^d)$ be such that
\begin{align*}
\|h\|_{\operatorname{Lip}} := \underset{x,y \in \bbr^d,\,  x \ne y}{ \sup} \dfrac{|h(x) - h(y)|}{\|x -y\|} = \underset{x \in \bbr^d}{\sup} \|\nabla(h)(x)\| \leq 1, 
\end{align*}
and let $f_h$ be defined, for all $x \in \bbr^d$, by
\begin{align}\label{eq:solution_stein_equation}
f_h(x) = - \int_{0}^{+\infty} \left(P^{\Sigma}_t(h)(x) - \bbe h(X)\right) dt, \quad X \sim \gamma_\Sigma. 
\end{align}
The next lemma recalls regularity results for $f_h$ as well as a representation formula for its Hessian matrix which allows to obtain dimension free bounds for the supremum norms involving the operator or the Hilbert-Schmidt norms of $\operatorname{Hess}(f_h)$.  

\begin{lem}\label{lem:Stein_Factors_anisotropic_Gaussian}
Let $d \geq 1$ and let $\Sigma$ be a nondegenerate $d \times d$ covariance matrix.~Let $h \in \mathcal{C}^{\infty}_c(\bbr^d)$ be such that $\|h\|_{\operatorname{Lip}} \leq 1$ and let $f_h$ be given by \eqref{eq:solution_stein_equation}.~Then, $f_h$ is well-defined, twice continuously differentiable on $\bbr^d$, and 
\begin{align*}
\underset{x \in \bbr^d}{\sup} \|\nabla(f_h)(x)\| \leq 1, \quad \underset{x \in \bbr^d}{\sup} \|\operatorname{Hess}(f_h)(x)\|_{op} \leq  \sqrt{\frac{2}{\pi}} \|\Sigma^{-\frac{1}{2}}\|_{op},\quad  \underset{x \in \bbr^d}{\sup} \|\operatorname{Hess}(f_h)(x)\|_{HS} \leq \|\Sigma^{-\frac{1}{2}}\|_{op}.
\end{align*}
Moreover, if $h \in \mathcal{C}^{\infty}_c(\bbr^d)$ is such that 
\begin{align*}
\|h\|_{\operatorname{Lip}} \leq 1, \quad \underset{x \in \bbr^d}{\sup}\|\operatorname{Hess}(h)(x)\|_{op} \leq 1,
\end{align*}
then, 
\begin{align*}
& \underset{x \in \bbr^d}{\sup} \|\nabla(f_h)(x)\| \leq 1, \quad \underset{x \in \bbr^d}{\sup} \|\operatorname{Hess}(f_h)(x)\|_{op} \leq  \frac{1}{2}.
\end{align*}
Finally,  if
$h \in \mathcal{C}^{\infty}_c(\bbr^d)$ is such that 
\begin{align*}
\|h\|_{\operatorname{Lip}} \leq 1, \quad \tilde{M}_2(h) := \underset{x \in \bbr^d}{\sup}\|\operatorname{Hess}(h)(x)\|_{HS} \leq 1,
\end{align*}
then, 
\begin{align*}
& \underset{x \in \bbr^d}{\sup} \|\nabla(f_h)(x)\| \leq 1, \quad \underset{x \in \bbr^d}{\sup} \|\operatorname{Hess}(f_h)(x)\|_{HS} \leq \frac{1}{2}.
\end{align*}
\end{lem}

\begin{proof}
Thanks to the Mehler formula, for all $t \geq 0$ and all $x \in \bbr^d$, 
\begin{align*}
\left| P^{\Sigma}_t(h)(x) - \bbe h(X)\right| & \leq \int_{\bbr^d} \left| h\left(xe^{-t} + \sqrt{1-e^{-2t}} y\right) - h(y) \right| \gamma_\Sigma(dy),\\
& \leq \|h\|_{\operatorname{Lip}} \left(e^{-t} \|x\| + \left| 1 - \sqrt{1-e^{-2t}}\right| \int_{\bbr^d} \|y\| \gamma_{\Sigma}(dy) \right). 
\end{align*}
The right-hand side of the previous inequality is clearly integrable, with respect to $t$, on $(0,+\infty)$. Thus, $f_h$ is well-defined on $\bbr^d$. The fact that $f_h$ is twice continuously differentiable on $\bbr^d$ follows from the commutation formula $\nabla P_t^{\Sigma}(h) = e^{-t} P_t^{\Sigma}(\nabla(h))$.  Now, for all $x \in \bbr^d$, 
\begin{align*}
\nabla(f_h)(x) = - \int_0^{+\infty} e^{-t} P_t^{\Sigma}(\nabla(h))(x) dt.
\end{align*}
Thus, for all $u \in \bbr^d$ such that $\|u\| = 1$, 
\begin{align*}
\langle \nabla(f_h)(x) ; u \rangle  =  - \int_0^{+\infty} e^{-t} P_t^{\Sigma}(\langle u ;  \nabla(h) \rangle)(x) dt. 
\end{align*}
Then, for all $x \in \bbr^d$ and all $u \in \bbr^d$ such that $\|u\| = 1$,
\begin{align*}
\left|\langle \nabla(f_h)(x) ; u \rangle\right| \leq \left(\int_0^{+\infty} e^{-t}dt \right) \|h\|_{\operatorname{Lip}} \leq 1. 
\end{align*}
Next, let us deal with the Hessian matrix of $f_h$.  For all $k,\ell \in \{1, \dots, d\}$ and all $x \in \bbr^d$, 
\begin{align*}
\partial_k \partial_\ell \left(f_h\right)(x) = - \int_0^{+\infty} e^{-2t} P^{\Sigma}_t(\partial_k \partial_\ell(h))(x)dt. 
\end{align*}
Moreover, thanks to Bismut's formula,  for all $k,\ell \in \{1, \dots, d\}$, all $x \in \bbr^d$ and all $t \geq 0$, 
\begin{align*}
P^{\Sigma}_t(\partial_k \partial_\ell(h))(x) & = \int_{\bbr^d} \partial_k \partial_\ell(h)\left(x e^{-t} + \sqrt{1 -e^{-2t}} y\right) \gamma_{\Sigma}(dy) \\
& = \dfrac{1}{\sqrt{1-e^{-2t}}} \int_{\bbr^d} \langle \Sigma^{-1}(e_{\ell}) ; y \rangle \partial_{k}(h)\left(x e^{-t} + \sqrt{1 -e^{-2t}} y\right) \gamma_{\Sigma}(dy), 
\end{align*}
and so, for all $\ell, k \in \{1, \dots, d\}$ and all $x \in \bbr^d$, 
\begin{align*}
\partial_k \partial_\ell \left(f_h\right)(x) = - \int_0^{+\infty} \dfrac{e^{-2t}}{\sqrt{1-e^{-2t}}} \left(\int_{\bbr^d} \langle \Sigma^{-1}(e_{\ell}) ; y \rangle \partial_{k}(h)\left(x e^{-t} + \sqrt{1 -e^{-2t}} y\right) \gamma_{\Sigma}(dy) \right) dt .
\end{align*}
Thus, for all $x \in \bbr^d$, 
\begin{align}\label{eq:rep_Hessian_Bismut}
\operatorname{Hess}(f_h)(x) = - \int_{0}^{+\infty} \dfrac{e^{-2t}}{\sqrt{1-e^{-2t}}} \int_{\bbr^d} \Sigma^{-1}(y) \left(\nabla(h) \left(x e^{-t} + \sqrt{1 -e^{-2t}} y\right) \right)^T \gamma_{\Sigma}(dy)  dt.
\end{align}
Now, let $u,v \in \bbr^d$ be such that $\|u\| = \|v\| = 1$. Then, for all $x \in \bbr^d$, 
\begin{align*}
\left| \langle \operatorname{Hess}(f_h)(x)u ; v \rangle \right| & = \left| \int_0^{+\infty} \dfrac{e^{-2t}}{\sqrt{1-e^{-2t}}} \left( \int_{\bbr^d} \langle \Sigma^{-1}(y) ; v \rangle \langle \nabla(h)\left( xe^{-t} + \sqrt{1-e^{-2t}} y \right) ; u \rangle \gamma_{\Sigma}(dy) \right) dt \right| , \\
& \leq \left( \int_{0}^{+\infty} \dfrac{e^{-2t}}{\sqrt{1-e^{-2t}}} dt \right) \int_{\bbr^d} \left| \langle \Sigma^{-1}(y)  ; v \rangle \right| \gamma_{\Sigma}(dy) , \\
& \leq \int_{\bbr^d} \left| \langle \Sigma^{-\frac{1}{2}}(y)  ; v \rangle \right| \gamma(dy), \\
& \leq \|\Sigma^{-\frac{1}{2}}\|_{op} \left(\int_{\bbr}|x| e^{-\frac{x^2}{2}} \frac{dx}{\sqrt{2\pi}} \right) , \\
& \leq \sqrt{\frac{2}{\pi}} \|\Sigma^{-\frac{1}{2}}\|_{op},
\end{align*}
where $\gamma$ is the standard Gaussian measure on $\bbr^d$ (i.e., with covariance matrix given by $I_d$) and since, under $\gamma$,  for all $u \in \mathbb{S}^{d-1}$, $\langle \Sigma^{-\frac{1}{2}}(y) ; u \rangle$ is a centered normal random variable with variance $\|\Sigma^{-\frac{1}{2}}(u)\|^2$. It remains to estimate the Hilbert-Schmidt norm of $\operatorname{Hess}(f_h)(x)$ based on \eqref{eq:rep_Hessian_Bismut}.  By a similar argument and using H\"older's inequality for Schatten norms,  for all $x \in \bbr^d$,
\begin{align*}
\| \operatorname{Hess}(f_h)(x)\|_{HS} \leq \underset{A \in \mathcal{M}_{d \times d}(\bbr), \, \|A\|_{HS}=1}{\sup} \|A \Sigma^{-\frac{1}{2}}\|_{HS} \leq \|\Sigma^{-\frac{1}{2}}\|_{op},
\end{align*}
where $\mathcal{M}_{d \times d}(\bbr)$ denotes the set of $d\times d$ matrices with real coefficients.
\end{proof}

\begin{rem}\label{rem:bound}
(i) To the best of our knowledge, the bound,
\begin{align*}
 \underset{x \in \bbr^d}{\sup} \|\operatorname{Hess}(f_h)(x)\|_{HS} \leq \|\Sigma^{-\frac{1}{2}}\|_{op},
\end{align*}
for $h \in \mathcal{C}^1(\bbr^d)$ with $\| h \|_{\operatorname{Lip}} \leq 1$, is the best available in the literature.~It generalizes the bound obtained in \cite[Lemma $2.2$]{ChM08} for the isotropic case, amends \cite[Proof of Lemma $2$]{Mec09} and improves on the bound obtained in \cite[Inequality $(13)$]{NPY21} (see also \cite[Lemma $3.3$]{NPR10}).\\
(ii) Assuming that $d=2$ and that $\Sigma = I_2$, let us compute the quantity 
\begin{align*}
J(A):=\int_{\bbr^2} \|A(y)\| \gamma(dy),
\end{align*}
for some specific values of $A$, a $2 \times 2$ matrix with $\|A\|_{HS} = 1$.~Take,  for instance $A = I_2/\sqrt{2}$.~Then, by standard computations using polar coordinates, 
\begin{align*}
J \left(\frac{1}{\sqrt{2}} I_2 \right)= \frac{1}{\sqrt{2}} \int_{\bbr^2} \|y\| \exp\left(- \frac{ \|y\|^2}{2}\right) \frac{dy}{2\pi} = \frac{\sqrt{\pi}}{2},
\end{align*}
which seems to question the bound obtained in \cite[Proof of Lemma $2$, page $161$]{Mec09}.
\end{rem}

\noindent
Next, let us discuss the notion and the existence of Stein's kernels with respect to the Gaussian probability measure $\gamma_{\Sigma}$.  The idea is the following: let $\mu$ be a probability measure on $\bbr^d$ with finite second moment such that $\int_{\bbr^d} x  \mu(dx) = 0$ and such that $\operatorname{Cov}(X_\mu,X_\mu) = \Sigma$, where $X_\mu \sim \mu$. Moreover, assume that there exists $\tau_\mu$, a function defined on $\bbr^d$ with values in $\mathcal{M}_{d \times d}(\bbr)$, such that, for all appropriate vector-valued 
functions $f$ defined on $\bbr^d$, 
\begin{align}\label{eq:Stein_Kernel_multi}
\int_{\bbr^d} \langle  \tau_\mu(x) ;  \nabla(f)(x)  \rangle_{HS} \mu(dx) = \int_{\bbr^d} \langle x ; f(x) \rangle \mu(dx).  
\end{align}
In the vector-valued case, $\nabla(f)$ denotes the Jacobian matrix of $f$.  Then, the classical argument for bounding distances goes as follows:  let $h \in \mathcal{C}^{1}(\bbr^d)$ be such that $\|h\|_{\operatorname{Lip}} \leq 1$, and let $f_h$ be given by \eqref{eq:solution_stein_equation}. (Actually, in finite dimension, one can take, without loss of generality, $h\in \mathcal{C}_c^{\infty}(\bbr^d)$; the main point being that $\|h\|_{\operatorname{Lip}} \leq 1$.) Then,  $f_h$ is a strong solution to the following partial differential equation: for all $x \in \bbr^d$, 
\begin{align*}
-\langle x ; \nabla\left(f_h\right)(x) \rangle + \langle \Sigma ; \operatorname{Hess}\left(f_h\right)(x) \rangle_{HS}  = h(x) - \bbe h(X), \quad X \sim \gamma_{\Sigma}. 
\end{align*}
Integrating with respect to $\mu$ and using the formal definition of $\tau_\mu$ give
\begin{align*}
\left| \bbe h(X_\mu)  - \bbe h(X) \right|  & =  \left| \bbe \left( -\langle X_\mu ; \nabla\left(f_h\right)(X_\mu) \rangle + \langle \Sigma ; \operatorname{Hess}\left(f_h\right)(X_\mu) \rangle_{HS} \right) \right| ,  \\
&  =\left| \bbe \left(  \langle \Sigma - \tau_\mu(X_\mu) ; \operatorname{Hess}\left(f_h\right)(X_\mu) \rangle_{HS} \right) \right|,
\end{align*}
with $X_\mu \sim \mu$. Then, by the Cauchy-Schwarz inequality and the bound obtained in Lemma \ref{lem:Stein_Factors_anisotropic_Gaussian}, 
\begin{align}\label{ineq:Stein_bound}
\left| \bbe h(X_\mu)  - \bbe h(X) \right|  \leq  \|\Sigma^{-\frac{1}{2}}\|_{op} \left(\bbe\left(  \|\tau_\mu(X_\mu) - \Sigma \|^2_{HS} \right)\right)^{\frac{1}{2}}.
\end{align}
Observe that the right-hand side of the previous inequality does not depend on $h$ anymore.~In the sequel, let us explain how to prove the existence of $\tau_\mu$ and how to bound the Stein discrepancy based on closed forms techniques. For this purpose, let us consider the following bilinear symmetric non-negative definite form defined, for all $f,g \in \mathcal{C}_c^{\infty}(\bbr^d, \bbr^d)$, by
\begin{align*}
\mathcal{E}_{\Sigma, \mu} (f,g) = \int_{\bbr^d} \langle \Sigma \left(\nabla(g)(x)\right) ; \nabla(f)(x) \rangle_{HS} \mu(dx),
\end{align*}
where $\Sigma$ is a nondegenerate covariance matrix and where $\mu$ is a probability measure on $\bbr^d$ with finite second moment such that 
\begin{align*}
\int_{\bbr^d} x \mu(dx) = 0 , \quad \int_{\bbr^d} x x^T \mu(dx) = \Sigma. 
\end{align*}

\begin{prop}\label{prop:existence_stein_kernel}
Let $d \geq 1$ and let $\Sigma$ be a nondegenerate $d \times d$ covariance matrix.~Let $\mu$ be a probability measure on $\bbr^d$ with finite second moment such that 
\begin{align*}
\int_{\bbr^d} x \mu(dx) = 0 , \quad \int_{\bbr^d} x x^T \mu(dx) = \Sigma. 
\end{align*}
Let the form $\left( \mathcal{E}_{\Sigma, \mu} , \mathcal{C}_c^{\infty}(\bbr^d, \bbr^d) \right)$ be closable.~Finally, let there exists $U_{\Sigma,\mu}>0$ such that for all $f \in \mathcal{C}^{\infty}_c(\bbr^d, \bbr^d)$ with $\int_{\bbr^d} f(x) \mu(dx) = 0$, 
\begin{align}\label{ineq:Poincare_type_inequality}
\int_{\bbr^d} \|f(x)\|^2 \mu(dx) \leq U_{\Sigma,\mu} \int_{\bbr^d} \langle \Sigma \left(\nabla(f)(x)\right) ; \nabla(f)(x) \rangle_{HS} \mu(dx). 
\end{align} 
Then, there exists $\tau_\mu$ such that, for all $f \in \mathcal{D}(\mathcal{E}_{\Sigma, \mu})$, 
\begin{align*}
\int_{\bbr^d} \langle x ; f(x) \rangle \mu(dx) = \int_{\bbr^d} \langle \nabla(f)(x) ; \tau_\mu(x) \rangle_{HS} \mu(dx).
\end{align*}
Moreover, 
\begin{align*}
\int_{\bbr^d} \| \tau_\mu(x)\|^2_{HS} \mu(dx) \leq U_{\Sigma,\mu} \left\| \Sigma \right\|^2_{HS}.
\end{align*}
\end{prop}

\begin{proof}
First, let us build the Stein's kernel $\tau_\mu$. Since the form $\left( \mathcal{E}_{\Sigma, \mu} , \mathcal{C}_c^{\infty}(\bbr^d, \bbr^d) \right)$ is closable, consider its smallest closed extension denoted by $\left(\mathcal{E}_{\Sigma, \mu} , \mathcal{D}(\mathcal{E}_{\Sigma, \mu})\right)$,  where $\mathcal{D}(\mathcal{E}_{\Sigma, \mu})$ is its dense linear domain.~Moreover, let $\mathcal{L}^{\mu}$, $(G^{\mu}_{\delta})_{\delta >0}$ and $(P^\mu_t)_{t \geq 0}$ be the corresponding generator, the strongly continuous resolvent and the strongly continuous semigroup.~In particular, recall that, for all $f \in L^2(\bbr^d, \bbr^d, \mu)$ and all $\delta >0 $, 
\begin{align*}
G^{\mu}_\delta(f)  = \int_0^{+\infty} e^{- \delta t} P^{\mu}_t(f)dt. 
\end{align*}
Next, let $f \in \mathcal{C}^{\infty}_c(\bbr^d, \bbr^d)$ be such that $\int_{\bbr^d} f(x) \mu(dx) = 0$. Then, by the very definition of the generator $\mathcal{L}^{\mu}$ and an integration by parts, 
\begin{align*}
\dfrac{d}{dt} \left( \bbe \|P^\mu_t(f)(X_\mu)\|^2\right) & = 2 \bbe \langle P^{\mu}_t(f)(X_\mu) ; \mathcal{L}^{\mu} P^{\mu}_t(f)(X_\mu) \rangle = -2 \mathcal{E}_{\Sigma, \mu} \left(P^{\mu}_t(f),P^{\mu}_t(f)\right) \\
&  \leq -\frac{2}{U_{\Sigma,\mu}} \bbe \|P^\mu_t(f)(X_\mu)\|^2,
\end{align*}
with $X_\mu \sim \mu$. Then, for all $f \in \mathcal{C}^{\infty}_c(\bbr^d,\bbr^d)$ such that $\int_{\bbr^d} f(x) \mu(dx) = 0 $, 
\begin{align*}
\left\|  P^{\mu}_t(f) \right\|^2_{L^2(\bbr^d, \bbr^d, \mu)} \leq \exp \left(- \dfrac{2 t}{U_{\Sigma,\mu}}\right) \left\| f \right\|^2_{L^2(\bbr^d, \bbr^d, \mu)},
\end{align*}
which, via a density argument, clearly extends to all $f \in L^2(\bbr^d, \bbr^d, \mu)$ with $\int_{\bbr^d} f(x) \mu(dx) = 0$.~Then, from \cite[Theorem $5.10$]{AH20_3}, for all $g \in L^2(\bbr^d, \bbr^d, \mu)$ such that $\int_{\bbr^d} g(x) \mu(dx) = 0$ and all $f \in \mathcal{D}(\mathcal{E}_{\Sigma, \mu})$, 
\begin{align}\label{eq:zero_potential_equation}
\mathcal{E}_{\Sigma,\mu}(G^\mu_{0^+}(g), f) = \int_{\bbr^d} \langle g(x) ;  f(x) \rangle \mu(dx). 
\end{align}
Now, since $\mu$ has finite second moment and since $\int_{\bbr^d} x \mu(dx) = 0$, set $g(x) = x$, for all $x \in \bbr^d$, and so for $\mu$-a.e. $x \in \bbr^d$, 
\begin{align*}
\tau_\mu(x) = \Sigma \nabla \left(G^\mu_{0^+}(g)\right)(x).
\end{align*}
Finally, taking $f(x) = \Sigma G^{\mu}_{0^+}(g)(x)$, $\mu$-a.e. $x \in \bbr^d$, in \eqref{eq:zero_potential_equation} gives
\begin{align*}
\int_{\bbr^d} \| \tau_\mu(x) \|_{HS}^2 \mu(dx) & = \int_{\bbr^d} \langle \Sigma x ; G^{\mu}_{0^+}(g)(x) \rangle \mu(dx) \\
& = \int_{\bbr^d} \langle \Sigma^{\frac{1}{2}} x ; \Sigma^{\frac{1}{2}} G^{\mu}_{0^+}(g)(x) \rangle \mu(dx).
\end{align*}
Then, by the Cauchy-Schwarz inequality, 
\begin{align*}
\int_{\bbr^d} \| \tau_\mu(x) \|^2_{HS} \mu(dx) \leq U_{\Sigma,\mu} \left\| \Sigma \right\|^2_{HS}. 
\end{align*}
This concludes the proof of the proposition. 
\end{proof}

\begin{rem}\label{rem:closability}
(i) Let us analyze the closability assumption on the bilinear form $\left(\mathcal{E}_{\Sigma, \mu} , \mathcal{C}_c^{\infty}(\bbr^d, \bbr^d)\right)$. If $\mu  = \gamma$, then, by the Gaussian integration by parts,  for all $f,g \in \mathcal{C}_c^{\infty}(\bbr^d, \bbr^d)$,
\begin{align*}
\int_{\bbr^d} \langle  \nabla\left(f\right)(x) ; \nabla\left(g\right)(x) \rangle_{HS} \gamma(dx) = \int_{\bbr^d} \langle f(x) ; \left(-\mathcal{L} \right) (g)(x) \rangle \gamma(dx),
\end{align*}
where,  for all $x \in \bbr^d$ and all $j \in \{1, \dots, d\}$,
\begin{align*}
\mathcal{L}(g_j)(x) = -\langle x ; \nabla(g_j)(x) \rangle + \Delta(g_j)(x).
\end{align*}
Now, let $(f_n)_{n \geq 1}$ be a sequence of functions such that, for all $n \geq 1$, $f_n \in \mathcal{C}^{\infty}_c(\bbr^d, \bbr^d)$, $\| f_n \|_{L^2(\bbr^d,\bbr^d, \gamma)} \rightarrow 0$, as $n$ tends to $+\infty$, and $(\nabla(f_n))_{n \geq 1}$ is a Cauchy sequence in $L^2(\bbr^d, \mathcal{H}, \gamma)$ where $(\mathcal{H}, \langle \cdot ; \cdot \rangle_{\mathcal{H}})=\left(\mathcal{M}_{d \times d}(\bbr), \langle \cdot ; \cdot \rangle_{HS}\right)$.  Since $L^2(\bbr^d ,  \mathcal{H} , \gamma)$ is complete,  there exists $F \in L^2(\bbr^d, \mathcal{H}, \gamma)$ such that $\nabla(f_n) \rightarrow F$, as $n$ tends to $+\infty$.  Moreover, for all $\psi \in \mathcal{C}_c^{\infty}(\bbr^d, \bbr^d)$, 
\begin{align*}
\int_{\bbr^d} \langle F ; \nabla(\psi)(x) \rangle_{HS} \gamma(dx) & =  \underset{n \rightarrow +\infty} {\lim} \int_{\bbr^d} \langle \nabla(f_n)(x) ; \nabla(\psi)(x) \rangle_{HS} \gamma(dx) , \\
& = \underset{n \rightarrow +\infty}{\lim} \int_{\bbr^d} \langle f_n(x) ; \left(-\mathcal{L} \right) (\psi)(x) \rangle \gamma(dx), \\
& = \underset{n \rightarrow +\infty}{ \lim} \langle f_n ; (-\mathcal{L})(\psi) \rangle_{L^2(\bbr^d, \bbr^d, \gamma)} = 0. 
\end{align*}
Since this is true for all $\psi \in \mathcal{C}_c^{\infty}(\bbr^d, \bbr^d)$, $F = 0$ in $L^2(\bbr^d , \mathcal{H} , \gamma)$, and therefore, the form is closable. \\
(ii) Let us assume that $\mu(dx) = \psi(x)dx$ where $\psi$ is the positive Radon-Nikodym derivative of $\mu$ with respect to the Lebesgue measure.~Moreover,  let us assume that $\Sigma = I_d$, the $d \times d$ identity matrix, and that $\psi\in \mathcal{C}^1(\bbr^d)$ with
\begin{align}\label{ineq:cond_log_derivative}
\int_{\bbr^d} \left| \dfrac{\partial_j(\psi)(x)}{\psi(x)} \right|^2 \mu(dx) < +\infty, \quad j \in\{1, \dots, d\}. 
\end{align}
Then, by a standard integration by parts, for all $f,g \in \mathcal{C}^{\infty}_{c}(\bbr^d, \bbr^d)$,
\begin{align*}
\int_{\bbr^d} \langle \nabla(f)(x) ; \nabla(g)(x) \rangle_{HS} \mu(dx)  = \int_{\bbr^d} \langle f(x) ; (-  \mathcal{L}^{\psi})(g)(x) \rangle \mu(dx),
\end{align*}
with,  for all $i \in \{1, \dots, d\}$ and all $x \in \bbr^d$,
\begin{align*}
\mathcal{L}^{\psi}(g_i)(x) = \Delta(g_i)(x) +\left\langle \frac{\nabla(\psi)(x)}{\psi(x)} ; \nabla(g_i)(x)  \right\rangle. 
\end{align*}
Finally, reasoning as in $(i)$, one can prove that the form $\left(\mathcal{E}_\mu, \mathcal{C}_c^{\infty}(\bbr^d, \bbr^d)\right)$ is closable since $\psi \in \mathcal{C}^1(\bbr^d)$ and since the condition \eqref{ineq:cond_log_derivative} holds. \\
(iii) In \cite{CF_96} (see also \cite{AR_JFA}), sharper sufficient conditions are put forward which ensure that the form $\left(\mathcal{E}_\mu, \mathcal{C}_c^{\infty}(\bbr^d, \bbr^d)\right)$ is closable.~Indeed, assume that $\mu(dx) = \psi(x)^2 dx$ with $\psi \in H^1_{loc}(\bbr^d,dx)$, where $H^1_{loc}(\bbr^d,dx)$ is the set of functions in $L^2_{loc}(\bbr^d,dx)$ such that their weak gradient belongs to $L^2_{loc}(\bbr^d,dx)$ (here, $L^2_{loc}(\bbr^d,dx)$ is the space of locally square integrable functions on $\bbr^d$).~Then, reasoning as in $(i)$ and $(ii)$, one can prove that the induced form is closable since, for any $K$ compact subset of $\bbr^d$, 
\begin{align*}
\int_{K} \left\| \nabla(\psi)(x) \right\|^2 dx < +\infty. 
\end{align*}
\end{rem}
\noindent
Let us pursue the discussion with a first stability result. 

\begin{thm}\label{thm:stability_result}
Let $d \geq 1$ and let $\Sigma$ be a nondegenerate $d \times d$ covariance matrix. Let $\gamma_{\Sigma}$ be the centered Gaussian probability measure with covariance matrix $\Sigma$ and let $\mu$ be a probability measure on $\bbr^d$ with finite second moment such that 
\begin{align*}
\int_{\bbr^d} x \mu(dx) = 0 , \quad \int_{\bbr^d} x x^T \mu(dx) = \Sigma. 
\end{align*}
Let the form $\left( \mathcal{E}_{\Sigma, \mu} , \mathcal{C}_c^{\infty}(\bbr^d, \bbr^d) \right)$ be closable.~Finally, let there exists $U_{\Sigma,\mu}>0$ such that for all $f \in \mathcal{C}^{\infty}_c(\bbr^d, \bbr^d)$ with $\int_{\bbr^d} f(x) \mu(dx) = 0$, 
\begin{align*}
\int_{\bbr^d} \|f(x)\|^2 \mu(dx) \leq U_{\Sigma,\mu} \int_{\bbr^d} \langle \Sigma \left(\nabla(f)(x)\right) ; \nabla(f)(x) \rangle_{HS} \mu(dx). 
\end{align*} 
Then, 
\begin{align}\label{ineq:W_1_Gauss_Anisotropic}
W_1(\mu , \gamma_\Sigma) \leq \|\Sigma^{-\frac{1}{2}}\|_{op} \| \Sigma \|_{HS} \sqrt{U_{\Sigma,\mu} - 1}. 
\end{align}
\end{thm}

\begin{proof}
Recall that, 
\begin{align*}
W_1(\mu , \gamma_\Sigma) = \underset{\| h \|_{\operatorname{Lip}} \leq 1}{\sup} \left| \bbe h(X_\mu) - \bbe h(X) \right|, 
\end{align*}
with $X_\mu \sim \mu$ and $X \sim \gamma_\Sigma$. Moreover, by standard approximation arguments (see Lemma \ref{lem:representation_W_1} and Lemma \ref{lem:representation_W_1_2}  of the Appendix),
\begin{align*}
W_1(\mu , \gamma_\Sigma) = \underset{h \in \mathcal{C}_c^{\infty}(\bbr^d), \, \| h \|_{\operatorname{Lip}} \leq 1}{\sup} \left| \bbe h(X_\mu) - \bbe h(X) \right|.
\end{align*} 
Now, let $h \in \mathcal{C}_c^{\infty}(\bbr^d)$ be such that $\| h \|_{\operatorname{Lip}} \leq 1$. Thus, thanks to Lemma \ref{lem:Stein_Factors_anisotropic_Gaussian} and to Stein's method applied to the multivariate Gaussian probability measure $\gamma_\Sigma$, 
\begin{align*}
\left| \bbe h(X_\mu)  - \bbe h(X) \right|  \leq  \|\Sigma^{-\frac{1}{2}}\|_{op} \left(\bbe\left( \|\tau_\mu(X_\mu) - \Sigma \|^2_{HS} \right)\right)^{\frac{1}{2}}.
\end{align*}
Next, from \eqref{eq:Stein_Kernel_multi} and Proposition \ref{prop:existence_stein_kernel},  
\begin{align*}
\bbe \left\| \tau_\mu\left(X_\mu\right) - \Sigma \right\|_{HS}^2 & =  \bbe \| \tau_\mu(X_\mu) \|^2_{HS} + \|\Sigma\|^2_{HS} - 2 \bbe \langle \tau_\mu(X_\mu) ;  \Sigma \rangle_{HS} , \\
& =   \bbe \| \tau_\mu(X_\mu) \|^2_{HS} + \|\Sigma\|^2_{HS} - 2  \bbe \langle  X_\mu ; \Sigma X_\mu \rangle , \\
& =   \bbe \| \tau_\mu(X_\mu) \|^2_{HS} - \|\Sigma\|^2_{HS} , \\
& \leq \|\Sigma\|^2_{HS} \left(U_{\Sigma,\mu}-1\right).
\end{align*}
\end{proof}

\begin{rem}\label{rem:smooth_wasserstein_2_HS}
(i) When $\Sigma = I_d$, the inequality \eqref{ineq:W_1_Gauss_Anisotropic} boils down to
\begin{align}\label{ineq:W_1_Gauss_isotropic}
W_1(\mu , \gamma) \leq  \sqrt{d} \sqrt{U_{I_d,\mu} - 1},
\end{align}
which matches the upper bound obtained in \cite[Theorem $4.1$]{CFP19} for the $2$-Wasserstein distance based on \cite[Proposition $3.1$]{LNP15}.\\
(ii) Note that the previous reasoning ensures as well the following bound (which is relevant in an infinite dimensional setting): for all $\mu$ as in Theorem \ref{thm:stability_result}, 
\begin{align}\label{ineq:smooth_d_2_HS}
\tilde{d}_{W_2}(\mu , \gamma_{\Sigma}) \leq \frac{1}{2} \| \Sigma \|_{HS} \sqrt{U_{\Sigma, \mu} - 1},
\end{align}
with, 
\begin{align*}
\tilde{d}_{W_2}(\mu , \gamma_{\Sigma}) := \underset{h \in \mathcal{C}^{2}(\bbr^d), \|h\|_{\operatorname{Lip}}\leq 1, \,  \tilde{M}_2(h) \leq 1}{\sup} \left| \int_{\bbr^d} h(x) \mu(dx)  - \int_{\bbr^d} h(x) \gamma_{\Sigma}(dx) \right|. 
\end{align*}
\end{rem}
\noindent
In the forthcoming result, a regularization argument shows how to remove the closability assumption.

\begin{thm}\label{thm:stability_result_without_closability}
Let $d \geq 1$ and let $\Sigma$ be a nondegenerate $d \times d$ covariance matrix. Let $\gamma_{\Sigma}$ be the centered Gaussian probability measure with covariance matrix $\Sigma$ and let $\mu$ be a probability measure on $\bbr^d$ with finite second moment such that 
\begin{align*}
\int_{\bbr^d} x \mu(dx) = 0 , \quad \int_{\bbr^d} x x^T \mu(dx) = \Sigma. 
\end{align*}
Finally, let there exists $U_{\Sigma,\mu}>0$ such that for all $f \in \mathcal{C}^{\infty}_c(\bbr^d, \bbr^d)$ with $\int_{\bbr^d} f(x) \mu(dx) = 0$, 
\begin{align*}
\int_{\bbr^d} \|f(x)\|^2 \mu(dx) \leq U_{\Sigma,\mu} \int_{\bbr^d} \langle \Sigma \left(\nabla(f)(x)\right) ; \nabla(f)(x) \rangle_{HS} \mu(dx). 
\end{align*} 
Then, 
\begin{align}\label{ineq:W_1_Gauss_Anisotropic_2}
W_1(\mu , \gamma_\Sigma) \leq \|\Sigma^{-\frac{1}{2}}\|_{op} \| \Sigma \|_{HS} \sqrt{U_{\Sigma,\mu} - 1}. 
\end{align}
\end{thm}

\begin{proof}
Let $d \geq 1$ and let $\Sigma$ be a nondegenerate covariance matrix.~Let $\varepsilon>0$ and let $\gamma_{\varepsilon}$ be the centered Gaussian probability measure on $\bbr^d$ with covariance matrix given by $\varepsilon^2 \Sigma$. 
Let $\mu$ be a centered probability measure on $\bbr^d$ with finite second moment such that 
\begin{align*}
\int_{\bbr^d} x x^T \mu(dx) = \Sigma,
\end{align*}
and satisfying the Poincar\'e-type inequality \eqref{ineq:Poincare_type_inequality} with constant $U_{\Sigma, \mu}$.~Next, let $\mu_{\varepsilon}$ be the probability measure on $\bbr^d$ defined through the following characteristic function: for all $\xi \in \bbr^d$, 
\begin{align*}
\hat{\mu}_{\varepsilon}\left( \xi \right) := \hat{\mu}(\xi) \exp \left( - \dfrac{\varepsilon^2 \langle \xi ; \Sigma(\xi) \rangle}{2}\right), 
\end{align*}
and let $X_\varepsilon \sim \mu_\varepsilon$. Then, 
\begin{align*}
X_\varepsilon =_{\cal L} X_\mu + Z_\varepsilon,
\end{align*}
where $(X_\mu , Z_\varepsilon)$ are independent with $X_\mu \sim \mu$ and $Z_\varepsilon \sim \gamma_\varepsilon$, and where $=_{\cal L}$ stands for equality in distribution.~Next, let $\mathcal{E}_\varepsilon$ be the bilinear symmetric form defined, for all $f, g \in \mathcal{C}_c^{\infty}(\bbr^d, \bbr^d)$, by
\begin{align*}
\mathcal{E}_\varepsilon(f,g) = \int_{\bbr^d} \langle \Sigma_\varepsilon\left(\nabla(f)(x)\right) ; \nabla(g)(x) \rangle_{HS} \mu_{\varepsilon}(dx),
\end{align*}
where $\Sigma_\varepsilon : = (1+ \varepsilon^2)\Sigma$. In particular, the probability measure $\mu_\varepsilon$ is absolutely continuous with respect to the 
Lebesgue measure with density $\psi_\varepsilon$ given, for all $x \in \bbr^d$, by
\begin{align*}
\psi_\varepsilon(x) = \int_{\bbr^d} p_\varepsilon(x-y) \mu(dy),
\end{align*}
where $p_\varepsilon$ is the density of the nondegenerate Gaussian probability measure $\gamma_\varepsilon$. Let us prove that the form 
$\left(\mathcal{E}_\varepsilon, \mathcal{C}_c^{\infty}(\bbr^d, \bbr^d)\right)$ is closable. Let $(f_n)_{n \geq 1}$ be a sequence of functions in $\mathcal{C}^\infty_c(\bbr^d, \bbr^d)$ such that $\|f_n\|_{L^2(\bbr^d, \bbr^d, \mu_\varepsilon)}$
 tends to $0$ as $n$ tends to $+\infty$ and such that $(\nabla(f_n))_{n \geq 1}$ is a Cauchy sequence in $L^2(\bbr^d, \mathcal{H} , \mu_\varepsilon)$ where $\left(\mathcal{H} , \langle \cdot ; \cdot \rangle_{\mathcal{H}}\right)$ 
 is given by $\left(\mathcal{M}_{d \times d}(\bbr) , \langle \Sigma_\varepsilon \cdot ; \cdot \rangle_{HS}\right)$.~Since $\Sigma_\varepsilon$ is nondegenerate, $L^2(\bbr^d, \mathcal{H} , \mu_\varepsilon)$ is complete. Thus, there exists $F \in L^2(\bbr^d, \mathcal{H} , \mu_\varepsilon)$ such that 
\begin{align*}
\nabla(f_n) \underset{n \rightarrow +\infty}{\longrightarrow} F , \quad L^2(\bbr^d, \mathcal{H} , \mu_\varepsilon). 
\end{align*} 
Next, let $\Psi \in \mathcal{C}_c^{\infty}\left(\bbr^d, \mathcal{M}_{d \times d}(\bbr) \right)$.~Then, integrating by parts,
\begin{align*}
\int_{\bbr^d} (1+\varepsilon^2) \langle \Sigma  F(x) ; \Psi(x) \rangle_{HS} \psi_\varepsilon(x) dx & = (1+ \varepsilon^2) \underset{n \longrightarrow +\infty}{\lim} \int_{\bbr^d} \langle \nabla(f_n)(x) ; \Sigma \Psi(x) \rangle_{HS} \psi_\varepsilon(x) dx, \\
& = (1+ \varepsilon^2) \underset{n \longrightarrow +\infty}{\lim}  \sum_{i,j = 1}^d \int_{\bbr^d} \partial_i(f_{n,j})(x) (\Sigma\Psi)_{i,j}(x)\psi_\varepsilon(x) dx , \\
& = (1+ \varepsilon^2) \underset{n \longrightarrow +\infty}{\lim}  \sum_{i,j = 1}^d - \int_{\bbr^d} f_{n,j}(x) \partial_{i} \left((\Sigma\Psi)_{i,j}(x)\psi_\varepsilon(x)\right) dx , \\
& = (1+ \varepsilon^2) \underset{n \longrightarrow +\infty}{\lim}  \sum_{i,j = 1}^d - \int_{\bbr^d} f_{n,j}(x) \bigg( \partial_i \left((\Sigma\Psi)_{i,j}\right)(x) \psi_\varepsilon(x) \\
&\qquad\qquad\qquad\qquad+ (\Sigma\Psi)_{i,j}(x) \partial_i\left(\psi_\varepsilon\right)(x)\bigg) dx. 
\end{align*}
Now, by the Cauchy-Schwarz inequality, for all $n \geq 1$, all $i,j \in \{1 , \dots, d\}$ and all $\varepsilon>0$, 
\begin{align*}
\left| \int_{\bbr^d} f_{n,j}(x) \partial_i \left( (\Sigma\Psi)_{i,j} \right)(x) \psi_\varepsilon(x) dx \right| \leq C_{i,j}(\Psi,\Sigma)\|f_{n,j}\|_{L^2(\bbr^d,\bbr, \mu_\varepsilon)},
\end{align*}
where $C_{i,j}(\Psi, \Sigma)>0$ depends on $i,j$, $\Psi$ and $\Sigma$ only. Moreover, by the Cauchy-Schwarz inequality again, for all $n \geq 1$, all $i,j \in \{1 , \dots, d\}$ and all $\varepsilon>0$, 
\begin{align*}
\left| \int_{\bbr^d} f_{n,j}(x) (\Sigma \Psi)_{i,j}(x) \dfrac{\partial_i\left(\psi_\varepsilon\right)(x)}{\psi_\varepsilon(x)}  \psi_{\varepsilon}(x) dx  \right| \leq \tilde{C}_{i,j}\left(\Psi, \Sigma \right) \| f_{n,j} \|_{L^2(\bbr^d, \bbr,\mu_\varepsilon)} \left\| \dfrac{\partial_i(\psi_\varepsilon)}{\psi_\varepsilon} \right\|_{L^2(K_\Psi, \bbr,\mu_\varepsilon)},
\end{align*}
for some $\tilde{C}_{i,j}(\Psi,\Sigma)>0$ only depending on $i,j$, on $\Psi$ and on $\Sigma$ and for some compact subset $K_\Psi$ of $\bbr^d$ depending only on $\Psi$. In particular, note that, for all $\varepsilon >0$ and all compact subsets $K$ of $\bbr^d$, 
\begin{align*}
\left\| \dfrac{\partial_i(\psi_\varepsilon)}{\psi_\varepsilon} \right\|^2_{L^2(K, \bbr,\mu_\varepsilon)} = \int_{K} \left| \dfrac{\partial_i(\psi_\varepsilon)(x)}{\psi_{\varepsilon}(x)} \right|^2 \psi_\varepsilon(x)dx <+\infty,
\end{align*}
since $\psi_\varepsilon \in \mathcal{C}^1(\bbr^d)$ and $\psi_\varepsilon>0$. Thus, for all $\Psi \in \mathcal{C}^\infty_c(\bbr^d , \mathcal{M}_{d \times d}(\bbr))$, 
\begin{align*}
\int_{\bbr^d} (1+ \varepsilon^2) \langle \Sigma F(x) ; \Psi(x) \rangle_{HS} \psi_\varepsilon(x) dx & = 0,
\end{align*}
which ensures that the form is closable. Moreover, for all $f \in \mathcal{C}_c^{\infty}(\bbr^d,\bbr^d)$ with $\int_{\bbr^d} f(x) \mu_\varepsilon(dx) = 0 $, 
\begin{align*}
\int_{\bbr^d} \|f(x)\|^2 \mu_\varepsilon(dx)  \leq  U_{\Sigma,\mu,\varepsilon} \int_{\bbr^d} \langle \Sigma (\nabla(f))(x) ; \nabla(f)(x) \rangle_{HS} \mu_\varepsilon(dx).
\end{align*}
Finally, note that the nondegenerate Gaussian probability measure $\gamma_{\varepsilon}$ verifies the following Poincar\'e-type inequality: for all $f \in \mathcal{C}_c^{\infty}(\bbr^d, \bbr^d)$ such that $\int_{\bbr^d} f(x) \gamma_{\varepsilon}(dx) = 0$, 
\begin{align*}
\int_{\bbr^d} \|f(x)\|^2 \gamma_{\varepsilon}(dx) \leq \varepsilon^2 \int_{\bbr^d} \langle \Sigma\nabla(f)(x) ; \nabla(f)(x) \rangle_{HS} \gamma_{\varepsilon}(dx),
\end{align*}
so that, with obvious notation, $U_{\Sigma}(\gamma_{\varepsilon}) = \varepsilon^2 U_{\Sigma, \varepsilon}(\gamma_{\varepsilon}) = \varepsilon^2$. So, to conclude, let us find an upper bound for the Poincar\'e constant $U_{\Sigma,\mu,\varepsilon}$ based on the fact that the probability measure $\mu_\varepsilon$ is the convolution of $\mu$ and $\gamma_\varepsilon$, both satisfying a Poincar\'e-type inequality with energy form given, for all $f \in \mathcal{C}_c^{\infty}(\bbr^d, \bbr^d)$, by 
\begin{align*}
\mathcal{E}_{\Sigma}(f,f) = \int_{\bbr^d} \langle \Sigma\nabla(f)(x) ; \nabla(f)(x) \rangle_{HS} \mu(dx),
\end{align*}
and with respective constants $U_{\Sigma,\mu}$ and $\varepsilon^2$.~The proof follows closely the one of \cite[Theorem 2, (vii)]{BU_84}. Let $f \in \mathcal{C}_c^{\infty}(\bbr^d, \bbr^d)$ be such that $\int_{\bbr^d}f(x) \mu_{\varepsilon}(dx) = 0$ and let $x \in \bbr^d$ be fixed.~Then, 
\begin{align}\label{ineq:poincare_mu_translated}
\int_{\bbr^d} \| \tau_x(f)(y) - \int_{\bbr^d} \tau_x(f)(y) \mu(dy) \|^2 \mu(dy) \leq U_{\Sigma, \mu} \int_{\bbr^d} \langle \Sigma \nabla(\tau_x(f))(y) ; \nabla(\tau_x(f))(y) \rangle_{HS} \mu(dy),
\end{align}
where $\tau_x$ is the translation operator defined, for all $f$ smooth enough and all $y \in \bbr^d$, by $\tau_x(f)(y) = f(x+y)$. Developing the square gives, 
\begin{align*}
\int_{\bbr^d} \| \tau_x(f)(y) \|^2 \mu(dy) &\leq U_{\Sigma, \mu} \int_{\bbr^d} \langle \Sigma \nabla(\tau_x(f))(y) ; \nabla(\tau_x(f))(y) \rangle_{HS} \mu(dy) \\
& \qquad\qquad\qquad\qquad +  \left\| \int_{\bbr^d} \tau_x(f)(y) \mu(dy) \right\|^2. 
\end{align*}
Integrating the previous inequality in the $x$ variable with respect to the probability measure $\gamma_{\varepsilon}$ gives
\begin{align*}
\int_{\bbr^d} \| f(z) \|^2 \mu_{\varepsilon}(dz) & \leq U_{\Sigma, \mu} \int_{\bbr^d} \langle \Sigma \nabla(f)(z) ; \nabla(f)(z) \rangle_{HS} \mu_{\varepsilon}(dz) \\
& \qquad\qquad\qquad\qquad +  \int_{\bbr^d}  \left\| \int_{\bbr^d} \tau_x(f)(y) \mu(dy) \right\|^2 \gamma_{\varepsilon}(dx). 
\end{align*}
Now, since $f \in \mathcal{C}_c^{\infty}(\bbr^d, \bbr^d)$, the function $G$ defined, for all $x \in \bbr^d$, by
\begin{align*}
G(x) :=  \int_{\bbr^d} \tau_x(f)(y) \mu(dy), 
\end{align*}
is in $\mathcal{C}^1(\bbr^d)$ with Jacobian matrix given, for all $x \in \bbr^d$, by
\begin{align*}
\nabla(G)(x) = \int_{\bbr^d} \nabla(f)(x+y) \mu(dy).
\end{align*}
Thus, 
\begin{align*}
\int_{\bbr^d}  \left\| G(x) \right\|^2 \gamma_{\varepsilon}(dx) \leq \varepsilon^2 \int_{\bbr^d} \langle \Sigma \nabla(G)(x) ; \nabla(G)(x) \rangle_{HS} \gamma_{\varepsilon}(dx) + \left\| \int_{\bbr^d} G(x) \gamma_{\varepsilon}(dx) \right\|^2.
\end{align*}
But, $\int_{\bbr^d} G(x) \gamma_{\varepsilon}(dx) = 0$.~Finally, by Jensen's inequality, for all $x \in \bbr^d$,
\begin{align*}
 \langle \Sigma \nabla(G)(x) ; \nabla(G)(x) \rangle_{HS} & = \left\| \Sigma^{\frac{1}{2}} \nabla(G)(x) \right \|_{HS}^2 , \\
 &= \left\| \Sigma^{\frac{1}{2}} \int_{\bbr^d} \nabla(f)(x+y) \mu(dy) \right\|_{HS}^2 , \\
 & \leq \int_{\bbr^d} \left\| \Sigma^{\frac{1}{2}} \nabla(f)(x+y) \right\|_{HS}^2 \mu(dy).
\end{align*}
Then, 
\begin{align*}
\int_{\bbr^d}  \left\| G(x) \right\|^2 \gamma_{\varepsilon}(dx) \leq \varepsilon^2 \int_{\bbr^d \times \bbr^d} \left\| \Sigma^{\frac{1}{2}} \nabla(f)(x+y) \right\|_{HS}^2 \mu(dy)\gamma_{\varepsilon}(dx),
\end{align*}
which implies that $U_{\Sigma, \mu, \varepsilon} \leq U_{\Sigma, \mu} + \varepsilon^2$. So from Theorem \ref{thm:stability_result}, 
\begin{align}\label{ineq:Wasserstein_1_varepsilon}
W_1 \left(\mu_{\varepsilon} , \tilde{\gamma}_{\varepsilon}\right) \leq \|\Sigma_{\varepsilon}^{-\frac{1}{2}}\|_{op} \| \Sigma_{\varepsilon} \|_{HS} \sqrt{U_{\Sigma,\mu, \varepsilon} - 1},
\end{align}
where $\tilde{\gamma}_{\varepsilon}$ is a nondegenerate centered Gaussian probability measure with covariance matrix given by $\Sigma_{\varepsilon} = (1+ \varepsilon^2) \Sigma$. Now, by L\'evy's continuity theorem, it is clear that $\mu_{\varepsilon}$ and $\tilde{\gamma}_{\varepsilon}$ converge weakly respectively to $\mu$ and to $\gamma_{\Sigma}$, as $\varepsilon \rightarrow 0^+$. Thus, let $h \in \mathcal{C}^{\infty}_c(\bbr^d)$ be such that $\| h \|_{\operatorname{Lip}} \leq 1$.  From the proof of Theorem \ref{thm:stability_result}, 
\begin{align*}
\left| \int_{\bbr^d} h(x)\mu_{\varepsilon}(dx) -  \int_{\bbr^d} h(x)\tilde{\gamma}_{\varepsilon}(dx)  \right| \leq \|\Sigma_{\varepsilon}^{-\frac{1}{2}}\|_{op} \| \Sigma_{\varepsilon} \|_{HS} \sqrt{U_{\Sigma, \mu} + \varepsilon^2 - 1}. 
\end{align*}
Letting $\varepsilon \rightarrow 0^+$ and then taking the supremum over all $h \in \mathcal{C}_c^{\infty}(\bbr^d)$ with $\| h \|_{\operatorname{Lip}} \leq 1$ leads to 
\begin{align*}
W_1 \left(\mu , \gamma_{\Sigma} \right) \leq \|\Sigma^{-\frac{1}{2}}\|_{op} \| \Sigma \|_{HS} \sqrt{U_{\Sigma,\mu} - 1},
\end{align*}
which concludes the proof of the theorem.
\end{proof}
\noindent
As an application of the previous techniques, let us provide a rate of convergence in $1$-Wasserstein distance in the multivariate central limit theorem where the limiting centered Gaussian probability measure on $\bbr^d$ has covariance matrix given by $\Sigma$.  The argument is based on a specific representation for the Stein's kernel of the standardized sum given, for all $n \geq 1$,  by
\begin{align}\label{eq:normalized_sum}
S_n  = \frac{1}{\sqrt{n}} \sum_{k = 1}^n X_k ,  
\end{align}
where $(X_k)_{k \geq 1}$ is a sequence of independent and identically distributed (iid) centered random vectors of $\bbr^d$ with finite second moment such that 
\begin{align*}
\bbe X_1 X_1^T  = \Sigma,
\end{align*}
and whose law satisfies the Poincar\'e-type inequality \eqref{ineq:Poincare_type_inequality}. 

\begin{thm}\label{thm:Rate_TCL_W1}
Let $d \geq 1$ and let $\Sigma$ be a nondegenerate $d \times d$ covariance matrix.~Let $\gamma_\Sigma$ be the nondegenerate centered Gaussian probability measure on $\bbr^d$ with covariance matrix given by $\Sigma$.  Let $\mu$ be a centered probability measure on $\bbr^d$ with finite second moments such that 
\begin{align*}
\int_{\bbr^d} x x^T \mu(dx) = \Sigma,
\end{align*}
and which satisfies the Poincar\'e-type inequality \eqref{ineq:Poincare_type_inequality} for some $U_{\Sigma, \mu} >0$.~Let $(X_k)_{k \geq 1}$ be a sequence of independent and identically distributed random vectors of $\bbr^d$ with law $\mu$ and let $(S_n)_{n \geq 1}$ be the sequence of normalized sums defined by \eqref{eq:normalized_sum} and with respective laws $(\mu_n)_{n \geq 1}$.~Then, for all $n \geq 1$, 
\begin{align}\label{ineq:rate_W1_CLT}
W_1(\mu_n, \gamma_\Sigma) \leq \dfrac{\|\Sigma^{-\frac{1}{2}}\|_{op} \| \Sigma \|_{HS}}{\sqrt{n}} \sqrt{U_{\Sigma,\mu} - 1}.
\end{align}
\end{thm}

\begin{proof}
First, let us assume that $\mu$ is such that the form $ \left( \mathcal{E}_{\mu}, \mathcal{C}_c^{\infty}(\bbr^d,  \bbr^d)\right)$ is closable. Then, by Proposition \ref{prop:existence_stein_kernel}, there exists $\tau_\mu$ such that, for all $f \in \mathcal{D}(\mathcal{E}_{\Sigma, \mu})$, 
\begin{align*}
\int_{\bbr^d} \langle x ; f(x) \rangle \mu(dx) = \int_{\bbr^d} \langle \nabla(f)(x) ; \tau_\mu(x) \rangle_{HS} \mu(dx).
\end{align*}
Next, let $\tau_n$ be defined, for all $n \geq 1$, by
\begin{align*}
\tau_n(x) = \bbe \left[  \frac{1}{n} \sum_{k = 1}^n \tau_{\mu}(X_k)   |  S_n = x  \right]. 
\end{align*}
Observe that, for all $n \geq 1$ and all $f$ smooth enough, 
\begin{align*}
\int_{\bbr^d} \langle \tau_n(x) ; \nabla(f)(x) \rangle_{HS} \mu_n(dx) & = \frac{1}{n} \sum_{k = 1}^n \bbe \langle \tau_{\mu}(X_k) ; \nabla(f)(S_n) \rangle_{HS} , \\
& = \bbe \langle S_n ; f(S_n) \rangle ,
\end{align*}
as the sequence $(X_k)_{k \geq 1}$ is a sequence of iid random vectors of $\bbr^d$ and that $\tau_\mu$ is a Stein's kernel for the law of $X_1$.  Next, let $h \in \mathcal{C}_c^{\infty}(\bbr^d)$ be such that $\| h \|_{\operatorname{Lip}} \leq 1$. Then, from the proof of Theorem \ref{thm:stability_result}, for all $n \geq 1$
\begin{align*}
\left| \bbe h(S_n) - \bbe h(X) \right| \leq \|\Sigma^{-\frac{1}{2}}\|_{op} \left( \bbe\left( \| \tau_n(S_n) - \Sigma  \|^2_{HS} \right)\right)^{\frac{1}{2}},
\end{align*}
where $X \sim \gamma_{\Sigma}$.  So, let us estimate the Stein discrepancy, i.e., 
the last term on the right-hand side of the above inequality.  By Jensen's inequality,  independence,  since $\bbe \tau_{\mu}(X_1) = \Sigma$ and from the proof of Theorem \ref{thm:stability_result}, 
\begin{align*}
\bbe \| \tau_n(S_n) - \Sigma  \|^2_{HS} & \leq \frac{1}{n^2} \sum_{k = 1}^{n} \bbe \|\tau_\mu(X_1) - \Sigma\|^2_{HS} , \\
& \leq \frac{1}{n}\bbe \|\tau_\mu(X_1) - \Sigma\|^2_{HS} \\
& \leq \frac{\|\Sigma\|^2_{HS}}{n} \left( U_{\Sigma, \mu}-1\right).
\end{align*}
Thus,  for all $n \geq 1$, 
\begin{align*}
\left| \bbe h(S_n) - \bbe h(X) \right| \leq \frac{\|\Sigma^{-\frac{1}{2}}\|_{op}\|\Sigma\|_{HS}}{\sqrt{n}} \sqrt{U_{\Sigma, \mu}-1},
\end{align*}
and so the bound \eqref{ineq:rate_W1_CLT} is proved when the form $\left(\mathcal{E}_{\mu}, \mathcal{C}_c^{\infty}(\bbr^d,\bbr^d)\right)$ is closable.~A regularization argument as in the proof of Theorem \ref{thm:stability_result_without_closability} allows to get the bound  \eqref{ineq:rate_W1_CLT} for the general case, concluding the proof of the theorem. 
\end{proof}
\noindent 

\begin{rem}\label{rem:High_dimensional_CLT}
(i) In Theorem \ref{thm:Rate_TCL_W1}, one could assume $(X_k)_{k\geq 1}$ to be a sequence of independent random vectors of $\bbr^d$ with laws $(\mu_k)_{k \geq 1}$ such that, for all $k \geq 1$,
\begin{align*}
\int_{\bbr^d} x \mu_k(dx) = 0 , \quad \int_{\bbr^d} xx^T \mu_k(dx) = \Sigma,
\end{align*}
and with Poincar\'e constants $(U_{\Sigma , \mu_k})_{k \geq 1}$. Then, by a completely similar argument, for all $n \geq 1$
\begin{align}\label{ineq:nonid}
W_1(\mu_n , \gamma_{\Sigma}) \leq  \dfrac{\|\Sigma^{-\frac{1}{2}}\|_{op} \| \Sigma \|_{HS}}{n} \left(\sum_{k = 1}^n (U_{\Sigma, \mu_k} - 1)\right)^{\frac{1}{2}},
\end{align}
with $S_n \sim \mu_n$.\\
(ii) When $\Sigma = I_d$, the bound \eqref{ineq:rate_W1_CLT} boils down to
\begin{align*}
W_1(\mu_n, \gamma) \leq  \sqrt{\frac{d}{n}}\sqrt{U_{d,\mu} - 1},
\end{align*}
which matches exactly the bound obtained in \cite[Theorem $4.1$]{CFP19} for the 2-Wasserstein distance (recall that by H\"older's inequality $W_1(\mu , \nu) \leq W_2(\mu , \nu)$, with $\mu, \nu$ two probability measures on $\bbr^d$ with finite second moment).~A large amount of work has been dedicated to rates of convergence in transportation distances for high dimensional central limit theorem.~Let us briefly recall some of these most recent results. In \cite{Zhai_2018}, under the condition that $\| X_1 \| \leq \beta$ a.s., for some $\beta>0$, the bound,
\begin{align}\label{ineq:Zhai_W2}
W_2(\mu_n, \gamma_\Sigma) \leq \dfrac{5 \sqrt{d} \beta \left(1+\log n\right)}{\sqrt{n}}, 
\end{align}
is proved, where $\mu_n$ is the law of the normalized sum $S_n$ defined by \eqref{eq:normalized_sum}.~In particular, note that \cite[Theorem $1.1$]{Zhai_2018} is established for all covariance matrices $\Sigma$ and not only for $\Sigma = I_d$.~Thus, our bound gets rid of the term $\beta \sqrt{d}$ in the general nondegenerate case under a finite Poincar\'e constant assumption which is not directly comparable to the condition $\| X_1 \| \leq \beta$ a.s. (see the discussion after \cite[Theorem $4.1$]{CFP19} and \cite{BGMZ_18}). Note that, in the isotropic case,  the bound \eqref{ineq:Zhai_W2} scales linearly with the dimension since $\beta \geq \sqrt{d}$.  An improvement of the bound \eqref{ineq:Zhai_W2} has been obtained in \cite[Theorem $1$]{EMZ_2020} with $\log n$ replaced by $\sqrt{\log n}$. Similarly, in \cite[Theorem $B.1$]{FSX_19}, using Stein's method and Bismut formula, the following bound is obtained at the level of the $1$-Wasserstein distance, 
\begin{align*}
W_1(\mu_n, \gamma)\leq \dfrac{Cd\beta \left(1+ \log n\right)}{\sqrt{n}},
\end{align*}
for some $C>0$ and under the assumption that $\|X_i\| \leq \beta$ a.s., for all $i \geq 1$.  The anisotropic case is not covered by this last result but the non-identically distributed case is. In \cite[Theorem $1$]{Bonis_PTRF_20}, under $\bbe \|X_1\|^4 <+\infty$, the following holds true:   
\begin{align*}
W_2(\mu_n, \gamma) \leq \dfrac{C d^{\frac{1}{4}} \| \bbe X_1X_1^T \|X_1\|^2 \|^{\frac{1}{2}}_{HS}}{\sqrt{n}}, 
\end{align*}
for some $C \in (0,14)$.~The previous bound scales at least linearly with the dimension as well. Finally, let us mention \cite{Fathi_AOP} where sharp rates of convergence in $p$-Wasserstein distances, for $p \geq 2$, are obtained under various (strong)-convexity assumptions. \\
(iii) Thanks to an inequality of Talagrand (see \cite{Tal_96}), quantitative rates of convergence, in relative entropy, towards the Gaussian probability measure $\gamma$ on $\bbr^d$ imply quantitative rates of convergence in $2$-Wasserstein distance. In \cite[Theorem $1$]{ABBN_04} and \cite[Theorem $1.3$]{JB_04},  under a spectral gap assumption, a rate of convergence of order $1/n$ is obtained in relative entropy in dimension $1$, while \cite[Theorem 1.1]{BHN_12} provides a quantitative entropy jump result under log-concavity in any dimension. Finally, \cite{CFP19} and \cite{EMZ_2020} contain quantitative high dimensional entropic CLT under various assumptions. \\
(iv) There is a vast literature on quantitative multivariate central limit theorems for different probability metrics.~For example, in \cite{Bentkus05}, a rate of convergence is obtained for the convex distance. Namely, for all $n \geq 1$,
\begin{align*}
\underset{A \in \mathcal{C}}{\sup} \left| \bbp \left(S_n \in A\right) - \bbp \left(Z \in A\right) \right| \leq \dfrac{c d^{\frac{1}{4}}}{\sqrt{n}} \bbe \| \Sigma^{-\frac{1}{2}} X_1\|^3, 
\end{align*}
where $\mathcal{C}$ is the set of all measurable convex subsets of $\bbr^d$, $Z$ is a centered Gaussian random vector of $\bbr^d$ with nondegenerate covariance matrix $\Sigma$, $c$ is a positive constant which can be made explicit (see \cite{Rai19}) and $(X_i)_{i \geq 1}$ is a sequence of iid random vectors of $\bbr^d$ such that $\bbe X_1 =0$, $\bbe X_1X_1^T = \Sigma$ and $\bbe \|X_1\|^3 < +\infty$.~In \cite{FK_21}, quantitative high dimensional CLTs are investigated by means of Stein's method but where the set $\mathcal{C}$ is replaced by the set of hyperrectangles of $\bbr^d$. In particular, \cite[Theorem $1.1$]{FK_21} provides an error bound using Stein's kernels which holds in the nondegenerate anisotropic case.~Note that \cite[Corollary $1.1$]{FK_21} uses the results contained in \cite{Fathi_AOP} in order to build Stein's kernels when the sampling distribution is a centered probability measure on $\bbr^d$ with a log-concave density and a nondegenerate covariance matrix $\Sigma$ with diagonal entries equal to $1$.~Finally, in \cite{Das_Lahiri_21}, the optimal growth rate of the dimension with the sample size for the probability metric over all hyperrectangles is completely identified under general moment conditions. 
\end{rem}

\begin{rem}\label{rem:example_measure_mu}
(i) Let $d \geq 1$,  let $\Sigma$ be a $d \times d$ nondegenerate covariance matrix and let $V$ be a non-negative function defined on $\bbr^d$ which is twice continuously differentiable everywhere on $\bbr^d$.~Let
\begin{align*}
C_V \int_{\bbr^d} e^{-V(x)} dx = 1 , 
\end{align*}
for some constant $C_V>0$ which depends on $V$ and on $d$. Let $\mu_V$ denote the induced probability measure on $\bbr^d$ and let us assume that
\begin{align*}
\int_{\bbr^d} x  \mu_V(dx) = 0 , \quad \int_{\bbr^d} xx^T  \mu_V(dx) = \Sigma.  
\end{align*}
Assume further that, for all $x \in \bbr^d$, 
\begin{align}\label{ineq:ellipticity}
\operatorname{Hess}(V)(x) \geq \kappa \Sigma^{-1},
\end{align}
for some $\kappa \in (0,1]$ (where the order is in the sense of positive semi-definite matrices).~Since $\Sigma$ is nondegenerate, \eqref{ineq:ellipticity} ensures that the probability measure $\mu_V$ is strongly log-concave.~Thus, by the Brascamp and Lieb inequality (see, e.g., \cite{BL_76,CCL13}),  for all $f \in \mathcal{C}_c^{\infty}(\bbr^d)$ with $\int_{\bbr^d} f(x) \mu_V(dx) = 0$, 
\begin{align*}
\int_{\bbr^d} \left| f(x)\right|^2 \mu_V(dx) \leq \int_{\bbr^d} \langle \nabla(f)(x) ; \operatorname{Hess}(V)(x)^{-1} \nabla(f)(x) \rangle \mu_{V}(dx). 
\end{align*}
The previous inequality readily implies, for all $f \in \mathcal{C}_c^{\infty}(\bbr^d, \bbr^d)$ with $\int_{\bbr^d} f(x) \mu_V(dx) = 0$, that 
\begin{align*}
\int_{\bbr^d} \|f(x)\|^2 \mu_V(dx) \leq \int_{\bbr^d} \langle \nabla(f)(x) ;  \operatorname{Hess}(V)(x)^{-1} \nabla(f)(x)  \rangle_{HS} \mu_V(dx),
\end{align*}
which gives, thanks to \eqref{ineq:ellipticity},
\begin{align*}
\int_{\bbr^d} \|f(x)\|^2 \mu_V(dx) \leq \frac{1}{\kappa} \int_{\bbr^d} \langle \nabla(f)(x) ;  \Sigma \nabla(f)(x)  \rangle_{HS} \mu_V(dx). 
\end{align*}
In particular, if $\kappa = 1$, then $\mu_V = \gamma_{\Sigma}$. Moreover, Theorem \ref{thm:Rate_TCL_W1} provides the following bound for $X_1 \sim \mu_V$: for all $n \geq 1$, 
\begin{align*}
W_1(\mu_n, \gamma_\Sigma) \leq \dfrac{\| \Sigma^{- \frac{1}{2}} \|_{op} \| \Sigma \|_{HS}}{\sqrt{n}} \sqrt{\frac{1}{\kappa}-1}. 
\end{align*}
(ii) Since \cite{Bob_AOP_99}, it is well-known that log-concave probability measures (i.e.,~probability measures with log-concave densities with respect to the Lebesgue measure) on $\bbr^d$ satisfy a Poincar\'e-type inequality. Let $d \geq 1$, let $\Sigma = I_d$ and let $\mu$ be a log-concave probability measure on $\bbr^d$ in isotropic position, i.e., such that 
\begin{align*}
\int_{\bbr^d} x\mu(dx) = 0 ,\quad \int_{\bbr^d} xx^T \mu(dx) = I_d.
\end{align*}
Then, for all $f \in \mathcal{C}_c^{\infty}(\bbr^d)$ such that $\int_{\bbr^d} f(x) \mu(dx) = 0$, 
\begin{align*}
\int_{\bbr^d} |f(x)|^2 \mu(dx) \leq C_p(\mu) \int_{\bbr^d} \langle \nabla(f)(x) ; \nabla(f)(x) \rangle \mu(dx),
\end{align*}
where $C_p(\mu)>0$ is the best constant for which the previous inequality holds.~According to the well known   
Kannan-Lov\'asz-Simonovits (KLS) conjecture (see, e.g., \cite{AB_2015}), the constant $C_p(\mu)$ should be uniformly upper bounded by 
some universal constant $C\geq 1$ (independent of the dimension) for all log-concave probability measures on $\bbr^d$ in isotropic position.~In Theorem \ref{thm:Rate_TCL_W1}, this conjecture would imply: for all $n \geq 1$,
\begin{align*}
W_1(\mu_n, \gamma) \leq \sqrt{\dfrac{d}{n}} \sqrt{C-1},
\end{align*}
if $X_1 \sim \mu$. To date, the best known bound on the constant $C_P(\mu)$ is provided by the very recent result in \cite{Chen_2021} which ensures a lower bound on the isoperimetric constant of an isotropic log-concave probability measure $\mu$ on $\bbr^d$:
\begin{align*}
I_P(\mu) \geq d^{-c' \left(\frac{\log \log d}{\log d}\right)^{\frac{1}{2}}},
\end{align*}
for some $c'>0$. According to Cheeger's inequality (see, e.g., \cite{Bob_AOP_99}),
\begin{align*}
C_P(\mu) \leq c_1 d^{c_2 \left(\frac{\log \log d}{\log d}\right)^{\frac{1}{2}}},
\end{align*}
for some positive numerical constants $c_1$ and $c_2$. 

In analogy with the KLS conjecture, and with regard to the general anisotropic case with a nondegenerate covariance matrix $\Sigma$, 
it seems natural to wonder if the functional $U_{\Sigma, \mu}$ is uniformly bounded over the class of centered log-concave probability measures on $\bbr^d$ with covariance structure $\Sigma$
and if the corresponding upper bound is dimension free.
\end{rem}

\section{Appendix}

\begin{lem}\label{lem:representation_W_1}
Let $d \geq 1$ and let $\mu, \nu$ be two probability measures on $\bbr^d$ with finite first moment.  Then, 
\begin{align}\label{eq:rep_W_1}
W_1(\mu, \nu) = \underset{h \in \mathcal{C}^\infty(\bbr^d), \, \|h\|_{\operatorname{Lip} \leq 1}}{\sup} \left| \int_{\bbr^d} h(x) \mu(dx) - \int_{\bbr^d} h(x) \nu(dx) \right|. 
\end{align}
\end{lem}

\begin{proof}
Recall that, by duality, 
\begin{align}\label{eq:KT_Lipschitz}
	W_1(\mu ,\nu) = \underset{h \in \operatorname{Lip}, \|h\|_{\operatorname{Lip}}\leq 1}{\sup} \left| \int_{\bbr^d} h(x) \mu(dx) - \int_{\bbr^d} h(x) \nu(dx) \right|,
\end{align}
where $\operatorname{Lip}$ is the space of Lipschitz functions on $\bbr^d$ with the Lipschitz semi-norm 
\begin{align*}
	\|h\|_{\operatorname{Lip}}  = \underset{x,y \in \bbr^d, x \ne y}{ \sup } \dfrac{|h(x) -h(y)|}{\|x - y\|}.  
\end{align*}	
So, at first, it is clear that 
\begin{align*}
W_1(\mu , \nu) \geq \underset{ h \in \mathcal{C}^{\infty}(\bbr^d), \, \| h \|_{\operatorname{Lip}}\leq 1}{\sup} \left| \int_{\bbr^d} h(x) \mu(dx) - \int_{\bbr^d} h(x) \nu(dx) \right|. 
\end{align*}
Next, let $h \in \mathcal{C}^1(\bbr^d)$ be such that $\| h \|_{\operatorname{Lip}}\leq 1$. Let $\varepsilon >0$ and let $p_{\varepsilon}$ be the centered multivariate Gaussian density with covariance matrix $\varepsilon I_d$, i.e., for all $y \in \bbr^d$, 
\begin{align*}
p_{\varepsilon}(y) = \frac{1}{(2 \pi \varepsilon)^{\frac{d}{2}}} \exp\left(  - \dfrac{\|y\|^2}{2 \varepsilon}\right).
\end{align*}
Moreover, let 
\begin{align*}
h_{\varepsilon}(x)  = \int_{\bbr^d} h(x-y) p_{\varepsilon}(y) dy, \quad x \in \bbr^d.
\end{align*}
It is clear that $h_\varepsilon\in \mathcal{C}^{\infty}(\bbr^d)$, and that
\begin{align*}
\| h_\varepsilon \|_{\operatorname{Lip}} \leq 1, \quad \varepsilon>0,
\end{align*}
since $\| h \|_{\operatorname{Lip}} \leq 1$.  Moreover, for all $\varepsilon>0$ and all $x \in \bbr^d$, 
\begin{align*}
\left| h(x) - h_\varepsilon(x) \right| & \leq \| h \|_{\operatorname{Lip}} \int_{\bbr^d} \| z \| p_{\varepsilon}(z)dz , \\
& \leq C_d\sqrt{\varepsilon},
\end{align*}
for some constant $C_d>0$ depending only on $d \geq 1$.  Thus, 
\begin{align*}
W_1(\mu,\nu) \leq \underset{h \in \mathcal{C}^{\infty}(\bbr^d), \, \|h\|_{\operatorname{Lip}}\leq 1}{\sup} \left| \int_{\bbr^d} h(x) \mu(dx) - \int_{\bbr^d} h(x) \nu(dx) \right| + 2 C_d \sqrt{\varepsilon}. 
\end{align*}
Letting $\varepsilon \rightarrow 0^+$ concludes the proof of this reduction principle.  
\end{proof}

\begin{lem}\label{lem:representation_W_1_2}
Let $d \geq 1$ and let $\mu, \nu$ be two probability measures on $\bbr^d$ with finite first moment.  Then, 
\begin{align}\label{eq:rep_W_1_2}
W_1(\mu, \nu) = \underset{h \in \mathcal{C}_c^\infty(\bbr^d), \, \|h\|_{\operatorname{Lip} \leq 1}}{\sup} \left| \int_{\bbr^d} h(x) \mu(dx) - \int_{\bbr^d} h(x) \nu(dx) \right|. 
\end{align}
\end{lem}

\begin{proof}
Let $h$ be a Lipschitz function on $\bbr^d$ such that 
\begin{align*}
\|h\|_{\operatorname{Lip}}  = \underset{x,y \in \bbr^d, x \ne y}{ \sup } \dfrac{|h(x) -h(y)|}{\|x - y\|}\leq 1.  
\end{align*}
Recall that by Rademacher's theorem such a function $h$ is differentiable almost everywhere on $\bbr^d$.  Actually, as shown next, it is possible to restrict the supremum appearing in \eqref{eq:KT_Lipschitz} to bounded Lipschitz functions $h$ defined on $\bbr^d$ and such that $\|h\|_{\operatorname{Lip}} \leq 1$. Indeed, let $R>0$ and let $G_R$ be the function defined, for all $y \in \bbr$, by
\begin{align*}
G_R(y) = (-R) \vee \left(y \wedge R\right).
\end{align*}
$G_R$ is clearly bounded on $\bbr$ by $R$ and, for all $y \in \bbr$ fixed, 
\begin{align*}
\underset{R \rightarrow + \infty}{\lim} G_R(y) = y.
\end{align*}
Moreover, $\| G_R \|_{\operatorname{Lip}} \leq 1$ by construction. Thus, for all $R>0$, let $h_R$ be defined, for all $x \in \bbr^d$, by $h_R(x) = G_R(h(x))$.  The function $h_R$ is bounded on $\bbr^d$ and $1$-Lipschitz by composition, and moreover, $\underset{R \rightarrow + \infty}{\lim} G_R(h(x)) = h(x)$, for all $x \in \bbr^d$. Then, 
\begin{align*}
\left| \int_{\bbr^d} h(x) \mu(dx) - \int_{\bbr^d} h(x) \nu(dx) \right| & \leq \left|\int_{\bbr^d} h_R(x) \mu(dx) - \int_{\bbr^d} h_R(x) \nu(dx) \right| \\
& \quad\quad + \left| \int_{\bbr^d} (h_R(x) - h(x)) \mu(dx) \right| + \left| \int_{\bbr^d} (h_R(x)  - h(x)) \nu(dx) \right| .  
\end{align*}
Thus, 
\begin{align*}
\left| \int_{\bbr^d} h(x) \mu(dx) - \int_{\bbr^d} h(x) \nu(dx) \right| & \leq \underset{h \in \operatorname{Lip}_b,\, \|h\|_{\operatorname{Lip}}\leq 1}{\sup} \left|\int_{\bbr^d} h(x) \mu(dx) - \int_{\bbr^d} h(x) \nu(dx) \right| \\
& \quad\quad + \left| \int_{\bbr^d} (h_R(x) - h(x)) \mu(dx) \right| + \left| \int_{\bbr^d} (h_R(x)  - h(x)) \nu(dx) \right|, 
\end{align*}
where $\operatorname{Lip}_b$ is the set of bounded Lipschitz functions on $\bbr^d$.  Next, without loss of generality, let us assume that $h(0) = 0$. Now, since, for all $x \in \bbr^d$ and all $R>0$, 
\begin{align*}  
\left| h_R(x) - h(x) \right| \leq 2 |h(x)| \leq 2 \|x\|,
\end{align*}
the dominated convergence theorem ensures that
\begin{align*}
\underset{R \rightarrow +\infty}{\lim} \int_{\bbr^d} (h_R(x) - h(x)) \mu(dx) = 0,
\end{align*}
and similarly for $\nu$. Thus, 
\begin{align*}
\left| \int_{\bbr^d} h(x) \mu(dx) - \int_{\bbr^d} h(x) \nu(dx) \right| & \leq \underset{h \in \operatorname{Lip}_b,\, \|h\|_{\operatorname{Lip}}\leq 1}{\sup} \left|\int_{\bbr^d} h(x) \mu(dx) - \int_{\bbr^d} h(x) \nu(dx) \right| .
\end{align*}
Next, applying the regularization procedure of Lemma \ref{lem:representation_W_1}, one has 
\begin{align*}
W_1(\mu, \nu) = \underset{h \in \mathcal{C}_b^\infty(\bbr^d), \, \|h\|_{\operatorname{Lip} \leq 1}}{\sup} \left| \int_{\bbr^d} h(x) \mu(dx) - \int_{\bbr^d} h(x) \nu(dx) \right|,
\end{align*}
where $\mathcal{C}_b^{\infty}(\bbr^d)$ is the set of infinitely differentiable  bounded functions on $\bbr^d$.  Finally, let $h \in \mathcal{C}_b^{\infty}(\bbr^d)$ be such that $\| h \|_{\operatorname{Lip}}\leq 1$ and let $\psi$ be a smooth compactly supported function with values in $[0,1]$ and with support included in the Euclidean ball centered at the origin and of  radius $2$ such that, for all $x \in \bbr^d$ with $\|x\| \leq 1$,  $\psi(x) = 1$. Then, let 
\begin{align*}
\tilde{h}_{R}(x) = \psi\left( \dfrac{x}{R} \right) h(x), \quad R \geq 1, \quad x \in \bbr^d. 
\end{align*}
Clearly, $\tilde{h}_R\in \mathcal{C}_c^{\infty}(\bbr^d)$, and moreover, for all $x \in \bbr^d$ and all $R \geq 1$, 
\begin{align*}
\nabla(\tilde{h}_R)(x) = \frac{1}{R} \nabla(\psi)(\frac{x}{R})h(x) + \psi\left( \frac{x}{R}\right) \nabla(h)(x). 
\end{align*} 
Then, for all $x \in \bbr^d$, 
\begin{align*}
\| \nabla(\tilde{h}_R)(x) \| \leq 1 + \frac{1}{R} \|h\|_{\infty} \|\nabla(\psi)\|_{\infty}. 
\end{align*}
Thus, 
\begin{align*}
\left| \int_{\bbr^d} h(x) \mu(dx) - \int_{\bbr^d} h(x) \nu(dx)  \right| & \leq \left| \int_{\bbr^d} \tilde{h}_R(x) \mu(dx) - \int_{\bbr^d} \tilde{h}_R(x) \nu(dx)  \right| + \left| \int_{\bbr^d} (h(x) - \tilde{h}_R(x)) \mu(dx) \right| \\
& \quad\quad + \left| \int_{\bbr^d} (h(x) - \tilde{h}_R(x)) \nu(dx) \right|, \\
& \leq \left(1 + \frac{1}{R} \|h\|_{\infty} \|\nabla(\psi)\|_{\infty}\right) \underset{ h \in \mathcal{C}^{\infty}_c(\bbr^d), \, \| h \|_{\operatorname{Lip}}\leq 1}{\sup} \left| \int_{\bbr^d} h(x) \mu(dx) - \int_{\bbr^d} h(x) \nu(dx) \right| \\
& \quad\quad + \left| \int_{\bbr^d} h(x) \left(1 - \psi\left(\frac{x}{R}\right)\right) \mu(dx) \right| + \left| \int_{\bbr^d} h(x) \left(1 - \psi\left(\frac{x}{R}\right)\right) \nu(dx) \right|.
\end{align*}
But, 
\begin{align*}
\left| \int_{\bbr^d} h(x)\left(1 - \psi\left(\frac{x}{R}\right)\right) \mu(dx) \right| \leq \|h\|_{\infty} \int_{\|x\| \geq R}  \mu(dx), 
\end{align*}
and similarly for $\nu$.  Thus,  letting $R \rightarrow +\infty$, 
\begin{align*}
\left| \int_{\bbr^d} h(x) \mu(dx) - \int_{\bbr^d} h(x) \nu(dx)  \right| & \leq \underset{ h \in \mathcal{C}^{\infty}_c(\bbr^d), \, \| h \|_{\operatorname{Lip}}\leq 1}{\sup} \left| \int_{\bbr^d} h(x) \mu(dx) - \int_{\bbr^d} h(x) \nu(dx) \right|,
\end{align*}
which concludes the proof of the lemma.
\end{proof}

\end{document}